\documentclass[11pt,a4paper,reqno]{amsart} 
\title{Exchange graphs and Ext quivers}
\author{Alastair King}
\address{Mathematical Sciences, University of Bath,
  Bath BA2 7AY, U.K.}
\email{a.d.king@bath.ac.uk}
\author{Yu Qiu}
\address{Institutt for matematiske fag, NTNU,
 N-7491 Trondheim, Norway.}
\email{yu.qiu@bath.edu}
\usepackage{geometry} 
\usepackage{amsfonts}
\usepackage{amsmath,amssymb,amsthm,amsxtra}
\usepackage{float}
\usepackage[colorlinks, linkcolor=RoyalBlue, anchorcolor=Periwinkle,
  citecolor=Orange, urlcolor=Emerald]{hyperref}
\usepackage[usenames,dvipsnames]{xcolor}
\usepackage{enumitem} 
\usepackage{tikz}
  \usetikzlibrary{matrix}
  \usetikzlibrary{arrows}
\usepackage{url,bookmark}
\usepackage[all]{xy}

\newcommand{\New}[1]{{#1}}

\theoremstyle{plain}
\newtheorem{theorem}{Theorem}[section]
\newtheorem{lemma}[theorem]{Lemma}
\newtheorem{corollary}[theorem]{Corollary}
\newtheorem{proposition}[theorem]{Proposition}

\theoremstyle{definition}
\newtheorem{definition}[theorem]{Definition}
\newtheorem{example}[theorem]{Example}
\newtheorem{remark}[theorem]{Remark}

\numberwithin{equation}{section}

\newcommand{\StartEnum}{\begin{enumerate}[label=\arabic*{$^\circ$}.]}
\newcommand{\StopEnum}{\end{enumerate}}

\newcommand{\NN}{\mathbb{N}}
\newcommand{\ZZ}{\mathbb{Z}}
\newcommand{\QQ}{\mathbb{Q}}
\newcommand{\CC}{\mathbb{C}}
\newcommand{\hua}{\mathcal}

\renewcommand{\>}{\rangle}
\newcommand{\leftperp}[1]{{{}^\perp{#1}}}

\renewcommand{\mod}{\operatorname{mod}}

\newcommand{\Add}{\operatorname{Add}}
\newcommand{\Aut}{\operatorname{Aut}}
\newcommand{\Ind}{\operatorname{Ind}}
\newcommand{\Sim}{\operatorname{Sim}}
\newcommand{\Hom}{\operatorname{Hom}}
\newcommand{\End}{\operatorname{End}}
\newcommand{\Ext}{\operatorname{Ext}}
\newcommand{\Irr}{\operatorname{Irr}}

\newcommand{\diff}{\operatorname{d}}
\newcommand{\Br}{\operatorname{Br}}

\newcommand{\Proj}{\operatorname{Proj}}

\newcommand{\Cone}{\operatorname{Cone}}
\newcommand{\per}{\operatorname{per}}
\newcommand\Sph{\operatorname{Sph}}
\newcommand\FG{\operatorname{FG}}
\newcommand\gr{\operatorname{gr}}
\newcommand{\PSL}[2]{\operatorname{PSL}_{#1}(#2)}

\newcommand{\Grot}{K} 
\newcommand{\EG}{\operatorname{EG}}   
\newcommand{\EGp}{\operatorname{EG}^\circ}   
\newcommand{\C}[2]{\operatorname{\hua{C}}_{#1}(#2)}   
\newcommand{\shift}[1]{\operatorname{\Sigma}_{#1}}   
\newcommand{\CEG}[2]{\operatorname{CEG}_{#1}(#2)}   
\newcommand{\CEGun}[1]{\operatorname{CEG}^*(#1)}   
\newcommand{\D}{\operatorname{\hua{D}}}   
\newcommand{\twists}{\operatorname{Tw}}

\newcommand{\h}{\operatorname{\hua{H}}}   
\newcommand{\nh}{\operatorname{\widehat{\hua{H}}}}
\renewcommand{\k}{\mathbf{k}}
\newcommand{\Ho}[1]{\operatorname{\bf H}_{#1}}
\newcommand{\tilt}[3]{{#1}^{#2}_{#3}}
\newcommand{\torpr}[2]{\langle{#1},{#2}\rangle}
\newcommand{\qq}[1]{\operatorname{\Gamma}_{#1}Q}
\newcommand{\cc}[1]{\overline{#1}}

\newcommand{\shifts}[1]{#1[\ZZ]}
\newcommand{\zero}{\hua{H}_\Gamma}
\newcommand{\nzero}{\hua{H}_Q}

\newcommand{\cluster}{\pi}
\newcommand{\clu}{\cluster}
\newcommand{\cts}{Y}
\newcommand{\CTS}{\mathbf{Y}}

\newcommand{\fundom}[1]{\hua{S}_{#1}}
\newcommand{\Tri}{\bigtriangleup}
\newcommand{\imm}{\hua{L}}
\newcommand{\com}{\hua{J}}
\newcommand{\cov}{\alpha}
\newcommand{\covtil}{\widetilde{\cov}}

\newcommand{\Q}[1]{\mathcal{Q}(#1)}
\newcommand{\Qaug}[1]{\mathcal{Q}^+(#1)}
\newcommand{\Qext}[1]{\mathcal{Q}^{\varepsilon}(#1)}
\newcommand{\double}[2]{\Pi_{#2}(#1)}

\newcommand{\CY}[1]{CY-$#1$}
\begin{document}
\begin{abstract}
We study the oriented exchange graph
$\textrm{EG}^\circ(\Gamma_{N}\,Q)$
of reachable hearts in the finite-dimensional derived category
$\mathcal{D}(\Gamma_{N}\,Q)$
of the CY-$N$ Ginzburg algebra $\Gamma_{N}Q$
associated to an acyclic quiver $Q$.
We show that any such heart is induced from some heart in
the bounded derived category $\mathcal{D}(Q)$
via some `Lagrangian immersion'
$\mathcal{L}:\mathcal{D}(Q)\to\mathcal{D}(\Gamma_{N}\,Q)$.
We build on this to show that the quotient of
$\textrm{EG}^\circ(\Gamma_{N}\,Q)$
by the Seidel-Thomas braid group is the exchange graph
$\textrm{CEG}_{N-1}(Q)$ of cluster tilting sets in
the (higher) cluster category $\mathcal{C}_{N-1}(Q)$.
As an application, we interpret Buan-Thomas' coloured quiver
for a cluster tilting set in terms of the Ext quiver of
any corresponding heart in $\mathcal{D}(\Gamma_{N}\,Q)$.

    \vskip .3cm
    {\parindent =0pt
    \it Key words:} exchange graph, t-structure, Calabi-Yau category, higher cluster theory
\end{abstract}

\maketitle
\section{Introduction}

In this paper, we study three closely related triangulated categories
associated to an acyclic quiver $Q$.
These are the bounded derived category $\D(Q)$,
and, for any $N\geq 3$, the finite-dimensional derived category $\D(\qq{N})$ of
the \CY{N} Ginzburg algebra $\qq{N}$ (\cite{G})
and the cluster category $\C{N-1}{Q}$.
The later is defined as a quotient of $\D(Q)$ by a certain
cluster shift (see Section~\ref{sec:cluster} for details),
but may also be regarded as a quotient of
the perfect derived category $\per(\qq{N})$ by $\D(\qq{N})$,
using the natural \CY{N} version of Amiot's construction
(\cite{A},\cite{Guo1}).
In the case $N=3$, categories like $\D(\qq{N})$,
more generally associated to quivers with potential,
originally arose in studying the local geometry of
Calabi-Yau 3-folds,
hence the CY label.
They also arise on the other side of mirror symmetry;
for example, when $Q$ is of type $A_n$,
Khovanov, Seidel and Thomas (\cite{KS},\cite{ST}) identified
$\D(\qq{N})$ inside the derived Fukaya category of
the Milnor fibre of an $A_n$ singularity.

One key to understanding triangulated categories such as $\D(Q)$ and $\D(\qq{N})$
is to understand the collection of their (bounded) t-structures.
Each (bounded) t-structure carries an abelian category sitting inside it,
known as its heart.
The hearts of nearby t-structures are related by
tilting, in the sense of Happel-Reiten-Smal\o,
and the most elementary tilt is with respect to a single
rigid simple object in the heart (see Section~\ref{sec:tilting} for more details).
The relationship of simple tilting gives the set of
t-structures/hearts in $\D$ the structure of an oriented graph,
which we call the exchange graph $\EG(\D)$.
The operation of repeated tilting with respect to the same
simple (up to shift) gives the graph a sort of linear structure,
which will play an important role in this paper.

Typically, the set of all t-structures is unmanageably large and so one can restrict
attention to those hearts which are `finite', that is,
generated by finitely many simple objects.
This is a full subgraph of $\EG(\D)$, but it is typically not known to be connected
and so we further restrict attention to the principal component $\EGp(\D)$,
containing those hearts which are `reachable' from a canonical finite heart.
This exists for our examples of interest,
namely $\EGp(\D(Q))$ and $\EGp(\D(\qq{N}))$,
which we abbreviate as $\EGp(Q)$ and $\EGp(\qq{N})$.
We denote their canonical hearts by $\nzero$ and $\zero$.

We start by showing that $\EGp(Q)$ has a particularly uniform structure.
More precisely, we prove (Theorem~\ref{ppthm:psp}) that
all the simples in any heart in $\EGp(Q)$ are rigid and so can be tilted.
We do this by also showing
that each such heart has a dual set of projectives,
which is a silting set (see Remark~\ref{rem:silting-set})
and which mutates when the heart tilts.
The rigidity then follows by relating this silting set to a cluster tilting set.

Recall that, in (higher) cluster theory,
there also arises an exchange graph $\CEG{N-1}{Q}$
of cluster tilting sets in the cluster category $\C{N-1}{Q}$,
related by mutation.
Using the quotient functor
$\D(Q)\to \C{N-1}{Q}$, it is possible (cf. \cite[Thm.~2.4]{BRT}) to
identify cluster tilting sets with certain silting sets in $\D(Q)$.
Thus, by associating to a heart $\h$ its (silting) set of projectives $\Proj\h$,
we can define a map
\[
 \com\colon \EGp_N(Q,\nzero) \to \CEG{N-1}{Q}
\]
on a subgraph $\EGp_N(Q,\nzero) \subset \EGp(Q)$
\New{consisting of hearts between $\nzero$ and its shift $\nzero[N-2]$
(see Definition~\ref{def:eg} for precise details).}
We can then show that this map
induces a graph isomorphism (Theorem~\ref{thm:comparison})
\begin{equation} \label{eq:EGQisoCEG}
  \cc{\com}\colon\cc{\EGp_N}(Q,\nzero) \cong \CEG{N-1}{Q},
\end{equation}
where the domain is the `cyclic completion' (Definition~\ref{def:convex})
of $\EGp_N(Q,\nzero)$,
defined using the linear structure mentioned earlier.

On the other hand, via Amiot's construction,
we can also associate to every heart $\h\in\EGp(\qq{N})$ the
silting set $\Proj\h$
in $\per(\qq{N})$, whose image under the quotient functor
$\per(\qq{N})\to \C{N-1}{Q}$
is a cluster tilting set.
Thus we obtain a map
\[
\cov\colon \EGp(\qq{N}) \to \CEG{N-1}{Q}
\]
Our first main result (Theorem~\ref{thm:main}) is that the map $\cov$ is
invariant under the action 
of the Seidel-Thomas braid group, generated by the spherical twist functors,
\[
  \Br \subset\Aut\D(\qq{N})
\]
and that this map induces an isomorphism
\begin{equation} \label{eq:EGqqisoCEG}
 \covtil\colon \EGp(\qq{N})/\Br \cong \CEG{N-1}{Q}.
\end{equation}
In the \CY{3} case, the isomorphism \eqref{eq:EGqqisoCEG} is due to
Keller-Nicol\'{a}s in the more general context of quivers with potential
(see \cite[Thm.~5.6]{K6}).

We prove \eqref{eq:EGqqisoCEG} using \eqref{eq:EGQisoCEG},
by first proving (Theorem~\ref{thm:inducing}) that the analogous subgraph
\[
  \EGp_N(\qq{N},\zero) \subset \EGp(\qq{N})
\]
is isomorphic to $\EGp_N(Q,\nzero)$ via a canonical functor
\[
  \hua{I}\colon \D(Q)\to\D(\qq{N})
\]
which is a `Lagrangian immersion' (Definition~\ref{def:L-imm})
in the following sense.
The `tangent algebra', i.e. the derived endomorphism algebra,
$\Hom_Q^\bullet(X,X)$ of any object $X\in\D(Q)$
is identified with a subspace of
$\Hom_{\qq{N}}^\bullet\bigl(\hua{I}(X),\hua{I}(X)\bigr)$
whose quotient is dual to it (up to a shift).
Note that a heart $\h$ in $\EGp(\qq{N})$ is `induced' by
$\nh\in \EGp(Q)$, and we write $\h=\hua{I}_*(\nh)$,
when $\hua{I}$ maps the simples of $\nh$ to the simples of $\h$.
Note also that the isomorphism
\[
  \hua{I}_*\colon \EGp_N(Q,\nzero) \cong \EGp_N(\qq{N},\zero)
\]
preserves the linear structures and so it induces an isomorphism between
their cyclic completions.
Thus, in Theorem~\ref{thm:main}, we actually prove that
$\EGp_N(\qq{N},\zero)$ is a fundamental domain
for the action of $\Br$ on $\EGp(\qq{N})$,
in such a way that
\[
  \cc{\EGp_N}(\qq{N},\zero) \cong \EGp(\qq{N})/\Br
\]
and, in the process, we obtain a commutative diagram of graph isomorphisms
\[
\xymatrix@C=3pc{
    \cc{\EGp_N}(Q, \nzero) \ar[d]^{\cc{\hua{I}_*}}
        \ar[r]^{\cc{\com}}
        & \CEG{N-1}{Q}\\
    \cc{\EGp_N}(\qq{N}, \zero) \ar@{->}[r]
        & \EGp(\qq{N})/\Br \ar@{->}[u]^{\covtil}}
\]
The fact that the $\Br$-translates of $\EGp_N(\qq{N},\zero)$ cover
$\EGp(\qq{N})$ means (Corollary~\ref{cor:all-induced}) that
every heart in $\EGp(\qq{N})$ is induced by some Lagrangian
immersion $\imm=\varphi\circ \hua{I}\colon \D(Q)\to\D(\qq{N})$,
for some $\varphi\in\Br$.

For our second main result (Theorem~\ref{thm:quiver}),
we exploit this circle of identifications to interpret
the coloured quiver of Buan-Thomas \cite{BT},
which is associated to any cluster tilting set $\CTS=\cov(\h)$
in $\CEG{N-1}{Q}$,
in terms of the Ext quiver of any heart $\h$ in
the corresponding $\Br$-orbit in $\EGp(\qq{N})$,
or equivalently in the fundamental domain $\EGp_N(\qq{N},\zero)$.

Note that the Ext quiver $\Qext{\h}$ of a finite heart $\h$
is the (positively) graded quiver
whose vertices correspond to the simples $\{S_i\}$ of $\h$ and
whose graded edges $i\to j$ correspond to a basis of
$\Hom^\bullet(S_i, S_j)$.

The Ext quivers of a heart $\nh$ in $\EGp_N(Q, \nzero)$
and the corresponding induced heart $\h=\hua{I}_*(\nh)$ in $\EGp_N(\qq{N},\zero)$
may be easily related by the fact that $\hua{I}$
is a Lagrangian immersion.
More precisely we see (Proposition~\ref{pp:mono2}) that
\[
 \Qext{\h}=\double{\Qext{\nh}}{N}
\]
where the \CY{N} double $\double{\hua{Q}}{N}$
is the quiver obtained from $\hua{Q}$
by adding an arrow $j\to i$ of degree $N-k$ for each arrow $i\to j$ of degree $k$
and adding a loop of degree $N$ at each vertex.

Suppose that $\CTS=\cov(\h)=\com(\nh)$ is the corresponding
cluster tilting set in $\C{N-1}{Q}$. Then we define a
graded quiver $\Qaug{\CTS}$, which is a minor modification of
Buan-Thomas' coloured quiver, in that its arrows are graded $1,\ldots,N-1$
rather than $0,\ldots,N-2$ and it contains an extra loop (graded $N$)
at each vertex.
This `augmented' graded quiver $\Qaug{\CTS}$ has the expected symmetry
of the Ext quiver of a heart in a \CY{N} category,
and our key step (Theorem~\ref{thm:mono1}) is to prove that
\[
  \Qaug{\com(\nh)} = \double{\Qext{\nh}}{N},
\]
and so deduce that $\Qext{\h}=\Qaug{\cov(\h)}$, providing the promised interpretation
of the coloured quiver.

The \CY{3} case
\New{is much more studied (as surveyed in \cite{BY}) and}
is in many ways more uniform.
Indeed, we go on to show (Theorem~\ref{thm:egx}) that $\EGp_3(\qq{3},\h)$
is a fundamental domain for the $\Br$-action,
for \emph{any} heart $\h\in\EGp(\qq{3})$.
Since the oriented exchange graph $\CEG{2}{Q}$ is obtained
from the original unoriented cluster exchange graph $\CEGun{Q}$
by replacing each edge by a two-cycle, we may therefore consider that
$\EGp_3(\qq{3},\h)$ is an oriented version of $\CEGun{Q}$
and thus $\EGp(\qq{3})$ is covered by many such oriented versions.

To illustrate several of the main ideas, we conclude the paper
by explicitly describing the quotient of the
exchange graph $\EGp(\Gamma_ N Q)$ by the shift functor,
for a quiver $Q$ of type $A_2$,
and show how it is a rough combinatorial dual to the Farey graph.
This relationship has been made more geometric, in the \CY{3} case,
by Sutherland \cite{S1},
who shows that the hyperbolic disc, in which the Farey graph lives,
is naturally (the $\CC$-quotient of) the space of Bridgeland stability conditions
for $\D(\Gamma_3 Q)$.

This paper is part of the second author's PhD thesis \cite{Qy},
which also provides several other applications of exchange graphs,
such as to spaces of stability conditions and to quantum dilogarithm identities
(cf. \cite{Q2}).

\subsection*{Acknowledgements}
The second author's Ph.D studies were supervised by the first author and supported
by a U.K. Overseas Research Studentship.
The second author would also like to thank Bernhard Keller for
an invitation to Paris 7 in March 2011 and
for sharing his insight and expertise on cluster theory,
which helped to substantially improve an early version of this work.
The visit to Paris was supported by the U.K.-France network RepNet.

\section{Preliminaries}

For simplicity, let $\k$ be a fixed algebraically-closed field.
Let $Q$ be an acyclic quiver,
that is, a directed graph without oriented cycles.
The path algebra $\k Q$ is then finite dimensional.
We denote by $\mod\k Q$ the category of finite dimensional $\k Q$-modules and
let $\hua{D}(Q)=\hua{D}^b(\mod\k Q)$ be its bounded derived category,
which is a triangulated category.
Note that $\mod\k Q$ is hereditary, i.e. $\Ext^2(M,N)=0$ for all modules $M,N$,
and hence \cite{H}
\begin{equation} \label{eq:Happ}
  \Ind\D(Q)=\bigcup_{m\in\ZZ}\Ind(\mod\k Q)[m],
\end{equation}
where $\Ind\hua{C}$ denotes a complete set of indecomposables in an
additive category $\hua{C}$,
\New{that is, one indecomposable object from each isomorphism class}.
In addition, $\hua{D}(Q)$ has Auslander-Reiten (or Serre) duality,
i.e. a functor $\tau\colon\hua{D}(Q)\to\hua{D}(Q)$ with a natural duality
\begin{equation} \label{eq:AR}
  \Ext^1(X,Y)\cong \Hom(Y,\tau X)^*.
\end{equation}
for all objects $X,Y$ in $\hua{D}(Q)$.
Note: in any triangulated category $\D$,
we will write either $\Hom^k(X,Y)$ or $\Ext^k(X,Y)$ for $\Hom_{\D}(X,Y[k])$.

Recall (e.g. \cite[\S3]{B1}) that a \emph{t-structure}
on a triangulated category $\hua{D}$ is
a full subcategory $\hua{P} \subset \hua{D}$
with $\hua{P}[1] \subset \hua{P}$ and such that,
for every object $E\in\hua{D}$, there is a
\New{(necessarily unique)}
triangle $F \to E \to G\to F[1]$ in $\hua{D}$
with $F\in\hua{P}$ and $G\in\hua{P}^{\perp}$, where
\[
  \hua{P}^{\perp}=\{ G\in\hua{D}: \Hom_{\hua{D}}(F,G)=0,
  \forall F\in\hua{P}  \}.
\]
It follows immediately that we also have $\hua{P}^{\perp}[-1]\subset \hua{P}^{\perp}$ and
\[
  \hua{P} = \New{\leftperp{(\hua{P}^\perp)}} = \{ F\in\hua{D}: \Hom_{\hua{D}}(F,G)=0,
  \forall G\in\hua{P}^\perp  \}
\]
\New{and thus the $t$-structure is also determined by $\hua{P}^\perp$.}
Any t-structure is closed under sums and summands
and hence it is determined by the indecomposables in it.

A t-structure $\hua{P}$ is \emph{bounded} if
\[
  \hua{D}= \displaystyle\bigcup_{i,j \in \ZZ} \hua{P}^\perp[i] \cap \hua{P}[j],
\]
or equivalently if, for every object $M$,
the shifts $M[k]$ are in $\hua{P}$ for $k\gg0$ and in $\hua{P}^{\perp}$ for $k\ll0$.
The \emph{heart} of a t-structure $\hua{P}$ is the full subcategory
\[
  \h=  \hua{P}^\perp[1]\cap\hua{P}
\]
and any bounded t-structure is determined by its heart.
More precisely, any bounded t-structure $\hua{P}$
with heart $\h$ determines, for each $M$ in $\hua{D}$,
a canonical filtration (\cite[Lem.~3.2]{B1})
\begin{equation} \label{eq:canonfilt}
\xymatrix@C=0,5pc{
  0=M_0 \ar[rr] && M_1 \ar[dl] \ar[rr] &&  \cdots\ar[rr] && M_{m-1}
        \ar[rr] && M_m=M \ar[dl] \\
  & H_1[k_1] \ar@{-->}[ul]  && && && H_m[k_m] \ar@{-->}[ul]
  }
\end{equation}
where $H_i \in \h$ and $k_1 > \ldots > k_m$ are integers.
Using this filtration,
one can define the $k$-th homology of $M$, with respect to $\h$,
to be
\begin{equation}\label{eq:homology}
 \Ho{k}(M)=
 \begin{cases}
   H_i & \text{if $k=k_i$} \\
   0 & \text{otherwise.}
 \end{cases}
\end{equation}
Then $\hua{P}$ consists of those objects
with no (nonzero) negative homology,
$\hua{P}^\perp$ those with only negative homology
and $\h$ those with homology only in degree 0.

In this paper, we only consider bounded t-structures and their hearts,
and use the phrase `a triangulated category $\hua{D}$ with heart $\h$'
to mean that $\h$ is the heart of a bounded t-structure on $\hua{D}$.
Furthermore, all categories will be implicitly assumed to be $\k$-linear.

\begin{definition}\label{def:proj}
Let $\h$ be a heart in a triangulated category $\D$,
with corresponding t-structure $\hua{P}$.
We say that an object $P\in\D$ is a \emph{projective of} $\h$
if
\[
\text{for all $M\in\h$ and all $k\neq0$,}\;
  \Hom^k(P,M)=0;
\]
or equivalently if $P\in\hua{P}$ and $\Hom^1(P,L)=0$, for all $L\in\hua{P}$.

We denote by $\Proj\h$ a complete set of indecomposable projectives of $\h$.
\end{definition}

Note that a projective \emph{of} $\h$ is not necessary \emph{in} $\h$.
When $\hua{D}=\D(Q)$, we see more explicitly,
using Auslander-Reiten duality \eqref{eq:AR},
that $P\in\D(Q)$ is a projective of a heart $\h$ if and only if
\begin{equation}\label{eq:defproj}
    P\in\hua{P}\cap\tau^{-1}\hua{P}^\perp.
\end{equation}

\begin{proposition}\label{pp:proj}
Let $\h$ be a heart in a triangulated category $\D$
and $P$ a projective of $\h$.
Then, for any $k\in\ZZ$,
\begin{equation}\label{eq:projhom}
    \Hom^{-k}(P, M)=\Hom(P,\Ho{k}(M))
\end{equation}
where $\Ho{\bullet}$ is homology with respect to $\h$, as in \eqref{eq:homology}.

Hence, in particular, $\Proj\h$ is a partial silting set, that is,
\begin{equation}\label{eq:partial-silting}
 \text{for all $P_i,P_j\in\Proj\h$ and all $k>0$,}\quad
 \Hom^k(P_i,P_j)=0.
\end{equation}

Suppose further that $\Proj\h$ `spans' $\h$,
in the sense that, for any $M\in \h$,
\[
 \text{if $\Hom(P, M)=0$ for all $P\in \Proj\h$, then $M=0$.}
\]
Then $\Proj\h$ determines $\h$, by
\[
   M\in\h \;\iff\;
   \text{for all $P\in \Proj\h$ and all $k\neq0$,}\quad
   \Hom^k(P,M)=0.
\]
\end{proposition}

\begin{proof}
Suppose $M$ has a filtration as in \eqref{eq:canonfilt},
so that $M\in \hua{P}[k_m]\cap \hua{P}^\perp[k_1+1]$.
Then, by the definition of projective, we have
\[
   \Hom^{\geq1}(P, M)=0=\Hom^{\leq0}(P, L)
\]
for any $M\in\hua{P}$ and $L\in\hua{P}^\perp$.
Thus $\Hom^{-k}(P, M)=0$, for $k>k_1$ and $k<k_m$,
and $\Ho{k}(M)=0$ for the same range of $k$.
Now, applying $\Hom(P, -)$ to the triangle
\[
    M'[-1] \to \Ho{k_1}(M)[k_1] \to M \to M',
\]
gives $\Hom^{-k_1}(P, M)=\Hom(P,\Ho{k_1}(M))$,
because $M'\in\hua{P}^\perp[k_1]$.
But also
\[
  \Hom^{-k}(P, M)=\Hom^{-k}(P, M'),\;\forall k<k_1,
\]
because $\Ho{k_1}(M)\in\hua{P}^\perp[1]$.
Thus \eqref{eq:projhom} follows by induction.

Then \eqref{eq:partial-silting} is immediate,
because $P_j\in\hua{P}$, so $\Ho{k}(P_j)=0$, for all $k<0$.
The last part is also straightforward,
because, when $\Proj\h$ spans, the RHS implies that
$\Ho{k}(M)=0$, for all $k\neq0$,
which is equivalent to $M\in\h$.
\end{proof}

\begin{remark}\label{rem:silting-set}
In keeping with more recent usage (e.g. \cite[Sec.~2.1]{BRT}),
we use the term ``partial silting set'' for any set satisfying
\eqref{eq:partial-silting} and reserve ``silting set'' for a set which
is also maximal with respect to this property.
This is different from the original usage in \cite{KV},
where a `silting set' is simply required to satisfy \eqref{eq:partial-silting}.
\end{remark}

\New{We easily see that, if $\h_1$ and $\h_2$ are hearts with
$\h_1\subset\h_2$, then $\h_1=\h_2$.
However, there is an alternative, more interesting}
partial order on hearts given by inclusion of their corresponding t-structures.
More precisely, for two hearts $\h_1$ and $\h_2$ in $\hua{D}$,
with t-structures $\hua{P}_1$ and $\hua{P}_2$,
we say
\begin{equation}\label{def:ineq}
  \h_1 \leq \h_2
\end{equation}
if and only if $\hua{P}_2\subset\hua{P}_1$ ,
or equivalently $\h_2\subset \hua{P}_1$,
or equivalently $\hua{P}^\perp_1\subset\hua{P}^\perp_2$,
or equivalently $\h_1\subset \hua{P}^\perp_2[1]$.

A useful elementary observation is the following.
\begin{lemma}\label{lem:3H}
Given hearts $\h_1 \leq \h_2 \leq \h_3$,
any object $T$ in $\h_1$ and $\h_3$ is also in $\h_2$.
\end{lemma}

\begin{proof}
By assumption $T\in\hua{P}_3\subset\hua{P}_2$
and $T\in \hua{P}^\perp_1[1]\subset\hua{P}^\perp_2[1]$.
\end{proof}

\begin{remark}\label{rem:stdt}
Note that the heart $\h$ of a t-structure on $\hua{D}$
is always an abelian category,
but $\hua{D}$ is not necessarily equivalent to the derived category of $\h$.
On the other hand, any abelian category $\hua{C}$ is the heart
of a canonical t-structure on $\hua{D}(\hua{C})$.
Indeed, any object in $\hua{D}(\hua{C})$ may be considered as a complex in $\hua{C}$
and its ordinary homology objects are the factors
of the filtration \eqref{eq:canonfilt} associated to this canonical t-structure.
Moreover, in such cases the projectives of $\hua{C}$ coincide with the normal definition.
For instance, $\D(Q)$ has a canonical heart $\mod \k Q$,
which we will write as $\nzero$ from now on.
\end{remark}

\section{Tilting Theory} \label{sec:tilting}

A \New{parallel} notion to a t-structure on a triangulated category
is a torsion pair in an abelian category.
Tilting with respect to a torsion pair in the heart of a t-structure
provides a way to pass between different t-structures.

\begin{definition}
A \emph{torsion pair} in an abelian category $\hua{C}$ is an \New{ordered} pair of
full subcategories $\torpr{\hua{F}}{\hua{T}}$, 
such that $\Hom_{\hua{C}}(\hua{T},\hua{F})=0$ and, for 
every $E \in \hua{C}$, there is a short exact sequence
$0 \to E^{\hua{T}} \to E \to E^{\hua{F}} \to 0$,
for some $E^{\hua{T}} \in \hua{T}$ and $E^{\hua{F}} \in \hua{F}$.
\end{definition}

\New{It follows immediately that the short exact sequence is unique
and that, for any torsion pair $\torpr{\hua{F}}{\hua{T}}$,
we have $\hua{F}=\hua{T}^\perp$ and $\hua{T}=\leftperp{\hua{F}}$,
that is, either part of the pair determines the other.}

\begin{proposition} [Happel, Reiten, Smal\o \cite{HRS}] \label{pp:HRS}
Let $\h$ be a heart in a triangulated category $\hua{D}$.
Suppose further that $\torpr{\hua{F}}{\hua{T}}$ is a torsion pair in $\h$.
Then the full subcategory
\[
    \h^\sharp
    =\{ E \in \hua{D}:\Ho1(E) \in \hua{F}, \Ho0(E) \in \hua{T}
        \mbox{ and } \Ho{i}(E)=0 \mbox{ otherwise} \}
\]
is also a heart in $\hua{D}$, as is
\[
    \h^\flat
    =\{ E \in \hua{D}:\Ho{0}(E) \in \hua{F}, \Ho{-1}(E) \in \hua{T}
        \mbox{ and } \Ho{i}(E)=0 \mbox{ otherwise} \},
\]
where $\Ho{\bullet}$ is homology with respect to $\h$, as in \eqref{eq:homology}.
\end{proposition}

We call $\h^\sharp$ the \emph{forward tilt} of $\h$
with respect to $\torpr{\hua{F}}{\hua{T}}$
and $\h^\flat$ the \emph{backward tilt}. 
Note that $\h^\flat=\h^\sharp[-1]$.
Furthermore, $\h^\sharp$ has a torsion pair $\torpr{\hua{T}}{\hua{F}[1]}$
with respect to which the forward and backward tilts are
$\bigl(\h^\sharp\bigr)^\sharp=\h[1]$ and $\bigl(\h^\sharp\bigr)^\flat=\h$.
Similarly with respect to the torsion pair $\torpr{\hua{T}[-1]}{\hua{F}}$ in $\h^\flat$,
we have $\bigl(\h^\flat\bigr)^\sharp=\h$, $\bigl(\h^\flat\bigr)^\flat=\h[-1]$.

\begin{remark}\label{rem:mult-free}
It is immediate from Proposition~\ref{pp:HRS} that the forward tilt $\h^\sharp$
or the backward tilt $\h^\flat$, together with $\h$, determines
the corresponding torsion pair $\torpr{\hua{F}}{\hua{T}}$ by
\begin{equation*}
\hua{F}=\h^\flat\cap \h = \h^\sharp[-1]\cap\h,
\qquad
\hua{T}=\h^\sharp\cap\h = \h^\flat[1] \cap \h.
\end{equation*}
\end{remark}

\begin{proposition}\label{pp:filtration}
Let M be an indecomposable in $\hua{D}$ with
canonical filtration with respect to a heart $\h$,
as in \eqref{eq:canonfilt}.
Given a torsion pair $\torpr{\hua{F}}{\hua{T}}$ in $\h$,
the short exact sequences
\[
  0 \to H_i^{\hua{T}} \to H_i \to H_i^{\hua{F}} \to 0,
\]
can be used to refine the canonical filtration of $M$
to a finer one with factors
\begin{equation}\label{eq:filt 1}
\left(H_1^{\hua{T}}[k_1],H_1^{\hua{F}}[k_1],\,\ldots\,,
    H_m^{\hua{T}}[k_m],H_m^{\hua{F}}[k_m]\right),
\end{equation}
Furthermore, if we take the canonical filtration of $M$
with respect to the heart $\h^\sharp$ and refine it
using the torsion pair $\torpr{\hua{T}}{\hua{F}[1]}$,
then we obtain essentially the same filtration
\begin{equation}\label{eq:filt 2}
\Big(H_1^{\hua{T}}[k_1],\tilde{H}_1^{\hua{F}}[k_1-1],\,\ldots\,,
    H_m^{\hua{T}}[k_m],\tilde{H}_m^{\hua{F}}[k_m-1]\Big),
\end{equation}
where $\tilde{H}_i=H_i[1]$.
\end{proposition}

\begin{proof}
The existence of the filtrations \eqref{eq:filt 1} and \eqref{eq:filt 2}
follows by repeated use of the Octahedral Axiom.
\end{proof}

We observe how tilting relates to the partial ordering of hearts defined in \eqref{def:ineq}.

\begin{lemma}\label{lem:tiltorder}
Let $\h$ be a heart in $\hua{D}(Q)$.
Then $\h<\h[m]$ for $m>0$.
For any forward tilt $\h^\sharp$
and backward tilt $\h^\flat$,
we have $\h[-1] \leq \h^\flat \leq \h \leq \h^\sharp \leq \h[1]$.
\end{lemma}

\begin{proof}
Since $\hua{P}\supsetneq\hua{P}[1]$,
we have $\h<\h[m]$ for $m>0$.
By Proposition~\ref{pp:filtration} we have $\hua{P} \supset \hua{P}^\sharp$,
hence $\h \leq \h^\sharp$.
Noticing that $\left(\h^\sharp\right)^\sharp=\h[1]$ with respect to
the torsion pair $\torpr{\hua{T}}{\hua{F}[1]}$, we have $\h^\sharp \leq \h[1]$.
Similarly, $\h[-1] \leq \h^\flat \leq \h$.
\end{proof}

In fact, the forward tilts $\h^\sharp$ can be characterised as
precisely the hearts between $\h$ and $\h[1]$ (cf. \cite{HRS}).
The backward tilts $\h^\flat$ are similarly those between $\h[-1]$ and $\h$.


Now, recall that an object in an abelian category is \emph{simple}
if it has no proper subobjects, or equivalently
it is not the middle term of any (non-trivial) short exact sequence.
An object $M$ is \emph{rigid} if $\Ext^1(M,M)=0$.

\begin{lemma}\label{lem:simpletilt}
Let $S$ be a rigid simple object in a Hom-finite abelian category $\hua{C}$.
Then $\hua{C}$ admits a torsion pair $\torpr{\hua{F}}{\hua{T}}$
such that $\hua{F}=\<S\>$.
More precisely, for any $M\in\h$, in the corresponding short exact sequence
\begin{equation}\label{eq:g_0}
    0\to M^{\hua{T}}\to M \to M^{\hua{F}} \to 0
\end{equation}
we have
$M^{\hua{F}}=S\otimes\Hom(M,S)^*$.
Similarly, there is also a torsion pair with the torsion part $\hua{T}=\<S\>$,
obtained by setting $M^\hua{T}=S\otimes\Hom(S,M)$.
\end{lemma}

\begin{proof}
If we define $M^{\hua{F}}$ as in the lemma, then
there is a canonical surjection $M\to M^{\hua{F}}$,
whose kernel we may define to be $M^{\hua{T}}$,
yielding the short exact sequence \eqref{eq:g_0}.

Applying $\Hom(-,S)$ to \eqref{eq:g_0}, we get
\[
    0 \to \Hom(M^{\hua{F}}, S) \to \Hom(M, S)
    \to \Hom(M^{\hua{T}}, S) \to \Ext^1(M^{\hua{F}}, S) \to\cdots .
\]
But
\begin{gather*}
    \Hom(M^{\hua{F}}, S)=\Hom( S\otimes\Hom(M,S)^*, S) \cong \Hom(M, S),\\
    \Ext^1(M^{\hua{F}}, S)=\Ext^1( S\otimes\Hom(M,S)^*, S)=0,
\end{gather*}
so we have $\Hom(M^{\hua{T}}, S)=0$ and hence
$\Hom(M^{\hua{T}}, M^{\hua{F}})=0$ as required.
The proof of the second statement is similar.
\end{proof}

\begin{definition}\label{def:simpletilt}
A forward tilt of a heart $\h$ is \emph{simple},
if, in the corresponding torsion \New{pair} $\torpr{\hua{F}}{\hua{T}}$, the 
part $\hua{F}$ is generated by a single rigid simple $S$
\New{and thus $\hua{T}=\leftperp{S}$.}
We denote the new heart by $\tilt{\h}{\sharp}{S}$.
Similarly, a backward tilt of $\h$ is simple if 
$\hua{T}$ is generated by a rigid simple $S$
\New{and thus $\hua{F}=S^\perp$.}
The new heart is denoted by $\tilt{\h}{\flat}{S}$.
\end{definition}

For the canonical heart
$\nzero$ in $\D(Q)$,
an APR tilt (\cite[p.~201]{ASS1}),
which reverses all arrows at a sink/source of $Q$,
is an example of a simple (forward/backward) tilt.

We now prove a basic result about simple tilting
for hearts in an arbitrary triangulated category $\D$,
which will play a key role in Section~\ref{sec:cy3}
in the special case of $\D(\qq{3})$.

\begin{lemma} \label{pp:either-or}
Let $\h,\h_0$ be hearts in $\D$ with $\h_0[-1]\leq \h\leq\h_0$
and let $S$ be a rigid simple object in $\h$.
Then
\StartEnum
\item
  $\tilt{\h}{\sharp}{S} \leq \h_0$
  if and only if $S\in \h_0[-1]$ if and only if $S\notin \h_0$,
\item
  $\h_0[-1]\leq\tilt{\h}{\flat}{S}$
  if and only if $S\in \h_0$ if and only if $S\notin \h_0[-1]$.
\StopEnum
In particular, this means that $S$ must be in one of $\h_0$ or $\h_0[-1]$.
\end{lemma}

\begin{proof}
In each case, the second `only if' is immediate and the first `only if' follows from Lemma~\ref{lem:3H},
since, for $1^\circ$, we also have $\h_0\leq \h[1]$ and $S[1]$ is in both $\tilt{\h}{\sharp}{S}$ and $\h[1]$,
while, for $2^\circ$, we also have $\h[-1]\leq\h_0[-1]$ and $S[-1]$ is in both $\h[-1]$ and $\tilt{\h}{\flat}{S}$.
Thus we only need to show, in each case, that the last condition implies the first.

For $1^\circ$,
note that $S\in\hua{P}^\perp[1]\subset\hua{P}_0^\perp$,
so that $S\notin \h_0$ implies $S\notin \hua{P}_0$.
Suppose, for contradiction, that there is an object $M\in\hua{P}_0$,
but with $M \notin \tilt{\hua{P}}{\sharp}{S}$.
Consider the filtrations \eqref{eq:filt 1} and \eqref{eq:filt 2} of $M$,
with respect to $\h$, and the torsion pair corresponding to $\tilt{\h}{\sharp}{S}$.
Since $M\in\hua{P}_0\subset\hua{P}$, we have $k_m\geq 0$.
But $M \notin \tilt{\hua{P}}{\sharp}{S}$ forces $k_m=0$ and $H_m^\hua{F}=S^t\neq 0$.
In this case, there is a triangle
$M' \to M \to S^t \to M'[1]$ with $M'\in\hua{P}$.
Hence we have $M'[1]\in \hua{P}[1] \subset \hua{P}_0$ and, as $M \in \hua{P}_0$,
this implies $S\in\hua{P}_0$,
contradicting the initial observation.
Thus $\hua{P}_0 \subset \tilt{\hua{P}}{\sharp}{S}$,
that is, $\tilt{\h}{\sharp}{S} \leq \h_0$.

Similarly, for $2^\circ$,
we have $S\in\hua{P}\subset \hua{P}_0[-1]$,
so $S\notin \h_0[-1]$ implies $S\notin\hua{P}_0^\perp$.
If there is an object $M\notin\tilt{\hua{P}}{\flat}{S}[1]^\perp$, but with
$M\in\hua{P}_0^\perp \subset \hua{P}^\perp[1]$,
we deduce as before that $k_1=0$ with $H^{\hua{T}}_1=S^t\neq0$ in \eqref{eq:filt 2}.
Hence there is a triangle $M'[-1]\to S^t\to M \to M'$ with
$M'[-1]\in\hua{P}^\perp\subset\hua{P}_0^\perp$,
which implies $S\in\hua{P}_0^\perp$, contradicting the initial observation.
Thus $\hua{P}^\perp_0 \subset \tilt{\hua{P}}{\flat}{S}[1]^\perp$,
that is, $\h_0\leq \tilt{\h}{\flat}{S}[1]$.
\end{proof}

\section{Cluster theory} \label{sec:cluster}

We review some notions from (higher) cluster theory and, in particular,
describe the relationship between hearts and $m$-cluster tilting sets.
Note: here the `classical' case is $m=2$, although it would be $m=1$
for some authors' indexing.

\begin{definition} [cf. \cite{BMRRT, BT, IY, ZZ}] \label{def:cluster}
For any integer $m\geq 2$, the \emph{$m$-cluster shift} is the auto-equivalence of $\D(Q)$ given by
$\shift{m}=\tau^{-1}\circ[m-1]$.
\begin{itemize}
\item
    The \emph{$m$-cluster category} $\C{m}{Q}$ is
    the orbit category $\D(Q)/\shift{m}$ (cf. \cite{K5}), that is,
\[\begin{array}{rll}
        \Ext^k_{\C{m}{Q}}(M,L)
        &=\Hom_{\C{m}{Q}}(M,L[k])\\
        &=\bigoplus_{t\in\ZZ} \Hom_{\D(Q)}(M, \shift{m}^t L[k]).
\end{array}\]
\item
  An \emph{$m$-cluster tilting set} $\{\cts_j\}_{j=1}^n$ in $\C{m}{Q}$ is
  an Ext-configuration, i.e. a maximal collection
  of non-isomorphic indecomposables with
  $\Ext^k_{\C{m}{Q}}(\cts_i, \cts_j)=0$, for all $1\leq k\leq m-1$.
  The sum $\CTS=\bigoplus_{i=1}^n \cts_i$ is then an
  \emph{$m$-cluster tilting object}, and is equivalent information.
  (We will often omit the ``$m$-'' from ``$m$-cluster''
  when it is clear from the context.)
  Note that a maximal set necessarily has $n=\# Q_0$ \New{elements}
 (\cite[Thm.~3.3]{ZZ}).
  An \emph{almost complete cluster tilting set} in $\C{m}{Q}$ is
  a subset of a cluster tilting set with $n-1$ elements.
\item
    The \emph{forward mutation} $\mu_i$ at the $i$-th object
    acts on an $m$-cluster tilting set $\{\cts_j\}_{j=1}^n$,
    by replacing $\cts_i$ by 
    \begin{equation}\label{eq:mutate1}
        \cts_i^\sharp = \Cone(\cts_i \to \bigoplus_{j\neq i} \Irr(\cts_i,\cts_j)^*\otimes \cts_j),
    \end{equation}
    where $\Irr(\cts_i,\cts_j)$ is a space of irreducible maps $\cts_i\to \cts_j$,
    in the additive subcategory $\Add \CTS$ of $\C{m}{Q}$.
    When $Q$ is acyclic, we have $\Irr(\cts_i,\cts_i)=0$,
    that is, the Gabriel quiver of $\End(\CTS)$ has no loops (cf. \cite{BT}).
    Furthermore, the \emph{backward mutation} $\mu_i^{-1}$ replaces $\cts_i$ by
    \begin{equation}\label{eq:mutate2}
      \cts_i^\flat =  \Cone(\bigoplus_{j\neq i} \Irr(\cts_j,\cts_i)\otimes \cts_j \to \cts_i)[-1].
    \end{equation}
\item
    The \emph{exchange graph} $\CEG{m}{Q}$ of $m$-clusters is
    the oriented graph whose vertices are $m$-cluster tilting sets
    and whose edges are the forward mutations.
    Note that $\CEG{m}{Q}$ is connected (\cite[Prop.~7.1]{BRT}).
\end{itemize}
\end{definition}

In the case $m=2$, the exchange graph is usually presented as an unoriented graph
$\CEGun{Q}$, from which $\CEG{2}{Q}$ is obtained by replacing
each unoriented edge by an oriented two-cycle.
For instance, for $Q$ of type $A_3$, $\CEGun{Q}$ is the underlying
unoriented graph of Figure~\ref{fig:1} (cf. \cite[Fig.~4]{BMRRT}).
We will explain why this should be the case in Section~\ref{sec:cy3}.

To relate hearts in $\D(Q)$ and cluster tilting sets in $\C{m}{Q}$,
we consider the restriction of the quotient functor $\cluster_m\colon \D(Q)\to\C{m}{Q}$
to a fundamental domain $\fundom{m}$, defined (as in \cite[Sec.~2.2]{BRT3}, \cite[Prop.~2.2]{Z})
by
\begin{equation}\label{eq:fundom}
 \Ind \fundom{m} = \Proj\nzero [m-1] \cup \bigcup_{j=0}^{m-2} \Ind\nzero[j],
\end{equation}
where we recall that $\nzero=\mod \k Q$ is the canonical heart in $\D(Q)$
and that $\Proj\h$ is a complete set of indecomposable projectives for a heart $\h$,
while $\Ind\hua{C}$ is a complete set of indecomposables in an additive category $\hua{C}$.
Thus we obtain a bijection
\begin{equation}\label{eq:cluster}
   \cluster_m \colon \Ind \fundom{m} \cong \Ind\C{m}{Q}.
\end{equation}

\begin{lemma}\label{lem:inj}
Suppose that $\h$ is a heart in $\D(Q)$ with $\nzero\leq \h\leq \nzero[M]$.

If $m\geq M+1$, then $\Proj\h \subset \fundom{m}$ and
thus $\cluster_m\bigl(\Proj\h\bigr)$ is a partial $m$-cluster tilting set
and $\#\Proj\h\leq\#Q_0$.
If actually $\#\Proj\h=\#Q_0$, then $\cluster_m\bigl(\Proj\h\bigr)$
is an $m$-cluster tilting set and $\Proj\h$ is a silting set.

If $m\geq M+2$, then, for $\mathbf{P}=\bigoplus\Proj\h$, we have
\begin{equation}\label{eq:DC}
    \End_{\D(Q)}(\mathbf{P})
    =\End_{\C{m}{Q}}(\cluster_m(\mathbf{P})).
\end{equation}
\end{lemma}

\begin{proof}
By \eqref{eq:defproj} and the bounds on $\h$, we have
\begin{equation}\label{eq:where-proj}
 \Proj\h \;\subset\; \hua{P}\cap \tau^{-1} \hua{P}^\perp
 \;\subset\;  \hua{P}_Q\cap \tau^{-1} \hua{P}_Q^\perp[M] = \fundom{M+1},
\end{equation}
while $\fundom{M+1} \subset \fundom{m}$, when $m\geq M+1$.
This means in particular that (the images of) the projectives in $\Proj\h$ are distinct,
i.e. non-isomorphic, in $\C{m}{Q}$.

Since $\Proj\h$ is a partial silting set (see \eqref{eq:partial-silting}
in Proposition~\ref{pp:proj}),
we deduce (see \cite[Lem.~1.1]{W}) that
$\Ext_{\C{m}{Q}}^k(P_i,P_j)=0$, for all $1\leq k\leq m$,
and so $\cluster_m\bigl(\Proj\h\bigr)$ is a partial
$m$-cluster tilting set. In particular, $\#\Proj\h \leq\#Q_0$.
The claim in the case of equality follows from
\cite[Sec.~2]{BRT3} or \cite[Sec.~3]{ZZ}.

To prove \eqref{eq:DC}, we must show that,
for all $P_i,P_j\in\Proj\h$ and all $t\neq0$,
we have $\Hom(\shift{m}^t P_i, P_j)=0$.
Suppose first that $t\geq1$. Then \eqref{eq:where-proj} gives
\begin{equation}\label{eq:leminjcalc1}
    \shift{m}^t P_i
    \;\in\; \shift{m}^t \hua{P}_Q
    \;\subset\; \shift{m}^{t-1} \tau^{-1} \hua{P}_Q [m-1]
    \;\subset\; \tau^{-1} \hua{P}_Q [m-1]
\end{equation}
and also $P_j\in \tau^{-1} \hua{P}_Q^\perp[M]$.
But if $m-1\geq M$, then $\Hom(\hua{P}_Q[m-1],\hua{P}_Q^\perp[M])=0$
and so $\Hom(\shift{m}^t P_i, P_j)=0$.
On the other hand, suppose that $t\leq -1$.
Then
\begin{equation}\label{eq:leminjcalc2}
    \shift{m}^{t} P_i
    \;\in\;     \shift{m}^{t}  \tau^{-1}\hua{P}_Q^\perp[M]
    \;\subset\; \shift{m}^{t+1} \hua{P}_Q^\perp[M-m+1]
    \;\subset\; \hua{P}_Q^\perp[M-m+1]
\end{equation}
Thus $\shift{m}^{t} P_i \in \hua{P}_Q^\perp[-1]$,
when $m\geq M+2$, while $P_j\in\hua{P}_Q$.
But $\nzero$ is hereditary, so $\Hom(\hua{P}_Q^\perp[-1],\hua{P}_Q)=0$,
and so $\Hom(\shift{m}^{t} P_i, P_j)=0$
in this case as well.
\end{proof}

\begin{remark}\label{rem:silting}
Any heart $\h$ in $\D(Q)$ can be shifted to lie between
$\nzero$ and $\nzero[M]$ for some sufficiently large $M$ and so
we can deduce, from Lemma~\ref{lem:inj},
that $\#\Proj\h\leq\#Q_0$ and that $\Proj\h$ is a silting set if and only if $\#\Proj\h=\#Q_0$.

In the case that $\Proj\h$ is a silting set, i.e. $\mathbf{P}=\bigoplus\Proj\h$ is a silting object,
we can also then use \eqref{eq:DC} to deduce that the
Gabriel quiver of $\End_{\D(Q)}(\mathbf{P})$ has no loops.
This is because we know, by \cite[Sec.~2]{BT},
that this holds for the Gabriel quiver of $\End_{\C{m}{Q}}(\cluster_m(\mathbf{P}))$,
since $\cluster_m(\mathbf{P})$ is a cluster tilting object.
\end{remark}

\section{Exchange graphs}

The relationship of simple tilting (Definition~\ref{def:simpletilt}) gives a natural
graph structure on the collection of all hearts in a triangulated category.
We use this and the partial order \eqref{def:ineq} to define the main objects of study of the paper.

\begin{definition}\label{def:eg}
The \emph{total exchange graph} $\EG(\D)$ of a triangulated category $\D$
is the oriented graph
whose vertices are all hearts in $\D$
and whose edges correspond to simple forward tiltings between them.
When $\D$ has a canonical heart $\h_{\D}$,
the \emph{principal component} $\EGp(\D)$ is the connected component
of $\EG(\D)$ that contains $\h_{\D}$.
It contains precisely those hearts `reachable' by tilting from $\h_{\D}$.

For any heart $\h_0$ in $\D$ and any $N\geq 3$,
the \emph{interval} of length $N-2$ at $\h_0$
is the full subgraph of $\EG(\D)$ given by
\begin{equation}\label{eq:EGinterval}
 \EG_N(\D, \h_0) =  \bigl\{\h\in\EG(\D) \mid
        \h_0 \leq\h\leq\h_0[N-2] \bigr\},
\end{equation}
and we define the \emph{based exchange graph} $\EGp_N(\D, \h_0)$ with base $\h_0$
to be the principal component of the interval,
that is, the connected component that contains $\h_0$.
\end{definition}

We label each edge in $\EG(\D)$ (and its subgraphs) by the simple object of the tilting,
i.e. the edge from $\h$ to $\tilt{\h}{\sharp}{S}$ is labelled by $S$.
By Lemma~\ref{lem:tiltorder}, we have $\h<\tilt{\h}{\sharp}{S}$ for any simple tilting,
which implies that there are no loops or oriented cycles in the exchange graph.
By Remark~\ref{rem:mult-free}, if two hearts are related by a simple tilt,
then the simple is uniquely determined, so the exchange graph also has no multiple edges.

Note that $\EGp_N(\D, \h_0)$ is the principal component of the interval
and not the interval in the principal component, that is, $\EGp(\D)\cap\EG_N(\D,\h_0)$,
which is not necessary connected.
Indeed, even if $\h_0\in\EGp(\D)$, then we only know \emph{a priori} that
\begin{equation}\label{eq:inclusion}
  \EGp_N(\D, \h_0) \subset \EGp(\D)\cap\EG_N(\D,\h_0).
\end{equation}
In the special case when $\D=\D(Q)$ and $\h_0=\nzero$,
we will find that this is actually an equality: see \eqref{eq:prin}.

Although the definition of $\EGp_N(\D, \h_0)$ favours
$\h_0$ asymmetrically, we will see that, in cases of interest,
$\EGp_N(\D, \h_0)$ does usually contain $\h_0[N-2]$,
and indeed $\h_0[j]$ for $0\leq j\leq N-2$:
see Corollary~\ref{cor:tilt-shift}, Remark~\ref{rem:induced} and Corollary~\ref{cor:sink-source}.

When $\D=\D(Q)$, for a quiver $Q$, we will shorten $\D(Q)$ to $Q$
in the notation for exchange graphs,
e.g. write $\EGp(Q)$ for the principal component $\EGp(\D(Q))$ containing the canonical heart $\nzero$.
Note that, even for an acyclic quiver, $\EGp(Q)$ is usually a proper subgraph of $\EG(Q)$.
However, if $Q$ is a Dynkin quiver, then the two are equal, by a result
of Keller-Vossieck \cite{KV};
an alternative proof can be found in \cite[Appendix~A]{Q2}.

\begin{example}\label{ex:flaw}
Let Q be the quiver of type $A_3$ with straight orientation
and simple modules $X,Y,Z$.
A piece of the Auslander-Reiten quiver of $\D(Q)$ is as follows
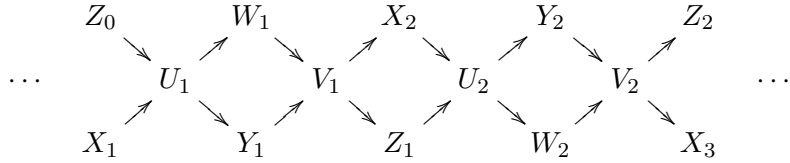
\[
\xymatrix@R=.7pc@C=.7pc{
   & Z_0 \ar[dr] && W_1 \ar[dr] && X_2 \ar[dr] && Y_2 \ar[dr] && Z_2\\
   \cdots & & U_1 \ar[ur]\ar[dr] && V_1 \ar[ur]\ar[dr] && U_2 \ar[ur]\ar[dr]
   && V_2 \ar[ur]\ar[dr] && \cdots \\
   & X_1 \ar[ur] && Y_1 \ar[ur] && Z_1 \ar[ur] && W_2 \ar[ur] && X_3}
\]
where $M_i=M[i-1]$ for $M\in\Ind\nzero$.
Figure~\ref{fig:1} is the exchange graph $\EGp_3(Q, \nzero)$,
where we denote each heart by a complete set of simples.

\begin{figure}\centering
\begin{tikzpicture}[scale=1.3, rotate=-117, xscale=-1,
 arrow/.style={->,>=stealth,thick},
 midlabel/.style={midway,fill=white}]
\path (0,0) node (x1)  {$_{X_1 Y_1 Z_2}$};
\path (2,1) node (x2)  {$_{X_1 Y_1 Z_1}$};
\path (-2,1) node (x3) {$_{X_2 U_1 Z_2}$};
\path (0,2) node (x4)  {$_{X_2 U_2 Z_1}$};
\path (0,4) node (x5)  {$_{W_1 Y_1 U_2}$};
\path (4,5) node (x6)  {$_{X_1 Y_2 V_1}$};
\path (2,5) node (x7)  {$_{X_2 Y_2 W_1}$};
\path (-2,5) node (x8) {$_{W_2 Y_1 Z_1}$};
\path (-4,5) node (x9) {$_{U_2 Y_1 Z_2}$};
\path (0,6) node  (x10)  {$_{V_1 Y_2 W_2}$};
\path (0,8) node  (x11) {$_{X_2 V_2 Z_1}$};
\path (2,9) node  (x12) {$_{X_1 V_2 Z_1}$};
\path (-2,9) node (x13) {$_{X_2 Y_2 Z_2}$};
\path (0,10) node  (x14){$_{X_1 Y_2 Z_2}$};

\draw[arrow] (x2)
    edge node[midlabel] {\tiny{$Z_1$}} (x1)
    edge node[midlabel] {\tiny{$X_1$}} (x4)
    edge node[midlabel] {\tiny{$Y_1$}} (x6);
\draw[arrow] (x1)
    edge node[midlabel] {\tiny{$X_1$}} (x3)
    edge [bend right=13, dashed] node[midlabel] {\tiny{$Y_1$}} (x14);
\draw[arrow] (x4)
    edge node[midlabel] {\tiny{$U_1$}} (x5)
    edge node[midlabel] {\tiny{$Z_1$}} (x3);
\draw[arrow] (x3)
    edge node[midlabel] {\tiny{$U_1$}} (x9);
\draw[arrow] (x5)
    edge node[midlabel] {\tiny{$Y_1$}} (x7)
    edge node[midlabel] {\tiny{$W_1$}} (x8);
\draw[arrow] (x6)
    edge node[midlabel] {\tiny{$X_1$}} (x7)
    edge node[midlabel] {\tiny{$V_1$}} (x12);
\draw[arrow] (x7)
    edge node[midlabel] {\tiny{$W_1$}} (x10);
\draw[arrow] (x8)
    edge node[midlabel] {\tiny{$Y_1$}} (x10)
    edge node[midlabel] {\tiny{$Z_1$}} (x9);
\draw[arrow] (x9)
    edge node[midlabel] {\tiny{$Y_1$}} (x13);
\draw[arrow] (x10)
    edge node[midlabel] {\tiny{$V_1$}} (x11);
\draw[arrow] (x11)
    edge node[midlabel] {\tiny{$Z_1$}} (x13);
\draw[arrow] (x14)
    edge node[midlabel] {\tiny{$X_1$}} (x13);
\draw[arrow] (x12)
    edge node[midlabel] {\tiny{$X_1$}} (x11)
    edge node[midlabel] {\tiny{$Z_1$}} (x14);
\end{tikzpicture}
\caption{The exchange graph $\EGp_3(Q, \nzero)$ for $Q$ of type $A_3$.}
\label{fig:1}
\end{figure}
\end{example}

\subsection{Finite hearts}

For any heart $\h$ in a triangulated category $\D$, we denote by
$\Sim\h$ a complete set of non-isomorphic simples in $\h$.

\begin{definition}
We say that a heart $\h$ is
\begin{itemize}
\item \emph{finite},
  if $\Sim\h$ is a finite set which generates $\h$ by means of extensions,
  i.e. every object $M$ in $\h$ has a finite filtration with simple factors.
  Note that, by the Jordan-H\"{o}lder Theorem, these factors are uniquely
  determined up to reordering.
\item \emph{rigid} if every simple $S$ in $\h$ is rigid,
 i.e. $\Ext^1(S,S)=0$.
\end{itemize}
Note that, if $\h$ is finite, then the classes of the simples in $\Sim\h$
form a basis of the Grothendieck group $\Grot(\h)=\Grot(\D)$.
\end{definition}

For example, when $Q$ is acyclic,
the canonical heart $\nzero$ in $\EGp(Q)$ is finite and rigid.
Our main interest is in the part of the exchange graph that contains finite hearts.
We also know, from Lemma~\ref{lem:simpletilt}, that we can tilt with respect
rigid simples, so a rigid heart is one for which we can tilt with respect to all simples.
Thus, to see how adjacent hearts in the exchange graph are related, we begin by determining how the simples of a finite heart change under simple tilting.

\begin{proposition}\label{pp:fini}
In any triangulated category $\D$,
let $S$ be a rigid simple in a finite heart $\h$.
Then after a forward or backward simple tilt
(Definition~\ref{def:simpletilt}) the new simples are
\begin{eqnarray}
  \Sim\tilt{\h}{\sharp}{S} &=&
  \{S[1]\}\;\cup\;\{\tilt{\psi}{\sharp}{S}(X)\mid X\in \Sim\h,X\neq S\},
\label{eq:psp-1} \\
  \Sim\tilt{\h}{\flat}{S} &=&
  \{S[-1]\}\;\cup\;\{\tilt{\psi}{\flat}{S}(X)\mid X\in \Sim\h,X\neq S\},
\label{eq:psp-2}
\end{eqnarray}
where
\begin{eqnarray}
\label{eq:psi+}
  \tilt{\psi}{\sharp}{S}(X)
&=&\Cone\left(X\to S[1]\otimes\Ext^1(X, S)^* \right) [-1],
\\ \label{eq:psi-}
  \tilt{\psi}{\flat}{S}(X)
&=&\Cone \left(S[-1]  \otimes  \Ext^1(S, X)\to X \right).
\end{eqnarray}
Thus $\tilt{\h}{\sharp}{S}$ and $\tilt{\h}{\flat}{S}$ are also finite, with
$\# \Sim\tilt{\h}{\sharp}{S} = \#\Sim\h = \#\Sim\tilt{\h}{\flat}{S}$.
\end{proposition}

\begin{proof}
We only deal with the case for forward tilting;
the backwards case is similar.
Let $\torpr{\hua{F}}{\hua{T}}$ be the torsion pair in $\h$ whose forward tilt
yields $\tilt{\h}{\sharp}{S}$.
Any simple in $\tilt{\h}{\sharp}{S}$ is either in $\hua{T}$ or $\hua{F}[1]$.
Since $S$ has no self extension,
we have $\hua{F}=\{S^m \mid m\in \NN\}$.
Furthermore, choose any simple quotient $S_0$ of $S[1]$ in $\tilt{\h}{\sharp}{S}$.
$S_0$ cannot be in $\hua{T}$ since $\Hom( \hua{F}[1], \hua{T})=0$.
Thus $S_0\in \hua{F}[1]$ which implies $S[1]=S_0$,
i.e. $S[1]\in\Sim\tilt{\h}{\sharp}{S}$.

Let $X\ncong S$ be any other simple in $\h$,
so that $X\in\hua{T}$.
Let $T$ be a simple \New{subobject} of $X$ in $\tilt{\h}{\sharp}{S}$ and
$f: T \to X$ be a non-zero map.
Since $\Hom( S[1], X)=0$,
\New{we know $T\not\in\hua{F}[1]$, so $T\in\hua{T}$.}
Because \New{$X$} is simple in $\h$
and $T$ is simple in $\tilt{\h}{\sharp}{S}$,
there are short exact sequences
\begin{equation}\label{eq:sesL}
    0 \to L \to T \xrightarrow{f} X \to 0
\qquad\text{and}\qquad
    0 \to T \xrightarrow{f} X \xrightarrow{g} M \to 0
\end{equation}
in $\h$ and $\tilt{\h}{\sharp}{S}$ respectively.
Thus $L=M[-1]$.
On the other hand $\h^\sharp[-1] \cap \h =\hua{F}$,
which implies $L\in \hua{F}$ \New{and $M\in \hua{F}[1]$ and so, canonically,}
\begin{equation}\label{eq:canon}
L \cong S \otimes \Hom(L, S)^*
\qquad\text{and}\qquad
 \New{M \cong S[1] \otimes \Hom^1(M, S)^*.}
\end{equation}

Applying $\Hom(-,S)$ to \eqref{eq:sesL}
\New{and noting that $T\in \hua{T}= \leftperp{S}$,
while $T$ and $S[1]$ are non-isomorphic simples in $\tilt{\h}{\sharp}{S}$,}
we get
\begin{equation}\label{eq:g*}
    0=\Hom(T, S) \to \Hom^1(M, S) \xrightarrow{g^*} \Hom^1(X, S) \to \Hom^1(T, S)=0
\end{equation}
and so $g^*$ is an isomorphism.
\New{By naturality of the (horizontal) universal maps, the following square commutes
\begin{equation}\label{eq:universal}
    \xymatrix@C=2.7pc{
    X \ar[r]\ar[d]^g&  S[1]\otimes\Hom^1(X, S)^* \ar[d]^{g_*}\\
    M \ar[r]^{\cong\qquad\qquad}& S[1]\otimes\Hom^1(M, S)^*      }
%
%
\end{equation}
and, since $g_*$ is induced by $g^*$ in \eqref{eq:g*}, it is also an isomorphism.
Thus we have identified $g\colon X\to M$ with the
universal map $X\to S[1]\otimes \Hom^1(X, S)^* $ and hence
$\tilt{\psi}{\sharp}{S}(X)$ is identified with $\Cone(g)[-1]=T$,
which is simple, as required.}

Now, if $\h$ is finite, with $\#\Sim\h=n$,
then the RHS of \eqref{eq:psp-1} contains
$n$ simples in $\tilt{\h}{\sharp}{S}$,
whose classes form a new basis of the Grothendieck group
$\Grot(\D)\cong\Grot(\h)\cong\ZZ^n$.
Hence these new simples of $\tilt{\h}{\sharp}{S}$
are non-isomorphic and must be a complete set of simples
as also $\Grot(\D)\cong \Grot(\tilt{\h}{\sharp}{S})$.
\end{proof}

A first useful consequence is the following.

\begin{corollary}\label{cor:tilt-shift}
 If $Q$ is an acyclic quiver,
 then there is a sequence of forward tilts from $\nzero$ to $\nzero[1]$.
 Hence
 $\nzero[k]\in\EGp_N(Q, \nzero)$, for $0\leq k\leq N-2$, and
 $\nzero[k]\in\EGp(Q)$, for all $k\in\ZZ$.
 \end{corollary}

\begin{proof}
Since $Q$ is acyclic, we can write
$\Sim\nzero=\{S_1,\ldots,S_n\}$, ordered so that
\begin{equation}\label{eq:ordered}
  \Ext^1(S_j,S_i)=0,\quad\text{for $1\leq i<j\leq n$.}
\end{equation}
Then Proposition~\ref{pp:fini} implies that forward tilting $\nzero$ by $S_1$,
gives a heart $\h$ with $\Sim\h=\{S_2,\ldots,S_n,S_1[1]\}$,
whose simples still satisfy \eqref{eq:ordered} in this new order.
Hence by iterated forward tilting $\nzero$
with respect to $S_1,\ldots,S_n$ we obtain $\nzero[1]$, as claimed.

Since this is a sequence of forward tilts, all the hearts in the sequence lie between
$\nzero$ and $\nzero[1]$, i.e. in the interval $\EG_3(Q, \nzero)$,
so $\nzero[1]$ is in its principal component $\EGp_3(Q, \nzero)$,
giving the second claim in the case $N=3$ (the case $N=2$ being trivial).

By shifting the sequence, we can get from $\nzero[1]$ to $\nzero[2]$, and so on,
by forward tilts, to get the second claim for general $N$.
We also get from $\nzero$ to $\nzero[-1]$, and so on, by backwards tilts,
to get the last claim for all $k\in\ZZ$.
\end{proof}

We next identify, again for a general triangulated category $\D$,
elementary criteria for when a heart
$\h$ is in the interval $\EG_N(\D,\h_0)$,
that is, $\h_0\leq\h\leq\h_0[N-2]$ (Definition~\ref{def:eg}),
and also when its forward and backward tilts remain in this interval.
Note that, for the later, the case $N=3$ is already covered by Lemma~\ref{pp:either-or}.

\begin{lemma}\label{lem:parallel}
Let $\h_0, \h\in \EG(\D)$ be finite hearts.
Then $\h\in\EG_N(\D,\h_0)$
if and only if, for all $S\in\Sim\h$,
\[
  \Ho{m}(S)=0,\quad\text{for $m\notin[0,N-2]$},
\]
where the homology $\Ho{\bullet}$ is with respect to $\h_0$.
Moreover,
\New{if $\h\in\EG_N(\D,\h_0)$, then,}
for any rigid $S\in\Sim\h$, 
\StartEnum
\item $\tilt{\h}{\flat}{S}\in\EG_N(\D,\h_0)$
  if and only if $\Ho{0}(S)=0$,
\item $\tilt{\h}{\sharp}{S}\in\EG_N(\D,\h_0)$
  if and only if $\Ho{N-2}(S)=0$.
\StopEnum
\end{lemma}

\begin{proof}
The condition that $\Ho{m}(S)=0$, for $m< 0$, is equivalent to
$S\in \hua{P}_0$ and as $\h$ is generated by its simples,
this is equivalent to $\h\subset \hua{P}_0$, i.e. $\h_0\leq\h$.
The other inequality is similar.
The necessity in the second assertions is then immediate.

As $\h$ is finite,
\eqref{eq:psp-1} and \eqref{eq:psp-2} in Proposition~\ref{pp:fini}
determine the simples in
$\tilt{\h}{\sharp}{S}$ and $\tilt{\h}{\flat}{S}$
and thus the sufficiency in the second assertions also
follows from the first one.
\end{proof}

\begin{corollary}\label{cor:ss}
Let $\h_0$ and $\h$ be finite hearts in $\EG(\D)$,
\New{with $\h\in\EG_N(\D,\h_0)$.}
Then, for any rigid $S\in\Sim\h$, we have
\StartEnum
\item $\tilt{\h}{\flat}{S}\in\EG_N(\D,\h_0)$
  if and only if $\Hom(S, T)=0$, for every $T\in\Sim\h_0$,
\item $\tilt{\h}{\sharp}{S}\in\EG_N(\D,\h_0)$
  if and only if $\Hom(T, S)=0$, for every $T\in\Sim\h_0[N-2]$.
\StopEnum
\end{corollary}
\begin{proof}
By the first part of Lemma~\ref{lem:parallel}, we have
$\Ho{m}(S)=0$ for $m\notin[0,N-2]$,
where $\Ho{\bullet}$ is with respect to $\h_0$.
Let $\Ho{0}(S)=H\in \h_0$. Then there is a triangle
\[
  H[-1] \to S' \to S \to H,
\]
where $\Ho{\leq0}S'=0$, i.e. $S'\in \hua{P}_0[1]$.
Hence, for any $T\in\Sim\h_0\subset \hua{P}^\perp_0[1]$,
we have $\Hom^{\leq0}(S', T)=0$ and so $\Hom(S, T)=\Hom(H, T)$.
This implies that $\Hom(S, T)=0$ for any $T\in\Sim\h_0$ if and only if $H=0$.
Then the first claim follows from
the second part of Lemma~\ref{lem:parallel}.

For the second claim, let $\Ho{N-2}(S)=H\in \h_0$.
Then there is a triangle
\[
  H[N-2] \to S \to S' \to H[N-1].
\]
Now, for any $T\in\Sim\h_0[N-2]$, we have
$\Hom^{\leq0}(T,S')=0$ and so $\Hom(T, S)=\Hom(T,H[N-2])$
and the claim follows as before.
\end{proof}

\subsection{Tilting and mutation of projectives}

Given a finite heart $\h$ with simples $\Sim\h=\{S_1,\ldots,S_n\}$,
it is natural to seek a set of projectives $\{P_1,\ldots,P_n\}$
\New{of $\h$ (as in Definition~\ref{def:proj}),}
which is \emph{dual} to $\Sim\h$, in the sense that
\begin{equation}\label{eq:dualPS}
    \dim\Hom(P_i, S_j)=\delta_{ij}.
\end{equation}
Furthermore, we expect that these are a complete set of indecomposable projectives.
This property holds, of course, for the canonical heart $\nzero$ in $\D(Q)$
and we will show that it is preserved under simple tilting.
In this way, we see in particular that, for all $\h\in\EGp(Q)$, the set
$\Proj\h$ has the right number of elements to make it a silting set,
by Lemma~\ref{lem:inj}.

We first show that this duality implies the familiar relationship
(e.g. in a module category)
between irreducible maps of projectives and extensions of simples.

\begin{proposition}\label{pp:irr=ext}
Let $\h$ be a finite heart in a triangulated category $\D$,
with simples $\Sim\h=\{S_1,\ldots,S_n\}$
and a dual set of projectives $\{P_1,\ldots,P_n\}$,
in the sense of \eqref{eq:dualPS}.
Then the $P_j$ are non-isomorphic, indecomposable
and span $\h$.
Furthermore
\begin{equation}\label{eq:irr=ext}
    \Irr(P_i, P_j) \cong \Ext^1(S_j,S_i)^*,
\end{equation}
where $\Irr(P_i, P_j)$ are the irreducible maps in the additive subcategory
$\Add\{P_1,\ldots,P_n\}$.
\end{proposition}

\begin{proof}
Note first that if a projective $P$ satisfies $\Hom(P,S)=0$ for all $S\in \Sim\h$,
then it satisfies $\Hom(P,M)=0$ for all $M\in\h$, as $\Sim\h$ generates $\h$.
But then $\Hom^k(P,M)=0$ for all $M\in\h$ and all $k$, as $P$ is projective
and so, using the canonical filtration \eqref{eq:canonfilt},
$\Hom(P,X)=0$ for all $X\in\D$, which means $P=0$.

Thus, as the duality \eqref{eq:dualPS} means that at most one indecomposable summand of
$P_j$ can have a non-zero map to $S_j$, any other summand would be trivial., i.e. $P_j$ is indecomposable.
The duality immediately implies that the $P_j$ are pairwise non-isomorphic and they span $\h$,
because it is generated by $\Sim\h$.

To prove \eqref{eq:irr=ext}, we start by defining $\Omega_j=\Cone \left(P_j\to S_j \right)[-1]$, so that
we have a triangle
\begin{equation}\label{eq:omega}
  S_j[-1] \xrightarrow{h} \Omega_j \to P_j \to S_j,
\end{equation}
Applying $\Hom(-,S_i)$ to this triangle, for any $S_i$, yields an isomorphism
\begin{equation}\label{eq:hom=ext}
  h^*\colon \Hom(\Omega_j,S_i)\xrightarrow{\cong}\Ext^1(S_j,S_i)
\end{equation}
and tells us that $\Hom(\Omega_j,S_i[-1])=0$, and so $\Hom(\Omega_j,M[-1])=0$, for all $M\in\h$.
Since \eqref{eq:omega} immediately gives $\Hom(\Omega_j,M[-k])=0$, for $k>1$ and all $M\in\h$,
we deduce that $\Omega_j\in\hua{P}$, the t-structure associated to $\h$.
Hence, applying $\Hom(P_i,-)$ to \eqref{eq:omega} yields a short exact sequence
\begin{equation}\label{eq:P-omega}
  0\to \Hom(P_i,\Omega_j)\to \Hom(P_i,P_j)\to\Hom(P_i,S_j)\to 0.
\end{equation}
Next define $\Omega^i_j=\Cone \left(\Omega_j\to S_i\otimes\Hom(\Omega_j,S_i)^*\right)[-1]$, so that
we have a triangle
\begin{equation}\label{eq:omega-ij}
  S_i\otimes\Hom(\Omega_j,S_i)^* [-1] \to \Omega^i_j \to \Omega_j \to S_i\otimes\Hom(\Omega_j,S_i)^*.
\end{equation}
Applying $\Hom(-,S_k)$ to this triangle, for any $k$, yields again $\Hom(\Omega^i_j,S_k[-1])=0$
and thus that $\Omega^i_j\in\hua{P}$ as before.
Hence applying $\Hom(P_k,-)$ to \eqref{eq:omega-ij} yields the exact sequence
\begin{equation}\label{eq:P-omega-ij}
  0\to \Hom(P_k,\Omega^i_j)\to \Hom(P_k,\Omega_j)\to\Hom(P_k,S_i)\otimes\Hom(\Omega_j,S_i)^*\to 0.
\end{equation}
Combining this, when $k=i$, with \eqref{eq:hom=ext} gives
\begin{equation}\label{eq:RHS=irr}
  \Ext^1(S_j,S_i)^* \cong \Hom(P_i,\Omega_j) / \Hom(P_i,\Omega^i_j)
\end{equation}
To see that the RHS is $\Irr(P_i,P_j)$, note that it would be had
$\Omega_j$ and $\Omega^i_j$ been constructed in the analogous way
inside $\mod\End(\mathbf{P})$, where $\mathbf{P}=\bigoplus_{i=1}^n P_i$.
However, the duality \eqref{eq:dualPS} means that $\Hom(\mathbf{P},S_j)$
are the simple $\End(\mathbf{P})$-modules,
while $\Hom(\mathbf{P},P_j)$ are the corresponding projectives.
Then the short exact sequences \eqref{eq:P-omega} and \eqref{eq:P-omega-ij} mean that
$\Hom(\mathbf{P},\Omega_j)$ and $\Hom(\mathbf{P},\Omega^i_j)$
are precisely the analogous $\End(\mathbf{P})$-modules
and so the result follows by the Yoneda Lemma.
\end{proof}

We now prove the first important result about the hearts in the exchange graph
$\EGp(Q)$,  observing in the process that, when a heart tilts, its projectives mutate
in a way precisely analogous to cluster tilting sets, as in
\eqref{eq:mutate1} and \eqref{eq:mutate2}.
A broader, more general result has been proved independently by
Koenig-Yang~\cite{KY}.

\begin{theorem}\label{ppthm:psp}
Let $Q$ be an acyclic quiver with $n$ vertices
and $\h$ be any heart in $\EGp(Q)$.
Then $\h$ is finite and rigid,
with exactly $n$ simples $\Sim\h=\{S_1,\ldots,S_n\}$
and exactly $n$ indecomposable projectives $\Proj\h=\{P_1,\ldots,P_n\}$,
which are dual in the sense of \eqref{eq:dualPS}.

If we tilt $\h$ with respect to any simple $S_i\in\Sim\h$, then
the simples change according to the formulae
\eqref{eq:psp-1} and \eqref{eq:psp-2} and
the projectives change according to the formulae
\begin{eqnarray}
\label{eq:pspp1}
    \Proj\tilt{\h}{\sharp}{S_i}&=&\Proj\h-\{P_i\}\cup\{P^\sharp_i\},\\
\label{eq:pspp2}
    \Proj\tilt{\h}{\flat}{S_i}&=&\Proj\h-\{P_i\}\cup\{P^\flat_i\},
\end{eqnarray}
where
\begin{eqnarray}
\label{eq:Psharp}
 P_i^\sharp &=& \Cone(P_i \to \bigoplus_{j\neq i} \Irr(P_i,P_j)^*\otimes P_j),
\\ \label{eq:Pflat}
 P_i^\flat &=&  \Cone(\bigoplus_{j\neq i} \Irr(P_j,P_i)\otimes P_j \to P_i)[-1],
\end{eqnarray}
with $\Irr(P_i,P_j)$ as in \eqref{eq:irr=ext}.
\end{theorem}

\begin{proof}
\newcommand{\univ}{\upsilon} 
We use induction starting from the canonical heart $\nzero$,
which we know is finite and rigid, with standard simples and projectives
satisfying \eqref{eq:dualPS}. We will only give the proof for forward tilting,
but a very similar proof works for backward tilting and thus we can reach all
hearts in $\EGp(Q)$.

So, suppose that $\h$ is a finite and rigid heart satisfying \eqref{eq:dualPS},
\New{with associated $t$-structure $\hua{P}$,}
and that $S_i\in\Sim\h$. By Proposition~\ref{pp:fini}, we know that
$\tilt{\h}{\sharp}{S_i}$ is finite with new simples given by \eqref{eq:psp-1},
that is, they are $S_i[1]$ together with $S_j^\sharp=\tilt{\psi}{\sharp}{S_i}(S_j)$,
for $j\neq i$, occurring in the triangle
\begin{equation}\label{eq:Sj-sharp}
  S_j[-1] \xrightarrow{u} E_j^*\otimes S_i \to S_j^\sharp \to S_j,
\end{equation}
where $E_j=\Ext^1(S_j,S_i)$ and $u$ is the universal map.

We claim that the new projectives are given by \eqref{eq:pspp1}
and that they satisfy \eqref{eq:dualPS} with respect to the new simples.
First, to see that $P_k$, for $k\neq i$, remains projective, note that
$\Ext^1(P_k,M)=0$ for any $M\in\hua{P}\supset\tilt{\hua{P}}{\sharp}{S_i}$,
so we just need to show $P_k$ is in
$\tilt{\hua{P}}{\sharp}{S_i}=\hua{P}\cap {^\perp}S_i$,
and this follows from \eqref{eq:dualPS}.
Next, applying $\Hom(P_k,-)$ to \eqref{eq:Sj-sharp}
gives $\Hom(P_k,S_j^\sharp)\cong\Hom(P_k,S_j)$ and, as also $\Hom(P_k,S_i[1])=0$,
most of the new duality \eqref{eq:dualPS} holds
and it remains to consider the case of $P_i^\sharp$.

Applying $\Hom(-,S_i)$ to \eqref{eq:Psharp} gives
$\Hom^k(P_i^\sharp,S_i[1])\cong\Hom^k(P_i,S_i)$, for all $k$,
so it remains to show that $\Hom^k(P_i^\sharp,S_j^\sharp)=0$, for all $k$
and all $j\neq i$.
As well as completing the duality,
this will mean that $\Hom^k(P_i^\sharp,M)=0$, for all $k\neq 0$
and all $M\in \tilt{\h}{\sharp}{S_i}$, so that $P_i^\sharp$ is a new projective.

Since $S_j^\sharp\in\h$, the required vanishing follows immediately from applying
$\Hom(-,S_j^\sharp)$ to \eqref{eq:Psharp}, except in the cases $k=0,1$.
These two cases appear in the long exact sequence
\begin{equation}\label{eq:delta-star}
0\to \Hom(P_i^\sharp,S_j^\sharp) \to \Irr(P_i,P_j)
 \xrightarrow{\delta_*} \Hom(P_i,S_j^\sharp) \to \Hom(P_i^\sharp,S_j^\sharp[1])\to 0,
\end{equation}
where $\delta\in\Hom(P_j,S_j^\sharp)$ is any non-zero map in this 1-dimensional space.
Thus what we must show is that $\delta_*$ is an isomorphism.

For this, we recall from the proof of Proposition~\ref{pp:irr=ext} that
we associated to $S_j$ the syzygy
\New{$\Omega_j\in\hua{P}$,}
defined by the triangle \eqref{eq:omega},
with a map $h\colon S_j[-1]\to \Omega_j$ inducing
an isomorphism $h^*\colon \Hom(\Omega_j, S_i)\cong E_j$.
Hence we may factor the universal map $u$ in \eqref{eq:Sj-sharp}
through $h$ and the universal map
\[
  \univ:\Omega_j\to \Hom(\Omega_j, S_i)^*\otimes S_i \cong E_j^*\otimes S_i.
\]
Applying the Octahedral Axiom to this factorisation $u=\univ\circ h$
gives the commutative diagram of triangles in Figure~\ref{fig:octahedron},
where $\Omega_j^i$ is as in \eqref{eq:omega-ij} and
so, in particular, $\Omega_j^i\in\hua{P}$.

\begin{figure}\centering\[
  \xymatrix{
    & \Omega_j^i \ar[d]\ar@{=}[r] & \Omega_j^i \ar[d]\\
    S_j[-1] \ar@{=}[d] \ar[r]^{h} & \Omega_j \ar[r] \ar[d]^{\univ}
        & P_j \ar[d]^{\delta} \ar[r] & S_j \ar@{=}[d]\\
    S_j[-1] \ar[r]^{u} & E_j^*\otimes S_i \ar[d] \ar[r] &
        S_j^\sharp \ar[d]\ar[r] & S_j\\
    & \Omega_j^i[1] \ar@{=}[r]& \Omega_j^i[1]\\
  }
\]\caption{Octahedral diagram for $u=\univ\circ h$.}
\label{fig:octahedron}
\end{figure}

Notice that the right square ensures that $\delta$ is nonzero and so
provides the map required in \eqref{eq:delta-star}.
In fact, we can also observe that $\Omega_j^i$ is the new syzygy $\Omega_j^\sharp$,
and hence is actually in $\hua{P}^\sharp_{S_i}\subset \hua{P}$,
although this is more than we need to know.
But now we can apply $\Hom(P_i,-)$ to the right hand vertical triangle in
Figure~\ref{fig:octahedron}
and deduce, as required, that $\delta_*$ is an isomorphism,
because, as explained after  \eqref{eq:RHS=irr}, $\Irr(P_i,P_j)$ is a complement
to $\Hom(P_i,\Omega_j^i)$ in $\Hom(P_i,\Omega_j)\cong\Hom(P_i,P_j)$
and also $\Hom(P_i,\Omega_j^i[1])=0$.

Thus we have found $n$ new projectives of the new heart
$\tilt{\h}{\sharp}{S_i}$, which are dual to the new simples.
Hence these new projectives are non-isomorphic and indecomposable,
by the first part of Proposition~\ref{pp:irr=ext},
and so they must form $\Proj\tilt{\h}{\sharp}{S_i}$,
as Remark~\ref{rem:silting} implies there can be no more projectives.
This also shows that $\Proj\tilt{\h}{\sharp}{S_i}$ is a silting set.

To complete the inductive step, we must observe that the
new heart $\tilt{\h}{\sharp}{S_i}$ is rigid.
By Proposition~\ref{pp:irr=ext}, this equivalent to the fact that
$\Irr(P,P)=0$, for all $P\in \Proj \tilt{\h}{\sharp}{S_i}$,
i.e. that the Gabriel quiver of $\End_{\D(Q)}(\mathbf{P})$ has no loops, for
$\mathbf{P}=\bigoplus \Proj \tilt{\h}{\sharp}{S_i}$.
This holds by Remark~\ref{rem:silting}, as $\Proj \tilt{\h}{\sharp}{S_i}$
is a silting set.
\end{proof}

\subsection{Convexity and comparison with cluster exchange graphs}

Using Theorem~\ref{ppthm:psp}, we can begin to relate the exchange graphs for
$\D(Q)$ and $\C{m}{Q}$.
In particular, it tells us that any heart in $\EGp_N(Q,\nzero)$ has $\# Q_0$ projectives
and so we can use Lemma~\ref{lem:inj}, for any $m\geq N-1$, to define a map
\begin{equation}\label{eq:inj}
     \com_{N,m}:\EGp_N(Q,\nzero) \rightarrow \CEG{m}{Q},
\end{equation}
sending a heart $\h$ to the $m$-cluster tilting set $\cluster_m\bigl(\Proj\h\bigr)$,
with $\cluster_m$ as in \eqref{eq:cluster}.

\begin{proposition}\label{pro:psp}
The map $\com_{N,m}$ is injective on vertices, for all $m\geq N-1$,
and preserves edges, when $m\geq N$.
\end{proposition}

\begin{proof}
The map is injective on vertices,
because $\Proj\h$ spans $\h$, by Proposition~\ref{pp:irr=ext} and
hence determines $\h$, by Proposition~\ref{pp:proj}.

For the map to preserve edges, we need to show that,
within $\EGp_N(Q,\nzero)$, the cluster tilting set
$\cluster_m(\Proj\tilt{\h}{\sharp}{S_i})$ is the forward mutation
of $\cluster_m(\Proj\h)$ at $\cluster_m(P_i)$,
for the corresponding projective $P_i$.
When $m\geq N$, this follows from \eqref{eq:DC} in Lemma~\ref{lem:inj},
as the mutation formula \eqref{eq:mutate1} for $\cluster_m(\Proj\h)$
then agrees with the mutation formula \eqref{eq:Psharp} for $\Proj\h$.
\end{proof}

To show that $\com_{N,N-1}$ also preserves edges, we need
to look more carefully at the exchange graph $\EGp_N(Q,\nzero)$
and in particular study its `convexity'.

\begin{definition}\label{def:line}
Let $\h$ be a heart in a triangulated category $\D$.
For $S \in \Sim\h$, we set $\tilt{\h}{0\sharp }{S}=\h$ and inductively define
\[
    \tilt{\h}{m\sharp }{S}
    ={  \Big( \tilt{\h}{ (m-1) \sharp}{S} \Big)  }^{\sharp}_{S[m-1]} \, ,
\]
for $m\geq1$, and similarly $\tilt{\h}{m\flat }{S}$, for $m\geq1$.
For $m<0$, we also set $\tilt{\h}{m\sharp }{S}=\tilt{\h}{-m\flat }{S}$.

The \emph{line} $l=l(\h, S)$ in $\EG(\D)$ is then the full subgraph
consisting of the vertices $\{ \tilt{\h}{m\sharp}{S} \}_{m\in\ZZ}$.
We say an edge in $\EG(\D)$ has \emph{direction} $T$
if its label is $T[m]$ for some integer $m$;
we say a line $l$ has direction $T$
if some (and hence every) edge in $l$ has direction $T$.

A \emph{line segment} of length $m$ in $\EG(\D)$
is the full subgraph consisting of vertices $\{\tilt{\h}{i\flat}{S}\}_{i=0}^{m-1}$
of some line $l(\h,S)$ in $\EG(\D)$ and for some positive integer $m$.
Notice that any line segment inherits a direction from the corresponding line
and in particular line segments of length one consisting of the same vertex
may differ by their directions.
\end{definition}

\newcommand{\subg}{G}

\begin{definition}\label{def:convex}
A subgraph $\subg$ of $\EG(\D)$ is \emph{convex} if any line in $\EG(\D)$
that meets $\subg$ meets it in a single line segment.
Define the \emph{cyclic completion} of a convex subgraph $\subg$
to be the oriented graph $\cc{\subg}$ obtained from $\subg$ by adding an edge
$e_l=\left(\h\to\tilt{\h}{(m-1)\flat}{S}\right)$ with direction $S$
for each line segment $l\cap{\subg}=\{\tilt{\h}{i\flat}{S}\}_{i=0}^{m-1}$
of direction $S$, in $\subg$.
Call the line segment $l\cap{\subg}$ together with $e_l$ a \emph{basic cycle}
(induced by $l$ with direction $S$) in $\subg$.
\end{definition}

It is easy to see that any interval
$\EG_N(\D, \h_0)$, as in \eqref{eq:EGinterval}, is a convex subgraph,
as is its initial component
$\EGp_N(\D, \h_0)$.

\begin{proposition}\label{pp:parallel}
$\EGp_N(Q,\nzero)$ is a convex subgraph of $\EGp(Q)$,
in which every maximal line segment has length $N-1$.
Further, $\EGp_N(Q,\nzero)$ has a unique source $\nzero$ and a unique sink $\nzero[N-2]$.
\end{proposition}

\begin{proof}
Let $\h\in\EGp_N(Q,\nzero)$ and $S\in\Sim\h$.
Then $S$ is indecomposable in $\D(Q)$ and hence,
by \eqref{eq:Happ} and the first part of Lemma~\ref{lem:parallel},
in $\nzero[m]$ for some integer $0\leq m\leq N-2$.
By the second part of Lemma~\ref{lem:parallel}, we have
\[l(\h,S)\cap\EGp_N(Q,\nzero)=\{\tilt{\h}{i\sharp}{S}\}_{i=-m}^{N-2-m}\]
which implies the first statement.

If $\h\in\EGp_N(Q,\nzero)$,
with $S\in\Sim\h$,
and $\tilt{\h}{\flat}{S}\not\in\EGp_N(Q,\nzero)$,
then by Lemma~\ref{lem:parallel}, $\Ho{0}(S)\neq0$
and so, by \eqref{eq:Happ}, $S\in\nzero$.
Thus, if $\h$ is a source, then $\h\subset\nzero$ and so $\h=\nzero$,
i.e. $\nzero$ is the unique source.
Similarly for the uniqueness of the sink.
\end{proof}

As an immediate consequence of Proposition~\ref{pp:parallel},
any basic cycle in the cyclic completion $\cc{\EGp_N}(Q,\nzero)$ is an $(N-1)$-cycle.

\begin{theorem}\label{thm:comparison}
The map $\com=\com_{N,N-1}$, as in \eqref{eq:inj}, induces a
canonical isomorphism
\begin{equation}\label{eq:J}
    \cc{\com}: \cc{\EGp_N}(Q, \nzero)  \cong  \CEG{N-1}{Q}.
\end{equation}
between oriented graphs. Moreover, this isomorphism induces a
bijection between basic cycles in $\cc{\EGp_N}(Q, \nzero)$
and almost complete cluster tilting sets in $\C{N-1}{Q}$.
\end{theorem}

\begin{proof}
We write $\clu=\cluster_{N-1}\colon \D(Q)\to \C{N-1}{Q}$.
Any maximal line segment in $\EGp_N(Q, \nzero)$ is of the form
\[
    l(\h,S_i)\cap\EGp_N(Q, \nzero)=\{\h_j\}_{j=0}^{N-2},
    \;\text{where $\h_j=\tilt{\h}{j\flat}{S_i}$.}
\]
Let $A_l=\bigcap_{j=0}^{N-2} \Proj\h_j$
and $P_i^j=\Proj\h_j-A_l$.
By \eqref{eq:pspp2}, we have $\#A_l=n-1$,
which implies $\clu(A_l)$ is an almost complete cluster tilting set.
By \cite[Thm.~4.3]{ZZ},
any almost complete cluster tilting set has precisely $N-1$ completions,
and hence $\{\com(\h_k)\}_{k=0}^{N-2}$
are all the completions of $\clu(A_l)$.

We claim that
\begin{equation}\label{eq:edge x}
  \com(\h_{j-1})=\mu_i \com(\h_j),
\end{equation}
for $j=2,\ldots,N-2$, where $\mu_i$ is mutation at $\clu(P_i^{j})$.
Assuming this for the moment,
we deduce that \eqref{eq:edge x} also holds for $j=1$ and
\[
    \com(\h_{N-2})=\mu_i\com(\h_0),
\]
since $\{\com(\h_k)\}_{k=0}^{N-2}$ forms a $(N-1)$-cycle in $\CEG{N-1}{Q}$
(cf. \cite{IY}).
Therefore $\com$ preserves edges
and can be extended to the required map $\cc{\com}$
that sends each new edge
$e_l=\left(\h\to\tilt{\h}{(N-2)\flat}{S_i}\right)$ in any basic cycle to the
mutation $\mu_i$ on $\com(\h)$ at $\clu(P_i^0)$.

The fact that $\cc{\com}$ is a
graph epimorphism, i.e. surjective on vertices and edges,
follows by induction, starting at $\nzero$ and using the fact that
$\CEG{N-1}{Q}$ is connected and that $\cc{\com}$ is a local isomorphism,
because both (oriented) graphs are $(n,n)$-regular,
i.e. have $n$ incoming and $n$ outgoing edges at each vertex,
where $n=\# Q_0$.
We already know, from Proposition~\ref{pro:psp}, that $\cc{\com}$ is
injective on vertices and hence on edges, since there are  no multiple edges or loops,
so we have the required isomorphism.
Moreover, $l \longmapsto A_l$ gives the canonical bijection
between basic cycles and almost complete cluster tilting sets.

To see that \eqref{eq:edge x} does indeed hold, we first show that
\begin{equation}\label{eq:DCx}
    \Hom_{D(Q)}(P_i^j, P)
    =\Hom_{\C{N-1}{Q}}(\clu(P_i^j), \clu(P))
\end{equation}
for any $P\in A_l$.
By the first part of Lemma~\ref{lem:parallel} and \eqref{eq:Happ},
we know that 
$S_i\in\nzero[N-2]$.
Further, by \eqref{eq:dualPS}, we have
\begin{equation}\label{eq:P_i^j}
    \Hom^\bullet(P_i^j,S_i[-j])=\Hom(P_i^j,S_i[-j])\neq0.
\end{equation}
But $\Hom(M,S_i[-j])=0$ for $M\in\hua{P}_Q[N-1-j]$
as $S_i[-j]\in\nzero[N-2-j]$,
so we have $P_i^j\in\hua{P}_Q\cap \hua{P}_Q^\perp[N-1-j]$.
Since $j\geq2$,
we have $\Hom(\shift{N-1}^t P_i^j, P)=0$, for any $t\neq 0$ and $P\in A_l$,
by the same calculation as in the proof of Lemma~\ref{lem:inj},
which implies \eqref{eq:DCx}.
Then we deduce that
$V_P'=\Irr(\clu(P_i^j), \clu(P))$
is induced from $V_P\subset\Irr(P_i^j, P)$, for any $P\in A_l$.
Hence the triangle
\[
    \clu(P_i^j)
    \to \bigoplus_{P \in A_l} (V_P')^*\otimes \clu(P)
    \to \mu_i(\clu(P_i^j))
    \to \clu(P_i^j)[1]
\]
in $\C{N-1}{Q}$ is induced from some triangle
\begin{equation}\label{eq:X}
    P_i^j \to \bigoplus_{P \in A_l} V_P^*\otimes P \to X \to P_i^j[1].
\end{equation}
Note that $\mu_i(\clu(P_i^j))$ is one of the complements of $A_l$
and thus $X=P_i^k$ for some $k\in[0,N-2]$.
Applying $\Hom(-,S_i)$ to \eqref{eq:X} gives
\[
    \Hom(X,S_i[1-j])=\Hom(P_i^j,S_i[-j])\neq0,
\]
which implies that $X=P_i^{j-1}$, by \eqref{eq:P_i^j},
and hence \eqref{eq:edge x} holds, as required.
\end{proof}

An immediate corollary is that we can
improve on \eqref{eq:inclusion} in this case, i.e. we have
\begin{equation}\label{eq:prin}
    \EGp_N(Q,\nzero)=\EGp(Q)\cap\EG_N(Q, \nzero).
\end{equation}
To see this, note that any heart $\h$ in the RHS is finite with $\# Q_0$ projectives,
by Theorem~\ref{ppthm:psp}, and so $\cluster_{N-1}\bigl(\Proj\h\bigr)$
is in $\CEG{N-1}{Q}$, by Lemma~\ref{lem:inj}.
Then Theorem~\ref{thm:comparison} implies that $\h$ is in the LHS of \eqref{eq:prin}, as required.

\begin{remark}\label{rem:BRT}
We expect that the isomorphism \eqref{eq:J} of Theorem~\ref{thm:comparison}
is equivalent to the bijection of Buan-Reiten-Thomas \cite[Thm.~2.4(c)]{BRT}
between $m$-cluster tilting sets and their $m$-$\Hom_{\leq0}$-configurations,
since one can plausibly conjecture that these are precisely the sets
$\Sim\h$ for $\h\in\EGp_N(Q, \nzero)$ and indeed, more generally, that their $\Hom_{\leq0}$-configurations
are the sets $\Sim\h$ for $\h\in\EGp(Q)$.
\end{remark}

\section{Coloured quivers and Ext quivers}

Recall from \cite[\S 2]{BT}, that any cluster tilting set $\CTS=\{\cts_1,\ldots,\cts_n\}$
in $\CEG{m}{Q}$
determines a coloured quiver $\Q{\CTS}$
\New{with vertex set $\CTS$. For each $0\leq c\leq m-1$, there are $r^{(c)}_{ij}$ arrows $\cts_i\to \cts_j$
with colour $c$, where
\[ r^{(c)}_{ij}=\dim\Irr(\cts^{(c)}_i,\cts_j)\]
and $\cts^{(c)}_i$ is defined
recursively by $\cts^{(0)}_i=\cts_i$ and $\cts^{(c+1)}_i=\bigl(\cts^{(c)}_i\bigr)^\sharp$,
as in \eqref{eq:mutate1}.}
%
Furthermore, $\Q{\CTS}$ is monochromatic and skew-symmetric,
in the sense that the arrows from $\cts_i$ to $\cts_j$ have just one colour,
$c$ say, 
and then the arrows from $\cts_j$ to $\cts_i$ have colour $m-1-c$,
with the same multiplicity $V$.
We will denote this situation by
\[\xymatrix@C=2pc{
    V \cdot \bigl( \cts_i \ar@/^.5pc/[r]^c
        \ar@/_.5pc/@{<-}[r]_{m-1-c} & \cts_j \bigr).
}\]

\begin{definition}\label{def:augquiv}
Given such a coloured quiver $\Q{\CTS}$, we will
define the \emph{augmented graded quiver}, denoted by $\Qaug{\CTS}$,
with the same vertex set, containing the arrows of $\Q{\CTS}$,
with degree equal to their colour plus one
(so that the degrees of opposite arrows now sum to $m+1$),
together with an additional loop of degree $m+1$ at each vertex.
\end{definition}

\begin{definition}\label{def:extquiv}
Let $\h$ be a finite heart in a triangulated category $\D$.
The \emph{Ext quiver} $\Qext{\h}$ is the (positively) graded quiver
whose vertices are the simples of $\h$ and
whose degree $k$ arrows $S_i \to S_j$ correspond to a basis of
$\Hom_{\D}(S_i, S_j[k])$.
Further, define the \emph{\CY{N} double} of a graded quiver $\hua{Q}$,
denoted by $\double{\hua{Q}}{N}$,
to be the quiver obtained from $\hua{Q}$
by adding an arrow $S_j\to S_i$ of degree $N-k$ for each arrow $S_i\to S_j$ of degree $k$
and adding a loop of degree $N$ at each vertex.
\end{definition}

The objective of this section is to relate the coloured quiver of a cluster tilting set
$\CTS$, or more precisely its augmented graded quiver, to the double of the
Ext quiver of a heart in $\D(Q)$ that corresponds to $\CTS$.
In the process, we uncover another important property of such hearts.

\begin{definition}\label{def:mono}
We say that a heart $\h$ is
\begin{itemize}
\item \emph{monochromatic} (cf. \cite{BT})
  if, for any simples $S\neq T$ in $\Sim\h$,
  $\Hom^\bullet(S,T)$ is concentrated in a single (positive) degree;
\item \emph{strongly monochromatic}
  if it is monochromatic and in addition,
  for any simples $S\neq T$ in $\Sim\h$,
  $\Hom^\bullet(S,T)=0$ or $\Hom^\bullet(T,S)=0$;
\end{itemize}
\end{definition}

\begin{proposition}\label{pp:mono1}
Any heart in $\EGp(Q)$ is strongly monochromatic.
Moreover, for any heart $\h\in\EGp_N(Q, \nzero)$
and any $m\geq N$,
\begin{equation}\label{eq:QCY1}
    \Qaug{ \com_{N,m}(\h) }=\double{\Qext{\h}}{m+1}.
\end{equation}
where $\com_{N,m}=\cluster_m\circ\Proj$ is as in \eqref{eq:inj}.
\end{proposition}
\begin{proof}
Note first that the loops on both sides match by construction
and so \eqref{eq:QCY1} just needs to be checked between two vertices.
\New{The key to the proof is to apply 
\eqref{eq:irr=ext} at a suitable place
on a maximal line segment in $\EGp_{m+1}(Q, \nzero)$ and to follow how
the degrees change under mutation/tilting.}

Choose any simples $T_i, T_j\in\h$.
By the first part of Lemma~\ref{lem:parallel} and \eqref{eq:Happ},
we have $T_i,T_j\in\bigcup_{t=0}^{N-2}\nzero[t]$
and hence
\begin{equation}\label{eq:hom-bound}
  \Hom^{\geq N}(T_i,T_j)=0,
\end{equation}
since $\nzero$ is hereditary.
Thus the maximum degree of any arrow in $\Qext{\h}$ is $N-1$.

Consider the maximal line segment
\begin{equation}\label{eq:maxseg}
    l(\h, T_i)\cap\EGp_{m+1}(Q, \nzero)=\{\h_k\}_{k=0}^{m-1},
\end{equation}
where $\h_k=\tilt{(\h_0)}{k\sharp}{S_i}$
and $\h_0$ has simples $S_i, S_j$, corresponding to $T_i,T_j$, with
associated projectives $P_i,P_j$.
Note that $\h_0\in\EGp_N(Q,\nzero)$
and that $\h=\h_h$ for some $0\leq h\leq N-2$, with $T_i=S_i[h]$ and $S_i\in\nzero$.
Therefore $\Hom^{\geq2}(S_i, S_j)=0$,
because $S_j\in\bigcup_{t=0}^{N-2}\nzero[t]$ and $\nzero$ is hereditary.
Thus
\begin{equation}\label{eq:1}
    \Hom^\bullet(S_i, S_j)=\Hom^1(S_i, S_j).
\end{equation}

Suppose that, in the coloured quiver $\Q{\cluster_m(\Proj\h_0)}$,
the full sub-quiver between $\cluster_m(P_i)$ and $\cluster_m(P_j)$ is
\begin{equation}\label{eq:Pij}
    \xymatrix@C=2pc{
    V \cdot \bigl(
        \cluster_m(P_j) \ar@/^.5pc/[r]^c \ar@/_.5pc/@{<-}[r]_{m-1-c} & \cluster_m(P_i)
    \bigr)
}\end{equation}
for some $0\leq c\leq m-1$ and multiplicity $V$.
Then, by the mutation rule in \cite{BT},
in the coloured quivers $\Q{\cluster_m(\Proj\h_k)}$, we have the following
full sub-quivers:
\begin{eqnarray}
    \xymatrix@C=3pc{
    V \cdot \bigl(
      \cluster_m(P_j) \ar@/^.5pc/[r]^{c-k}
      \ar@/_.5pc/@{<-}[r]_{m-1-c+k} & \cluster_m(P_i^k) \bigr)}
    &\quad& k=0,\ldots,c
\label{eq:quiver 1}\\
    \xymatrix@C=3pc{
    V \cdot \bigl(
      \cluster_m(P_j) \ar@/^.5pc/[r]^{m-k+c}
      \ar@/_.5pc/@{<-}[r]_{k-c-1} & \cluster_m(P_i^k) \bigr)}
    &\quad& k=c+1,\ldots,m-1
\label{eq:quiver 2}
\end{eqnarray}
where $P_i^k$ is the projective in $\h_k$ corresponding to $S_i[k]$
and hence $\cluster_m(P_i^k)$ is the replacement of
\New{$\cluster_m(P_i)$} in $\cluster_m(\Proj\h_k)$.
For $0\leq k \leq m-2$, the hearts $\h_k$ satisfy \eqref{eq:DC}
and so, by  \eqref{eq:irr=ext},
the number of colour-zero arrows in \eqref{eq:quiver 1} or \eqref{eq:quiver 2}
equals $\dim\Ext^1$ between the corresponding simples.
Arguing inductively, using Proposition~\ref{pp:fini}, these are
$S_j$ and $S_i[k]$ for \eqref{eq:quiver 1}
and $S_j^\sharp$ and $S_i[k]$ for \eqref{eq:quiver 2},
where $S_j^\sharp=\tilt{\psi}{\sharp}{S_i[c]}(S_j)$.

Consider the case when $c\neq m-1$.
We have
\begin{gather*}
    \Ext^1(S_j,S_i[k])=0=\Ext^1(S_i[k],S_j),\quad k=0,\ldots,c-1,\\
    \dim\Ext^1(S_j,S_i[c])=V, \quad \Ext^1(S_i[c],S_j)=0,\\
    \dim\Ext^1(S_i[c+1],S_j^\sharp)=V, \quad \Ext^1(S_j^\sharp,S_i[c+1])=0, \quad
    \text{if $c<m-2$}\\
    \Ext^1(S_j^\sharp,S_i[k])=0=\Ext^1(S_i[k],S_j^\sharp),\quad k=c+2,\ldots,m-2.
\end{gather*}
In particular $\Ext^1(S_i,S_j)=0$ and so, by \eqref{eq:1},
$\Hom^{\bullet}(S_i,S_j)=0$.
Also $S_i$ is exceptional, because it is rigid and $\D(Q)$ is hereditary.
Then, applying $\Hom(S_i,-)$ and $\Hom(-,S_i)$ to \eqref{eq:psi+},
a direct calculation shows that
\begin{gather*}
    \Hom^{\bullet}(S_i,S_j^\sharp)=\Hom^{-c}(S_i, S_j^\sharp)\cong
    \Ext^1(S_j, S_i[c])^*,
\\  \Hom^k(S_j,S_i)\cong\Hom^k(S_j^\sharp,S_i),\quad \forall k\neq c, c+1,
\\  \Hom^c(S_j^\sharp,S_i)=\Hom^{c+1}(S_j^\sharp,S_i)=0.
\end{gather*}
Since the degree of any arrow in the Ext quiver $\Qext{\h_0}$
is between $1$ and $N-1$ and $m\geq N$, we have
\[
    \Hom^{\bullet}(S_j,S_i)=\Hom^{c+1}(S_j, S_i),\quad
    \Hom^{\bullet}(S_j^\sharp,S_i)=0.
\]
Therefore, the full sub-quiver between $T_j, T_i$ in $\Qext{\h}$ is
\begin{eqnarray*}
    \xymatrix@C=3pc{
    V \cdot \bigl(
      S_j \ar[r]^{c-h+1} & S_i[h] \bigr)}
  && \text{if $0\leq h\leq c$,}
\\  \xymatrix@C=3pc{
    V \cdot \bigl(
      S_j^\sharp & S_i[h] \ar[l]_{h-c} \bigr)}
      && \text{if $c+1\leq h\leq N-2$,}
\end{eqnarray*}
as required.

On the other hand, in the case $c=m-1$, we have
\begin{gather*}
    \dim\Ext^1(S_i,S_j)=V, \quad \Ext^1(S_j,S_i)=0,\\
    \Ext^1(S_i[k],S_j)=0=\Ext^1(S_j,S_i[k]),\quad k=2,\ldots,m-2.
\end{gather*}
As before we deduce that
the full sub-quiver between $T_j, T_i$ in $\Qext{\h}$ is
\[
    \xymatrix@C=3pc{
    V \cdot \bigl(
      S_j & S_i[h] \ar[l]_{h+1} \bigr),}
\]
as required.

Thus we have proved that \eqref{eq:QCY1} holds
and, in the process, seen that $\h$ is strongly monochromatic.
By Lemma~\ref{lem:parallel}, we can shift any heart in $\EGp(Q)$ into
$\EGp_N(Q, \nzero)$ for some $N\gg1$, which implies
that any heart in $\EGp(Q)$ is strongly monochromatic.
\end{proof}

Now that we know that every heart in $\EGp(Q)$ is strongly monochromatic,
a more careful analysis will show that \eqref{eq:QCY1} also holds for $m=N-1$.
To this end, we write
\[
  \clu=\cluster_{N-1}\colon \D(Q)\to \C{N-1}{Q}
\]
for the quotient functor, as in Section~4, and recall that
\[
\com=\com_{N,N-1} \colon \EGp_N(Q,\nzero) \rightarrow \CEG{N-1}{Q},
\]
as defined in \eqref{eq:inj}.
The key observation is the following.

\begin{lemma}\label{lem:quiver}
Let $\h\in\EGp_N(Q,\nzero)$ with simples $S,S'$ and
corresponding projectives $P,P'$.
If $\tilt{\h}{\sharp}{S}\in\EGp_N(Q,\nzero)$, then
$\dim\Ext^1(S',S)=\dim\Irr(\clu(P), \clu(P'))$.
\end{lemma}

\newcommand{\appr}{\gamma}  

\begin{proof}
By \eqref{eq:irr=ext} in Proposition~\ref{pp:irr=ext},
this is equivalent to showing that
\begin{equation}\label{eq:dimirr}
    \dim\Irr(P,P')=\dim\Irr(\clu(P), \clu(P')).
\end{equation}
Consider any $R\in\Proj\h$.
Writing $\hua{C}=\C{N-1}{Q}$ and $\D=\D(Q)$, recall that
\[
  \Hom_{\hua{C}}(\clu(P), \clu(R)) = \bigoplus_{t\in\ZZ} \Hom_{\D}(P,\shift{}^t R),
\]
where $\shift{}=\shift{N-1}$ is the cluster shift (Definition~\ref{def:cluster}).
If $t<0$, then, as in calculation~\eqref{eq:leminjcalc1} for Lemma~\ref{lem:inj},
we have $\shift{}^{-t} P\in \tau^{-1} \hua{P}_Q [N-2]$ and
$R\in \tau^{-1} \hua{P}^\perp_Q [N-2]$, and so $\Hom_{\D}(P,\shift{}^t R)=0$.
Thus the sum is only over $t\geq0$.

We claim that all maps in $\Irr(\clu(P),\clu(R))$ are induced from
$\Hom_{\D}(P, R)$.
Note that, when $R=P$, the no-loop condition in \cite[Sec.~2]{BT}
implies that $\Irr(\clu(P),\clu(P))=0$ and so the claim is trivial.
Assume then that $R\neq P$.

For $t>0$, let $P^\sharp$ be the new projective of $\tilt{\h}{\sharp}{S}$
that replaces $P$, that is, $P^\sharp$ is defined by the triangle
\begin{equation}\label{eq:defPsh}
    P\xrightarrow{\appr}
    \bigoplus_{X\in\Proj\h} \Irr(P,X)^*\otimes X \to P^\sharp \to P[1].
\end{equation}
Applying $\Hom_{\D}(-,\shift{}^{t} R)$ to this triangle 
gives the exact sequence
\begin{gather}\label{eq:last}
   \bigoplus_{X\in\Proj\h}  \Irr(P,X)\otimes \Hom_{\D}(X, \shift{}^{t} R)
    \xrightarrow{\appr^*} \Hom_{\D}(P, \shift{}^{t} R)\to \Hom_{\D}(P^\sharp,\shift{}^{t} R[1]).
\end{gather}
Since $\tilt{\h}{\sharp}{S}\in\EGp_N(Q,\nzero)$,
the t-structure $\tilt{\hua{P}}{\sharp}{S}$
corresponding to $\tilt{\h}{\sharp}{S}$
satisfies
$\tilt{\hua{P}}{\sharp}{S}\subset\hua{P}_Q$ and
$(\tilt{\hua{P}}{\sharp}{S})^\perp\subset\hua{P}^\perp_Q[N-2]$.
Then,
\[
    P^\sharp\in  \tau^{-1}(\tilt{\hua{P}}{\sharp}{S})^\perp\subset\tau^{-1}\hua{P}_Q^\perp[N-2],
\]
and, since $R\neq P$, we have
$R\in \Proj\tilt{\h}{\sharp}{S}\subset \tilt{\hua{P}}{\sharp}{S}$ and so
\[
    \shift{}^{t} R[1]=\tau^{-1} \shift{}^{t-1} R[N-1]\in\tau^{-1}
    \shift{}^{t-1} \tilt{\hua{P}}{\sharp}{S}[N-1]
    \subset \tau^{-1} \hua{P}_Q[N-1],
\]
which implies the last term in \eqref{eq:last} is zero, because $\h_Q$ is hereditary.
Hence no map in $\Irr(\clu(P),\clu(R))$ is induced from
$\Hom_{\D}(P,\shift{}^t R)$ for $t>0$.
Thus the claim holds.

Now, there may still be more maps in $\hua{C}$ than in $\D$,
which implies that
\begin{equation}\label{eq:VR}
  \Irr(\clu(P), \clu(R))=\clu(V_R),
\end{equation}
for some subspace $V_R$ of $\Irr(P,R)$.
Since $\tilt{\h}{\sharp}{S}\in\EGp_N(Q,\nzero)$,
Theorem~\ref{thm:comparison} implies that
$\com(\tilt{\h}{\sharp}{S})$ is the forward mutation at $\clu(P)$ of $\com(\h)$,
so that $\clu(P^\sharp)$ replaces $\clu(P)$.
Then the triangle defining $\clu(P^\sharp)$ in $\C{N-1}{Q}$ is
the image under $\clu$ of the triangle in $\D(Q)$
\[
  P\to \bigoplus_{X\in\Proj\h} V_X^*\otimes X \to P^\sharp \to P[1].
\]
By comparing this to the triangle \eqref{eq:defPsh},
we deduce that $V_X=\Irr(P,X)$,
because $\Proj\h$ is a basis for the Grothendieck group of $\D(Q)$
(by Theorem~\ref{ppthm:psp}).
In particular, setting $X=P'$ gives \eqref{eq:dimirr},
which completes the proof.
\end{proof}

\begin{remark}\label{rem:quiver}
In the setting of Lemma~\ref{lem:quiver},
if further $\Ext^1(S',S)$ and $\Irr(\clu(P), \clu(P'))$ are non-zero,
then, because the Ext quiver $\Qext{\h}$ is strongly monochromatic and
the augmented graded quiver $\Qaug{\com(\h)}$ is monochromatic and skew-symmetric,
the full sub-quiver between $S$ and $S'$
in the \CY{N} double $\double{\Qext{\h}}{N}$
is equal to the sub-quiver between $\clu(P)$ and $\clu(P')$ in $\Qaug{ \com(\h) }$.
\end{remark}

\begin{theorem}\label{thm:mono1}
For any heart $\h\in\EGp_N(Q, \nzero)$, we have
\begin{equation}\label{eq:QCY1+}
    \Qaug{ \com(\h) }=\double{\Qext{\h}}{N}.
\end{equation}
\end{theorem}

\begin{proof}
As before, i.e. in the proof of Proposition~\ref{pp:mono1},
the loops match and we only need to check that \eqref{eq:QCY1+} holds
between two simples $T_i,T_j\in\h$.
Also as before, we choose a maximal line segment \eqref{eq:maxseg},
with $m=N-1$, so that $\h=\h_h$, for some $0\leq h\leq N-2$,
and $\h_0$ has simples $S_i,S_j$ and projectives $P_i,P_j$.

Thus it is sufficient to show that the full sub-quiver between
$\clu(P_j)$ and $\clu(P_i^k)$ in $\Qaug{ \com(\h_k) }$
is equal to
the full sub-quiver between the corresponding simples in $\double{\Qext{\h_k}}{N}$,
for $0\leq k\leq N-2$.
Using the mutation rule for coloured quivers
(cf. \eqref{eq:quiver 1} and \eqref{eq:quiver 2})
and the change of simples formulae in Proposition~\ref{pp:fini},
once we know that the equality holds for one $\h_k$,
a direct calculation will show that it holds for all $\h_k$.

If there are no arrows in both the sub-quivers of
$\Qaug{ \com(\h_0) }$ and $\double{\Qext{\h_0}}{N}$, the equality holds for $\h_0$.
Suppose that there are arrows between $S_i$ and $S_j$ in $\double{\Qext{\h_0}}{N}$.
Since $\h_0$ is strongly monochromatic, there are three cases.
\StartEnum
\item $\Hom^\bullet(S_i, S_j)=\Hom^c(S_i, S_j)\neq0$ and $c=1$, by \eqref{eq:1}.
    Then applying Remark~\ref{rem:quiver} to $\h_0$ with simples $S=S_j$ and $S'=S_i$
    gives the equality for $\h_0$.
\item $\Hom^\bullet(S_j, S_i)=\Hom^{c+1}(S_j, S_i)\neq0$ for some $0\leq c\leq N-3$.
    Then applying Remark~\ref{rem:quiver} to $\h_c$ with simples $S=S_i[c]$ and $S'=S_j$
    gives the equality for $\h_c$.
\item $\Hom^\bullet(S_j, S_i)=\Hom^{N-1}(S_j, S_i)\neq0$
    and $S_j\in\nzero[N-2]$, because $S_i\in\nzero$.
    By Proposition~\ref{pp:fini},
    we deduce that $S_j,S_i[N-2]\in\Sim\h_{N-2}$.
    Since $S_j\not\in\nzero$, we have $\tilt{(\h_{N-2})}{\flat}{S_j}\in\EGp_N(Q,\nzero)$
    by Lemma~\ref{lem:parallel}.
    Then applying Remark~\ref{rem:quiver} to
    $\tilt{(\h_{N-2})}{\flat}{S_j}$ with simples $S=S_j[-1]$ and $S'=\tilt{\psi}{\flat}{S_j}(S_i[N-2])$
    gives the equality for $\tilt{(\h_{N-2})}{\flat}{S_j}$
    and hence for $\h_{N-2}$.
\StopEnum
On the other hand, suppose that there are arrows between
$\clu(P_j)$ and $\clu(P_i)$ in $\Qaug{ \com(\h_0) }$.
Similar to above, there are three cases and we obtain the required equality for an appropriate $\h_k$,
by applying Remark~\ref{rem:quiver}.
\end{proof}

Theorem~\ref{thm:mono1} will be a key step in the interpretation,
as provided by Theorem~\ref{thm:quiver}, of Buan-Thomas' coloured quiver,
or rather its augmented graded quiver,
in terms of hearts in the category $\D(\qq{N})$.

\section{Calabi-Yau categories} 

We now bring in the Calabi-Yau category $\D(\qq{N})$,
whose relationship with the derived category $\D(Q)$ and the cluster
category $\C{N-1}{Q}$ is the main focus of the paper.

\subsection{Ginzburg algebras}

Let $N\geq 2$ be an integer and $Q$ an acyclic quiver.
The \emph{\CY{N} Ginzburg (dg) algebra} $\qq{N}$  associated to $Q$
is constructed as follows (\cite[Sec.~7.2]{K10}):
\begin{itemize}
\item   Let $Q^N$ be the graded quiver whose vertex set is $Q_0$
and whose arrows are: the arrows in $Q$ with degree $0$;
an arrow $a^*:j\to i$ with degree $2-N$ for each arrow $a:i\to j$ in $Q$;
a loop $e^*:i\to i$ with degree $1-N$ for each vertex $e$ in $Q$.
\item   The underlying graded algebra of $\qq{N}$ is the completion of
the graded path algebra $\k Q^N$ in the category of graded vector spaces
with respect to the ideal generated by the arrows of $Q^N$.
\item   The differential of $\qq{N}$ is the unique continuous linear endomorphism homogeneous
of degree $1$ which satisfies the Leibniz rule and
takes the following values on the arrows of $Q^N$
\[
    \diff \sum_{e\in Q_0} e^*=\sum_{a\in Q_1} \, [a,a^*] .
\]
\end{itemize}
Write $\D(\qq{N})$ for $\hua{D}_{fd}(\mod  \qq{N})$,
the \emph{finite dimensional derived category} of $\qq{N}$
(cf. \cite[Sec.~7.3]{K10}).

Recall that a triangulated category $\hua{D}$ is called \emph{\CY{N}}
if, for any objects $L,M$ in $\hua{D}$ we have a natural isomorphism
\begin{equation}\label{eq:serre}
    \mathfrak{S}:\Hom_{\hua{D}}^{\bullet}(L,M)
        \xrightarrow{\sim}\Hom_{\hua{D}}^{\bullet}(M,L)^\vee[N].
\end{equation}
An object $S$ in a \CY{N} category is \emph{spherical} if
$\Hom^{\bullet}(S, S)=\k \oplus \k[-N]$.
Note that, as $N\geq 2$, any spherical object is rigid.

By \cite{K4} (see also \cite{KS},\cite{ST}), we know that
$\D(\qq{N})$ is a \CY{N} category,
which admits a canonical heart $\zero$ generated
by simple $\qq{N}$-modules $S_e$, for $e\in Q_0$,
each of which is spherical, because the quiver $Q$ has no loops.
We denote by $\EGp(\qq{N})$ the principal component of the exchange graph
$\EG(\qq{N})=\EG(\D(\qq{N}))$, that is, the component containing $\zero$.

\subsection{Twist functors and braid groups}

Recall (\cite{KS},\cite{ST}) that one may
associate to any spherical object $S$ in a \CY{N} category $\D$
a certain auto-equivalence, namely,
the \emph{twist functor} $\phi_S\colon \D\to\D$,
defined by
\begin{equation}\label{eq:sphtwist+}
    \phi_S(X)=\Cone\left(S\otimes\Hom^\bullet(S,X)\to X\right)
\end{equation}
with inverse
\begin{equation}\label{eq:sphtwist-}
    \phi_S^{-1}(X)=\Cone\left(X\to S\otimes\Hom^\bullet(X,S)^\vee \right)[-1].
\end{equation}
Note that the graded dual of a graded $\k$-vector space
$V=\oplus_{i\in\ZZ} V_i[i]$
is
\[V^\vee=\bigoplus_{i\in\ZZ} V_i^*[-i].\]
where $V_i$ is an ungraded $\k$-vector space and $V_i^*$ is its usual dual.

The \emph{Seidel-Thomas braid group} $\Br(\qq{N})$
is the subgroup of $\Aut\D(\qq{N})$ generated by
the twist functors of the simples in $\Sim\zero$.
As shown in \cite[Sec.~2c]{ST},
if $S_1, S_2$ are spherical, then so is $S=\phi_{S_2}(S_1)$
and we have
\begin{equation}\label{eq:ST}
    \phi_S=\phi_{S_2}\circ\phi_{S_1}\circ\phi_{S_2}^{-1}.
\end{equation}
Furthermore,
the generators $\{ \phi_S \mid S\in\Sim\zero\}$
satisfy the braid relations and so
$\Br(\qq{N})$ is a quotient of
the usual braid group $\Br_Q$ associated to the quiver $Q$.
It is not known, except when $Q$ has type $A_n$,
whether this quotient is an isomorphism.

\begin{remark}\label{rem:mono}
We will call a heart \emph{spherical}, if all its simples are spherical.
Observe that, if a finite heart $\h$ is
also monochromatic (Definition~\ref{def:mono}) and spherical,
then the change of simples formulae \eqref{eq:psi+} and \eqref{eq:psi-}
in Proposition~\ref{pp:fini}
can be expressed in terms of the spherical twists, as follows.
\begin{gather}
\tilt{\psi}{\sharp}{S}(X)=
\begin{cases}
\phi_S^{-1}(X) &  \text{if $\Ext^1(X,S)\neq 0$}\\
X & \text{if $\Ext^1(X,S)= 0$}
\end{cases}
\label{eq:psp1}\\
\tilt{\psi}{\flat}{S}(X)=
\begin{cases}
\phi_S(X) &  \text{if $\Ext^1(S,X)\neq 0$}\\
X & \text{if $\Ext^1(S,X)= 0$}
\end{cases}
\label{eq:psp2}
\end{gather}
\end{remark}

\subsection{Lagrangian immersions}

We now introduce the main tool for relating hearts in $\D(Q)$ and $\D(\qq{N})$.

\begin{definition}\label{def:L-imm}
An exact functor $\imm:\D(Q) \to \D(\qq{N})$
is called a \emph{Lagrangian immersion} (L-immersion) if
for any pair of objects $\widehat{S}$, $\widehat{X}$ in $\D(Q)$
there is a short exact sequence
\begin{equation}\label{eq:lagrangian}
    0 \to  \Hom^{\bullet}(\widehat{X},\widehat{S})
    \xrightarrow{\imm}
    \Hom^{\bullet}(\imm(\widehat{X}),\imm(\widehat{S}))
    \xrightarrow{\imm^\dagger}
    \Hom^{\bullet}(\widehat{S},\widehat{X})^\vee[N] \to 0,
\end{equation}
where $\imm^\dagger=\imm^\vee[N]\circ\mathfrak{S}$
is the following composition
\[
    \Hom^{\bullet}(\imm(\widehat{X}),\imm(\widehat{S}))
    \xrightarrow{\mathfrak{S}}
    \Hom^{\bullet}(\imm(\widehat{S}),\imm(\widehat{X}))^\vee[N]
    \xrightarrow{\imm^\vee[N]}
    \Hom^{\bullet}(\widehat{S},\widehat{X})^\vee[N],
\]
in which $\mathfrak{S}$ is the Serre duality isomorphism of \eqref{eq:serre}.
\end{definition}

\begin{definition} \label{def:induced}
Let $\h$ be a finite heart in $\D(\qq{N})$ with $\Sim\h=\{ S_1,\ldots,S_n \}$.
If there is a L-immersion $\imm:\D(Q) \to \D(\qq{N})$
and a finite heart $\nh\in\EGp(Q)$ with
$\Sim\nh=\{ \widehat{S}_1,\ldots,\widehat{S}_n \}$,
such that $\imm(\widehat{S}_i)=S_i$,
then we say that $\h$ is \emph{induced} via $\imm$ from $\nh$
and write $\imm_*(\nh)=\h$.
\end{definition}

In particular, an L-immersion ensures that the Ext quivers of the two hearts
are related precisely by the doubling of Definition~\ref{def:extquiv}.

\begin{proposition} \label{pp:mono2}
Let $\h=\imm_*(\nh)$ be a heart in $\EG(\qq{N})$ induced by a heart
$\nh$ in $\EGp(Q)$, via an L-immersion $\imm\colon \D(Q) \to \D(\qq{N})$.
Then $\h$ is monochromatic and
\begin{equation}\label{eq:QCY2}
    \Qext{\h}=\double{\Qext{\nh}}{N}.
\end{equation}
\end{proposition}
\begin{proof}
The fact that $\h$ is monochromatic follows from  \eqref{eq:lagrangian}
and the fact that $\nh$ is strongly monochromatic (Proposition~\ref{pp:mono1}).
Then \eqref{eq:QCY2} also follows directly from \eqref{eq:lagrangian}.
\end{proof}

Under appropriate conditions, tilting preserves the property of being induced.

\begin{proposition}\label{pp:comm}
Let $\h=\imm_*(\nh)$ be a heart in $\EG(\qq{N})$ induced by a heart
$\nh$ in $\EG(Q)$.
If $\nh$ is rigid, then $\h$ is spherical.
Moreover, suppose $S=\imm(\widehat{S})$ is a simple in $\h$,
induced by $\widehat{S}\in\Sim\nh$. Then
\StartEnum
\item
    if $\Hom^{N-1}(\widehat{S},\widehat{X})=0$ for any $\widehat{X}\in\Sim\nh$,
    then $\tilt{\h}{\sharp}{S}= \imm_*( \tilt{\nh}{\sharp}{\widehat{S}} )$.
\item
    if $\Hom^{N-1}(\widehat{X},\widehat{S})=0$ for any $\widehat{X}\in\Sim\nh$,
    then $\tilt{\h}{\flat}{S}= \imm_*( \tilt{\nh}{\flat}{\widehat{S}} )$.
\StopEnum
\end{proposition}

\begin{proof}
Since $\D(Q)$ is hereditary,
any rigid object $\widehat{M}\in \D(Q)$ is exceptional
and hence, by \eqref{eq:lagrangian}, $\imm(\widehat{M})$ is spherical.

For any $\widehat{X}(\ncong\widehat{S})$ in $\Sim\nh$,
let $X=\imm(\widehat{X})$.
Since $\Hom^{N-1}(\widehat{S},\widehat{X})=0$,
the short exact sequence \eqref{eq:lagrangian} gives an isomorphism
$
    \imm:\Hom^1(\widehat{X},\widehat{S})
        \xrightarrow{\sim} \Hom^1(X,S).
$
Since $\imm$ is exact, we have
\[
    \imm(\tilt{\psi}{\sharp}{\widehat{S}}(\widehat{X}))
    =\tilt{\psi}{\sharp}{S}(X),
\]
where $\tilt{\psi}{\sharp}{}$ is defined as in \eqref{eq:psi+}.
Then, by Proposition~\ref{pp:fini}, we have
$\imm_*( \tilt{\nh}{\sharp}{\widehat{S}} )= \tilt{\h}{\sharp}{S}$.
Similarly for $2^\circ$.
\end{proof}

\section{Main results}

In this section, we return to assuming that $N\geq 3$.

\subsection{Inducing hearts}

The natural quotient morphism $\qq{N} \to Q$ induces a functor
\begin{equation}\label{eq:hua I}
    \hua{I}:\D(Q) \to \D(\qq{N}).
\end{equation}
For more general dg algebras,
this functor was considered by Keller,
who showed (\cite[Lem.~4.4 (b)]{K8}) that
$\hua{I}$ is an L-immersion (and indeed that \eqref{eq:lagrangian} has a natural splitting).

Consider the subgraph $\EGp_N(\qq{N},\zero)$ in $\EGp(\qq{N})$
with canonical heart $\zero$ as base
(Definition~\ref{def:eg}).
Observe that $\hua{I}$ sends the simples in $\nzero$
to the corresponding simples in $\zero$ and hence we have
$\hua{I}_*(\nzero)=\zero$.

\begin{theorem}\label{thm:inducing}
Any heart in $\EGp_N(Q,\nzero)$ induces a heart in $\EGp_N(\qq{N},\zero)$
via the natural L-immersion $\hua{I}$ in \eqref{eq:hua I},
i.e. we have a well-defined map
\begin{equation}\label{eq:I1}
    \hua{I}_*\colon \EGp_N(Q,\nzero) \to \EGp_N(\qq{N},\zero).
\end{equation}
Moreover, it is an isomorphism between oriented graphs
and can be extended to an isomorphism
$\cc{\hua{I}_*}:\cc{\EGp_N}(Q,\nzero) \to \cc{\EGp_N}(\qq{N},\zero)$.
\end{theorem}

\begin{proof}
We prove that $\hua{I}_*$ is well-defined
by induction starting from $\hua{I}_*(\nzero)=\zero$.
Thus, if $\hua{I}_*(\nh)=\h$, for some
$\nh, \tilt{\nh}{\sharp}{\widehat{S}}\in\EG(Q,\nzero)$,
$\widehat{S}\in\Sim\nh$ and $\h\in\EGp_N(\qq{N},\zero)$,
then we need to show that
$\tilt{\nh}{\sharp}{\widehat{S}}$ induces a heart in $\EGp_N(\qq{N},\zero)$.

For any $\widehat{X}\in\Sim\nh$,
by the first part of Lemma~\ref{lem:parallel} and \eqref{eq:Happ}, we know that
$\widehat{X}\in\Ind\nzero[m]$ for some $0\leq m(\widehat{X})\leq N-2$.
By the second part of Lemma~\ref{lem:parallel}, we have $\Ho{N-2}(\widehat{S})=0$
which implies $m(\widehat{S})< N-2$,
where the homology $\Ho{\bullet}$ is with respect to $\nzero$.
Then $\Hom^{N-1}(\widehat{S},\widehat{X})=0$, since $\nzero$ is hereditary.
By Proposition~\ref{pp:comm} we have
$\imm_*( \tilt{\nh}{\sharp}{\widehat{S}} )= \tilt{\h}{\sharp}{S}$.
Since $S=\imm(\widehat{S})\in\zero[m(\widehat{S})]$ for some $m(\widehat{S})< N-2$,
by the second part of Lemma~\ref{lem:parallel}, we know that
$\tilt{\h}{\sharp}{S}$ is in $\EGp_N(\qq{N},\zero)$.

The injectivity of $\hua{I}_*$ follows from the facts that
a heart is determined by its simples and $\hua{I}$ is injective.

For surjectivity of $\hua{I}_*$, we consider the line segments.
By the first part of Lemma~\ref{lem:parallel},
any line segment in $\EGp_N(\qq{N}, \zero)$ has length less or equal than $N-1$.
Notice that,
by Proposition~\ref{pp:parallel}, any maximal line segment
in $\EGp_N(Q, \nzero)$ has length $N-1$,
and hence its image under $\hua{I}_*$ is a maximal line segment in $\EGp_N(\qq{N}, \zero)$.
This implies that,
if a heart $\h$ in $\EGp_N(\qq{N}, \zero)$ is induced from some heart
$\nh\in\EGp_N(Q, \nzero)$ via $\hua{I}$,
then the maximal line segment $l(\h,S)\cap\EGp_N(\qq{N}, \zero)$
is induced from the line segment $l(\nh,\widehat{S})\cap\EGp_N(Q, \nzero)$ via $\hua{I}$,
where $S\in\Sim\h$, and $\widehat{S}\in\Sim\nh$ such that $\hua{I}(\widehat{S})=S$.
Hence any simple tilt of an induced heart via $\hua{I}$ is also induced via $\hua{I}$,
provided this tilt is still in $\EGp_N(\qq{N}, \zero)$.
Thus, inductively, we deduce that $\hua{I}_*$ is surjective.

The last assertion follows from the facts that
we can cyclically complete $\EGp_N(Q,\nzero)$ (Proposition~\ref{pp:parallel})
and $\hua{I}_*$ preserves the structure of line segments.
\end{proof}

\begin{remark}\label{rem:induced}
As Theorem~\ref{thm:inducing} tells us that
every heart $\h\in \EGp_N(\qq{N},\zero)$ is induced
and hence finite, we also deduce that $\h$ is monochromatic,
by Proposition~\ref{pp:mono2}, and spherical, by
Proposition~\ref{pp:comm}, because every heart in $\EGp(Q)$
is rigid, by Theorem~\ref{ppthm:psp}.

We can also apply $\hua{I}_*$ to the sequence of tilts in Corollary~\ref{cor:tilt-shift}
and deduce, as there, that $\zero[k]\in\EGp_N(\qq{N}, \zero)$, for $0\leq k\leq N-2$, and
 $\zero[k]\in\EGp(\qq{N})$, for $k\in\ZZ$.
 \end{remark}

\begin{proposition}\label{pp:br inj}
$\Br(\qq{N}) \cdot \EGp_N(\qq{N},\zero)=\EGp(\qq{N})$.
\end{proposition}
\begin{proof}
We use induction starting from Theorem~\ref{thm:inducing}.
Suppose $\h'\in\EGp(\qq{N})$, with $\h'=\phi(\h)$ for $\phi\in\Br(\qq{N})$
and $\h\in\EGp_N(\qq{N}, \zero)$.
For any simple $S'$ in $\h'$ we have
$(\h')^{\sharp}_{S'}=\phi ( \tilt{\h}{\sharp}{S} )$,
for $S=\phi^{-1}(S')$.
Thus it is sufficient to prove that $\tilt{\h}{\sharp}{S}$
is in $\Br(\qq{N}) \cdot \EGp_N(\qq{N},\zero)$,
for any simple $S$ in $\h$.

If $\tilt{\h}{\sharp}{S}$ is still in $\EGp_N(\qq{N}, \zero)$,
then there is nothing to prove, so suppose that
$\tilt{\h}{\sharp}{S}\notin\EGp_N(\qq{N}, \zero)$.
As in the proof of Theorem~\ref{thm:inducing},
the maximal line segment
\[
  l(\h,S) \cap\EGp_N(\qq{N}, \zero)
  = \{\tilt{\h}{m\flat}{S}\}_{m=0}^{N-2}
\]
is induced from the maximal line segment
\[
  l(\nh,\widehat{S})\cap\EGp_N(Q,\nzero)
 = \{\tilt{\nh}{m\flat}{\widehat{S}}\}_{m=0}^{N-2},
\]
where $\h=\hua{I}_*(\nh)$ and $S= \hua{I}(\widehat{S})$.
Let $\h^-=\tilt{\h}{(N-2)\flat}{S}$ and $S^-=S[2-N]$.
By Remark~\ref{rem:induced}, each $\tilt{\h}{m\flat}{S}$ is finite, spherical
and monochromatic.
Applying Proposition~\ref{pp:fini}, with \eqref{eq:psi+} replaced by
\eqref{eq:psp1}, to the simple forward tilt of $\tilt{\h}{m\flat}{S}$
with respect to $S[-m]$, for $m=N-2,N-1,\ldots,0$,
we deduce that the changes of simples from $\h^-$ to $\tilt{\h}{\sharp}{S}$
are as follows:
\begin{itemize}
\item
    $S^-\in\Sim\h^-$ becomes $S[1]\in \Sim \tilt{\h}{\sharp}{S}$
    which equals $\phi^{-1}_S(S^-)$;
\item
    if $X\in\Sim\h^-$, with $X\neq S^-$ and $\Hom^\bullet(X,S)=0$,
    then it remains in $\Sim \tilt{\h}{\sharp}{S}$,
    but then also $X=\phi^{-1}_S(X)$.
\item
    if $X\in\Sim\h^-$, with $X\neq S^-$ and $\Hom^\bullet(X,S)\neq0$,
    then, since $\h^-$ is monochromatic, $\Hom^\bullet(X,S)=\Hom^k(X,S^-)$,
    for some integer $k>0$.
    Since $\h^-$ is induced from a heart in $\EGp_N(Q,\nzero)$,
    \eqref{eq:lagrangian}
    and \eqref{eq:hom-bound} imply that $1\leq k\leq N-1$.
    Then $X$ remains in $\Sim \tilt{\h}{m\flat}{S}$ for
    $m=N-2,\ldots,N-1-k$,
    then becomes and remains $\phi^{-1}_S(X)$
    in $\Sim \tilt{\h}{m\flat}{S}$ for $m=N-k,\ldots,-1$.
\end{itemize}
Thus $\Sim \tilt{\h}{\sharp}{S} = \phi^{-1}_{S} \bigl(\Sim \h^-\bigr)$ and so,
as the simples determine the heart, we conclude that
$\tilt{\h}{\sharp}{S}=\phi^{-1}_{S} (\h^-)$,
as required.
\end{proof}

\begin{corollary}\label{cor:all-induced}
Every heart in $\EGp(\qq{N})$ is induced (Definition~\ref{def:induced}) and
hence finite, spherical and monochromatic.
Moreover, for any heart $\h$ in $\EGp(\qq{N})$,
the set of twist functors of its simples is a set of generators of $\Br(\qq{N})$.
Further, for any $S\in\Sim\h$, we have
    \begin{equation}\label{eq:root}
        \tilt{\h}{ \pm(N-1)\sharp}{S} = \phi_{S}^{\mp1} (\h ).
    \end{equation}
\end{corollary}

\begin{proof}
Proposition~\ref{pp:br inj} shows that every heart is induced
via the L-immersion which is the composition of
the natural L-immersion $\hua{I}$ with some twist functors.
Then every heart is finite, spherical and monochromatic,
as in Remark~\ref{rem:induced}.
Hence Proposition~\ref{pp:fini} applies, with \eqref{eq:psi+} and
\eqref{eq:psi-} replaced by
\eqref{eq:psp1} and \eqref{eq:psp2},
and so, by \eqref{eq:ST}, the simple twist functors
of two hearts related by a simple tilt generate the same group.
Thus the second assertion follows by induction.

Further, we know that
\eqref{eq:root} is true for any heart $\h^-\in\hua{I}_*(\EGp_N(Q,\nzero))$
with simple $S^-$ as in Proposition~\ref{pp:br inj}.
Hence it is true for any hearts in $l(\h^-,S^-)$,
which implies it is also true for any heart induced via $\hua{I}_*$,
by Proposition~\ref{pp:parallel}.
Notice that the autoequivalences preserve \eqref{eq:root},
thus this equation holds for any heart in $\EGp(\qq{N})$ by Proposition~\ref{pp:br inj}.
\end{proof}

\begin{corollary}\label{cor:signs}
Let $\h$ and $\h'$ be hearts in $\EGp(\qq{N})$ in the same braid group orbit, i.e.
$\phi(\h)=\h'$ for some $\phi\in\Br(\qq{N})$.
Then there exists a sequence of spherical objects
$T_0,\ldots,T_{m-1}$ in hearts $\h_0,\ldots,\h_{m-1}$
(for some integer $m\geq 0$) together with signs
$\epsilon_i\in\{\pm1\}, i=0,\ldots,m-1$,
such that $\h_0=\h$,
\begin{equation}\label{eq:sequence}
    \h_{i+1}=\tilt{(\h_i)}{\epsilon_i(N-1)\sharp}{T_i}, \quad i=0,1,\ldots,m-1,
\end{equation}
and $\h_m=\h'$.
\end{corollary}

\begin{proof}
Fix $\h$ and let $\Sim\h=\{S_1,\ldots,S_n\}$, $\phi_k=\phi_{S_k}$ for $1\leq k\leq n$.
Since $\phi_1,\ldots,\phi_n$ generate $\Br(\qq{N})$ by Corollary~\ref{cor:all-induced},
we have
\[
    \phi=\phi_{t_{m-1}}^{\lambda_{m-1}}\circ\cdots\circ\phi_{t_{0}}^{\lambda_{0}}
\]
for some $t_j\in\{1,\ldots,n\}$ and $\lambda_j\in\{\pm1\}$.
Use induction on $m$.
If $m=0$, i.e. $\h=\h'$, there is nothing to prove.
Suppose the statement holds for $m\leq s$ and consider the case when $m=s+1$.
Write $\varphi=\phi_{t_s}^{\lambda_s}$.
By the inductive hypothesis, given hearts $\h$ and
\[
    \varphi^{-1}(\h')=
    \left(  \phi_{t_{s-1}}^{\lambda_{s-1}}
        \circ\cdots\circ\phi_{t_{0}}^{\lambda_{0}}  \right)(\h),
\]
there are spherical objects $R_0,R_2,\ldots,R_{s-1}$
and signs $\varepsilon_i\in\{\pm1\}$, such that $\h'_0=\h$,
\[
    \h'_{i+1}=\tilt{(\h'_i)}{\varepsilon_i(N-1)\sharp}{R_i},
    \quad i=0,1,\ldots,s-1
\]
and $\h'_{s}=\varphi^{-1}(\h')$.
Let $T_0=S_{t_m}, \epsilon_0=\lambda_m$
and $T_i=\varphi(R_{i-1}), \epsilon_i=\varepsilon_{i-1}$ for $i=1,\ldots,s$.
Then we have $\h_0=\h$, $\h_1=\varphi(\h_0)$ and (inductively)
\[
    \h_{i+1}=\tilt{(\h_i)}{\epsilon_i(N-1)\sharp}{T_i}
    =\tilt{\left(\varphi(\h'_{i-1})\right)}{\varepsilon_{i-1}(N-1)\sharp}
        {\varphi(R_{i-1})}
    =\varphi(\h'_i)
\]
for $i=1,\ldots,s$.
In particular, we have $\h_{s+1}=\varphi(\h'_s)=\h'$ as required.
\end{proof}

\subsection{The circle of identifications}
\newcommand{\taut}{p} 

By Proposition~\ref{pp:br inj}, the tautological map
\[
  \taut\colon \EGp_N(\qq{N}, \zero) \to \EGp(\qq{N})/\Br,
\]
is a surjection on vertices.

By Theorem~\ref{thm:inducing},
$\EGp_N(\qq{N}, \zero)$ carries the same linear structure as $\EGp_N(Q,\nzero)$,
and so, just as we extended $\com$ to $\cc{\com}$ in Theorem~\ref{thm:comparison},
we can also extend $\taut$ to the cyclic completion
\begin{equation}\label{eq:eg}
  \cc{\taut}\colon \cc{\EGp_N}(\qq{N},\zero) \to \EGp(\qq{N})/\Br
\end{equation}
to get an epimorphism of oriented graphs,
i.e. also a surjection on edges.
More precisely, in the notation of Definition~\ref{def:convex},
$\cc{\taut}$ sends the new edge
$e_l\colon \h \to \h^-=\tilt{\h}{(N-2)\flat}{S}$, in each basic cycle $c_l$,
to the edge in $\EGp(\qq{N})/\Br$
induced by $(\h \xrightarrow{S} \tilt{\h}{\sharp}{S})$,
where $c_l$ is induced by the line $l=l(\h,S)$ such that
\[
    l(\h,S)\cap\EGp_N(\qq{N}, \zero)=\{\tilt{\h}{i\flat}{S}\}_{i=0}^{N-2}.
\]
As shown in the proof of Proposition~\ref{pp:br inj},
we have $\tilt{\h}{\sharp}{S}=\phi^{-1}_{S} (\h^-)$,
so that $\cc{\taut}$ is indeed a map of graphs.
The fact that $\cc{\taut}$ is a surjection on edges follows,
as in the proof of Theorem~\ref{thm:comparison},
from the fact that both graphs are $(n,n)$-regular,
\New{i.e. have $n$ incoming and $n$ outgoing edges at each vertex,
where $n=\# Q_0$.}

Our goal is to show that $\cc{\taut}$ is an isomorphism, which in particular
means that $\EGp_N(\qq{N}, \zero)$ is a fundamental domain for the braid
group action on $\EGp(\qq{N})$. Since we already know,
by Theorems~\ref{thm:inducing} and~\ref{thm:comparison},
that
\[
  \cc{\EGp_N}(\qq{N},\zero) \;\cong\;
  \cc{\EGp_N}(Q, \nzero) \;\cong\;
  \CEG{N-1}{Q},
\]
this will enable us to deduce, as promised, that the braid group quotient
of the exchange graph of $\qq{N}$ is the cluster exchange graph, i.e.
\[
  \EGp(\qq{N})/\Br \cong \CEG{N-1}{Q}.
\]

To make the proof and also to see that this isomorphism is natural,
recall that the \CY{N} version of Amiot's construction
\cite[Sec.~2]{Guo1} (cf. \cite[Sec.~2]{A}) gives a quotient functor
$\per(\qq{N}) \to \C{N-1}{Q}$,
where $\per(\qq{N})$ is the perfect derived category of $\qq{N}$.
By \cite[Sec.4]{KY} and \cite[Sec.3]{Guo2},
every heart $\h\in\EGp(\qq{N})$ induces a t-structure on $\per(\qq{N})$
and determines a
silting set $\Proj\h$
in $\per(\qq{N})$, which maps by the quotient functor
to a cluster tilting set in $\C{N-1}{Q}$.
Thus we have a map
\begin{equation}\label{eq:cov}
    \cov:\EGp(\qq{N}) \to \CEG{N-1}{Q}.
\end{equation}
In particular, $\cov$ maps $\zero$ to the initial cluster tilting set
$\CTS_Q$.

\begin{theorem}\label{thm:main}
Let $Q$ be an acyclic quiver.
The map $\cc{\taut}$ in \eqref{eq:eg} gives an isomorphism of
oriented graphs
\begin{equation}\label{eq:eg1}
   \cc{\EGp_N}(\qq{N}, \zero)\cong\EGp(\qq{N})/\Br(\qq{N}).
\end{equation}
Furthermore, the map $\cov$ in \eqref{eq:cov} is a $\Br$-invariant
map of oriented graphs and induces a graph isomorphism
\begin{equation}\label{eq:eg2}
    \EGp(\qq{N})/\Br(\qq{N})    \cong  \CEG{N-1}{Q}.
\end{equation}
\end{theorem}

\begin{proof}
To see that $\cov$ is a graph map,
observe that simple tilting of hearts corresponds to mutation of silting sets,
by the argument of Koenig-Yang \cite[Thm.~7.12]{KY}
(in fact we need the result for homologically smooth non-positive dg algebras
attributed therein to  Keller-Nicol\'{a}s \cite{KN2}).
Then mutation of silting sets corresponds to mutation of cluster tilting sets,
using \cite[Prop.~2.15]{Guo1} (cf. \cite[Prop.~2.9]{A}).

To see that $\cov$ is $\Br$-invariant,
observe that, by Corollary~\ref{cor:signs},
if two hearts $\h,\h'\in\EGp(\qq{N})$ are in the same braid group orbit,
then $\h'$ can be obtained from $\h$ by a sequence of \New{$N-1$} simple tiltings
as in \eqref{eq:sequence}.
Then $\cov(\h)=\cov(\h')$ because repeating the same mutation $N-1$ times
returns every cluster tilting set back to itself.
Hence $\cov$ induces a graph map
\[
  \covtil:\EGp(\qq{N})/\Br\to\CEG{N-1}{Q}.
\]

Noting that the initial points $\nzero$, $\zero$ and $\CTS_Q$ match up
and that simple tilting of hearts in $\EGp_N(Q, \nzero)$ or $\EGp(\qq{N})$
corresponds to mutation of cluster tilting sets,
we obtain the following commutative diagram of graph maps
\begin{equation}\label{eq:diagram}
\xymatrix@C=3pc{
    \cc{\EGp_N}(Q, \nzero) \ar[d]^{\cc{\hua{I}_*}}
        \ar[r]^{\cc{\com}}
        & \CEG{N-1}{Q}\\
    \cc{\EGp_N}(\qq{N}, \zero) \ar@{->}[r]^{\cc{\taut}}
        & \EGp(\qq{N})/\Br \ar@{->}[u]^{\covtil}
}\end{equation}
The fact that $\cc{\com}$ and $\cc{\hua{I}_*}$ are isomorphisms
(Theorem~\ref{thm:comparison} and Theorem~\ref{thm:inducing})
and $\cc{\taut}$ is
an epimorphism
implies that both $\cc{\taut}$ and $\covtil$ are
isomorphisms, as required.
\end{proof}

\begin{remark}
We need the canonical heart as base on the left-hand-side
to ensure the isomorphism \eqref{eq:eg1} holds.
Example~\ref{ex:x} illustrates this phenomenon.
However, if $N=3$, the isomorphism \eqref{eq:eg1} holds for any heart
(see Section~\ref{sec:cy3}).
Further, for $N=3$, Keller-Nicol\'{a}s (cf. \cite[Thm.~5.6]{K6})
prove \eqref{eq:eg2} in full generality, that is,
when $Q$ is a loop-free, 2-cycle-free quiver with a polynomial potential $W$.
\end{remark}

\begin{example}\label{ex:x}
Let $Q$ be a quiver of type $A_2$
and $\Sim\zero=\{S,T\}$ with $\Ext^1(S,T)=\k$.
Figure~\ref{fig:2} shows the cyclic completions of two exchange graphs:
$\cc{\EGp_4}(\qq{4}, \zero)$ on the left and
$\cc{\EGp_4}(\qq{4}, \tilt{(\zero)}{\sharp}{S} )$ on the right.
The solid arrows are the edges in $\EGp(\qq{4})$
and the dotted arrows are the extra edges in the cyclic completions.
The vertices $\otimes$ and $\odot$ represent the source and sink
(i.e. $\h$ and $\h[2]$ in fact)
in the exchange graph $\EGp_4(\qq{4},\h)$ with base $\h$.
Notice that $\cc{\EGp_4}(\qq{4},\tilt{(\zero)}{\sharp}{S} )$
cannot be isomorphic to $\EGp(\qq{4})/\Br$,
because it has the wrong number of vertices.

\begin{figure}[t]\centering
\newcommand{\vsource}{\otimes}
\newcommand{\vsink}{\odot}
\newcommand{\vertx}{\bullet}
\begin{tikzpicture}[xscale=1,yscale=.8, 
  arrow/.style={->,>=stealth,thick},
  c-all/.style={black}]
\foreach \j/\x/\y in {2/1/1,3/2/2,4/2/-2,5/2/0,6/3/-1,7/4/2,8/4/0,9/5/1,10/3/-3,11/4/-2}
  \draw[c-all] (\x,\y) node (v\j) {$\vertx$};
\draw[c-all] (0,0) node (v1) {$\vsource$};
\draw[c-all] (6,0) node (v12) {$\vsink$};
\draw[c-all,arrow] (v1) edge (v2) edge (v4) ;
\draw[c-all,arrow] (v2) edge (v3) edge (v5) ;
\draw[c-all,arrow] (v3) edge[dotted, bend right] (v1) edge (v8) ;
\draw[c-all,arrow] (v4) edge (v6) edge (v10) ;
\draw[c-all,arrow] (v5) edge (v6) edge (v7) ;
\draw[c-all,arrow] (v7) edge (v9) edge[dotted] (v2) ;
\draw[c-all,arrow] (v6) edge (v8) edge (v11) ;
\draw[c-all,arrow] (v8) edge (v9) edge[dotted,bend left] (v4) ;
\draw[c-all,arrow] (v9) edge[dotted] (v3) edge (v12) ;
\draw[c-all,arrow] (v10) edge (v11) edge[dotted,bend left] (v1) ;
\draw[c-all,arrow] (v11) edge (v12) edge[dotted,bend left] (v5) ;
\draw[c-all,arrow] (v12) edge[dotted,bend right] (v7) edge[dotted,bend left] (v10) ;
\end{tikzpicture}
\quad
\begin{tikzpicture}[xscale=.65,yscale=.6,
  arrow/.style={->,>=stealth,thick},
  c-all/.style={black}]
\foreach \j/\x/\y in {2/2/2,3/2/-2,4/3/3,5/3/-3,6/4/0,7/5/3,8/5/-3,9/6/2,10/6/-2}
  \draw[c-all] (\x,\y) node (v\j) {$\vertx$};
\draw[c-all] (0,0) node (v1) {$\vsource$};
\draw[c-all] (8,0) node (v11) {$\vsink$};
\path[c-all,arrow] (v1) edge (v2) edge (v3) ;
\draw[c-all,arrow] (v2) edge (v4) edge (v6) ;
\draw[c-all,arrow] (v3) edge (v6) edge (v5) ;
\draw[c-all,arrow] (v4) edge[dotted,bend right=23] (v1) edge (v7) ;
\draw[c-all,arrow] (v5) edge[dotted,bend left=23] (v1) edge (v8) ;
\draw[c-all,arrow] (v7) edge (v9) edge[dotted,bend right] (v4) ;
\draw[c-all,arrow] (v6) edge (v9) edge (v10) ;
\draw[c-all,arrow] (v8) edge (v10) edge[dotted,bend left] (v5) ;
\draw[c-all,arrow] (v9) edge[dotted,bend left=23] (v3) edge (v11) ;
\draw[c-all,arrow] (v10) edge (v11) edge[dotted,bend left=23] (v2) ;
\draw[c-all,arrow] (v11) edge[dotted, bend right=23] (v7) edge[dotted,bend left=23] (v8) ;
\end{tikzpicture}
\caption{Two cyclic completions of CY-4 exchange graphs of type $A_2$.}
\label{fig:2}
\end{figure}
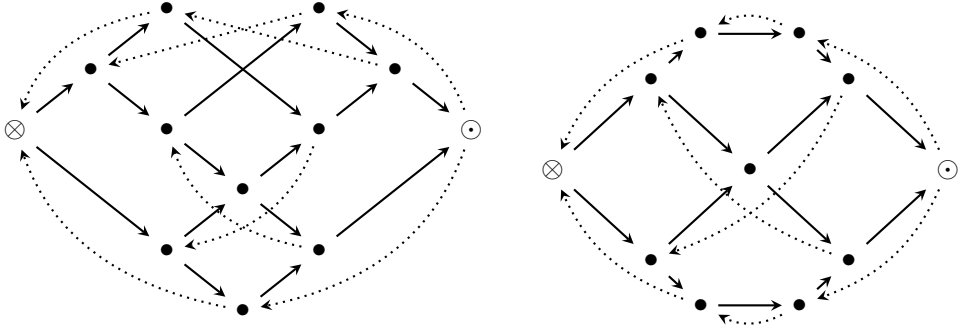
\end{example}

\begin{remark}
By Theorem~\ref{thm:comparison},
each almost complete cluster tilting set in $\C{N-1}{Q}$
can be identified with a basic cycle in
$\CEG{N-1}{Q}\cong\cc{\EGp_N}(Q, \nzero)$,
which can be identified with
a basic cycle in $\cc{\EGp_N}(\qq{N}, \zero)$ by Theorem~\ref{thm:inducing}.
By Theorem~\ref{thm:main},
these basic cycles also can be interpreted as
braid group orbits of lines of $\EGp(\qq{N})$ in $\EGp(\qq{N})/\Br$.
\end{remark}

\subsection{Interpretation of coloured quivers}

As an immediate application of Theorem~\ref{thm:main}, we can use
Theorem~\ref{thm:mono1} to
interpret Buan-Thomas' coloured quiver
(or more precisely its augmented graded quiver)
for cluster tilting sets via hearts in $\D(\qq{N})$.

\begin{theorem}\label{thm:quiver}
For any heart $\h\in\EGp(\qq{N})$, the Ext quiver $\Qext{\h}$ (Definition~\ref{def:extquiv})
is equal to the augmented graded quiver $\Qaug{\cov(\h)}$ (Definition~\ref{def:augquiv})
of the corresponding cluster tilting set $\cov(\h)$,
for $\cov$ as in \eqref{eq:cov}.
\end{theorem}

\begin{proof}
By Theorem~\ref{thm:inducing}, any heart $\h$ in $\EGp_N(\qq{N}, \zero)$
is induced from a heart $\nh$ in $\EGp_N(Q, \nzero)$,
i.e. $\h= \hua{I}_*(\nh)$.
Hence, combining \eqref{eq:QCY2}, \eqref{eq:QCY1+} and
\eqref{eq:diagram}, we see that
\[
 \Qext{\h}=\double{\Q{\nh}}{N}=\Qaug{\com(\nh)}=\Qaug{\cov(\h)}.
\]
But Theorem~\ref{thm:main} also tells us that $\EGp_N(\qq{N}, \zero)$
is a fundamental domain for the action (by automorphisms) of $\Br(\qq{N})$
and that $\cov$ is invariant under this action.
Hence, we deduce that the equality holds for all $\h\in\EGp(\qq{N})$.
\end{proof}

\section{Orientations of cluster exchange graphs} \label{sec:cy3}

We now consider just the case $N=3$.
Recall that, by Theorem~\ref{thm:main}, we have the following three
descriptions of the same oriented graph
\[
  \cc{\EGp_3}(\qq{3}, \zero)\cong \EGp(\qq{3})/\Br \cong \CEG{2}{Q}.
\]
In this graph, every basic cycle is a $2$-cycle and
hence we have an induced isomorphism of the oriented
graph $\EGp(\qq{3}, \zero)$ with the
usual unoriented cluster exchange graph
$\CEGun{Q}$, that is, the graph obtained from $\CEG{2}{Q}$
by replacing each basic $2$-cycle with an unoriented edge.
For example, for $Q$ of type $A_2$, Figure~\ref{fig:4} shows
$\EGp_3(\qq{3}, \zero)$ cyclically completed on the left
and the corresponding unoriented cluster exchange graph $\CEGun{Q}$
on the right.
\begin{figure}\centering
\newcommand{\vsource}{\otimes}
\newcommand{\vsink}{\odot}
\newcommand{\vertx}{\bullet}
\begin{tikzpicture}[scale=.7,
  arrow/.style={->,>=stealth,thick},
  c-all/.style={black}]
\foreach \j in {1,3,4}
   \draw[c-all] (72*\j:2cm) node (t\j) {$\vertx$};
\draw[c-all] (2*72:2cm) node (t2) {$\vsource$};
\draw[c-all] (5*72:2cm) node (t5) {$\vsink$};
\foreach \a/\b in {2/1,1/5,2/3,3/4,4/5}
  \draw[c-all] (t\a) edge[arrow]  (t\b);
\foreach \a/\b in {2/3,3/4,4/5}
  \draw[c-all] (t\b) edge[arrow,dotted,bend left]  (t\a);
\foreach \a/\b in {2/1,1/5}
  \draw[c-all] (t\b) edge[arrow,dotted,bend right]  (t\a);
\end{tikzpicture}
\qquad\quad
\begin{tikzpicture}[scale=.7,
 c-all/.style={black}]
\foreach \j in {1,...,5}
{   \draw[c-all] (\j*72:2cm) node (t\j) {$\vertx$};
    \draw[c-all] (\j*72+72:2cm) node (h\j) {$\vertx$} ;}
\foreach \j in {1,...,5}
{   \draw[c-all] (t\j) edge[thick] (h\j);  }
\end{tikzpicture}
\caption{$\cc{\EGp_3}(\qq{3}, \zero)$ and $\CEGun{Q}$ for a quiver $Q$ of type $A_2$.}
\label{fig:4}
\end{figure}
Thus $\EGp(\qq{3},\zero)$ is an oriented version of $\CEGun{Q}$.
In this section, we will see that the same holds for $\EGp(\qq{3},\h)$,
for any heart $\h\in\EGp(\qq{3})$, although the orientations will differ.

To achieve this, it is crucial that $N=3$ and we can use Lemma~\ref{pp:either-or}
rather than just Lemma~\ref{lem:parallel}.
Recall that, by Corollary~\ref{cor:all-induced}, any heart $\h\in\EGp_3(\qq{3})$
is finite and any $S\in\Sim\h$ is spherical and thus rigid.
We can use the dichotomy of Lemma~\ref{pp:either-or},
with its assumption rewritten as $\h\leq \h_0\leq \h[1]$,
to partition the interval $\EG_3(\qq{3}, \h)$ into two pieces,
given a choice of $S\in\Sim\h$:
\begin{eqnarray*}
\EG_3(\qq{3}, \h)_S^- &=& \bigl\{ \h_0\in \EG_3(\qq{3}, \h) \mid S\in \h_0 \bigr\},\\
\EG_3(\qq{3}, \h)_S^+  &=& \bigl\{ \h_0\in \EG_3(\qq{3}, \h) \mid S[1]\in \h_0 \bigr\}.
\end{eqnarray*}
and also write $\EGp_3(\qq{3}, \h)_S^\pm=\EGp_3(\qq{3}, \h) \cap \EG_3(\qq{3}, \h)_S^\pm$.

If $S$ labels an edge $\h_1\to \h_2$ in
$\EG_3(\qq{3},\h)$, then $\h_1\in\EG_3(\qq{3},\h)_S^-$
and $\h_2\in\EG_3(\qq{3},\h)_S^+$.
In fact, the converse is also true.

\begin{lemma}\label{lem:connecting}
Let $\h\in\EG(\qq{3}), S\in\Sim\h$
and $e$ be an edge between $\EG_3(\qq{3},\h)_S^-$ and
$\EG_3(\qq{3},\h)_S^+$.
Then the tail of $e$ is in $\EG_3(\qq{3},\h)_S^-$
and the label of $e$ is $S$.
\end{lemma}

\begin{proof}
Let $\h_1\in\EG_3(\qq{3}, \h)_S^-$ and $\h_2\in\EG_3(\qq{3}, \h)_S^+$
be the vertices of $e$.
We know that one of $\hua{P}_1$, $\hua{P}_2$ is contained in the other.
But, since $S\in\h_1\subset\hua{P}_1$ and $S\in\h_2[-1]\subset\hua{P}_2^\perp$,
we must have $\hua{P}_1\supset\hua{P}_2$, i.e. $\h_1$ is the tail of $e$.

Suppose $T$ labels $e$, so that $T\in\h_1$ and $T[1]\in\h_2$.
Then, since $\h\leq \h_i \leq \h[1]$, for $i=1,2$,
Lemma~\ref{pp:either-or} implies that $T$ is in $\h$ or $\h[1]$ and
also that $T[1]$ is in $\h$ or $\h[1]$.
Hence it must be that $T\in\h$.

Now let $\torpr{\hua{F}}{\hua{T}}$ \New{be} the torsion pair in $\h_1$ corresponding to $e$.
Noticing that $\hua{T}\subset\h_2$, but $S\notin\h_2$,
we have $S\notin\hua{T}$, which implies there is a nonzero map
$f\colon S \to T$.
Let $M=\Cone(f)[-1]$.
Since $T$ is simple in $\h_1$ and $S\in \h_1$,
we have $M\in\h_1\subset\hua{P}_1\subset\hua{P}$.
On the other hand, since $S$ is simple in $\h$ and $T\in\h$,
we have $M[1]\in\h\subset\hua{P}^\perp[1]$,
so $M\in\hua{P}^\perp$.
Hence $M=0$ and so $S\cong T$, as required.
\end{proof}

We can now describe how forward tilting the base heart transforms
the based exchange graphs.
There is an obvious modification for backwards tilting.

\begin{proposition}\label{pp:EG}
For any heart $\h\in\EGp(\qq{3})$ and $S\in\Sim\h$,
the exchange graph
$\EGp_3(\qq{3},\tilt{\h}{\sharp}{S})$ can be obtained from $\EGp_3(\qq{3},\h)$
by applying a `half-twist', that is,
applying $\phi_{S}^{-1}$ to $\EGp_3(\qq{3}, \h)_S^-$
and reversing all the connecting edges (labelled by $S$) in $\EGp_3(\qq{3},\h)$,
while relabelling them by $S[1]$.
\end{proposition}

\begin{proof}
First observe, by Lemma~\ref{pp:either-or}, that if $S[1]\in\h_0$, then
\[
   \h \leq \h_0 \leq \h[1]
   \iff
   \tilt{\h}{\sharp}{S} \leq \h_0 \leq \tilt{\h}{\sharp}{S}[1].
\]
On the other hand, if $S\in\h_0$ (or equivalently $S[2]\in \phi_S^{-1}(\h_0)$), then
\[
   \h \leq \h_0 \leq \h[1]
   \iff
   \tilt{\h}{\flat}{S} \leq \h_0 \leq \tilt{\h}{\flat}{S}[1]
   \iff
   \tilt{\h}{\sharp}{S} \leq \phi_S^{-1}(\h_0) \leq \tilt{\h}{\sharp}{S}[1],
\]
since $\phi_S^{-1}(\tilt{\h}{\flat}{S})=\tilt{\h}{\sharp}{S}$, by \eqref{eq:root}.
Thus we have
\begin{eqnarray}\label{eq:+-1}
   \EG_3(\qq{3}, \h)_S^+ = \EG_3(\qq{3}, \tilt{\h}{\sharp}{S})_{S[1]}^- \\\label{eq:+-2}
   \phi_S^{-1}\left(\EG_3(\qq{3}, \h)_S^-\right) = \EG_3(\qq{3},\tilt{\h}{\sharp}{S})_{S[1]}^+
\end{eqnarray}
These identifications automatically preserve edges lying in one part of the partition,
while, by Lemma~\ref{lem:connecting}, we know that
any edge connecting $\EG_3(\qq{3},\h)_S^\pm$ is labelled by $S$ and
any edge connecting $\EG_3(\qq{3},\tilt{\h}{\sharp}{S})_{S[1]}^\pm$
is labelled by $S[1]$.
Furthermore,
when $\h_1\in\EG_3(\qq{3}, \h)_S^-$ and $\h_2\in\EG_3(\qq{3}, \h)_S^+$,
we have $\h_2=\tilt{(\h_1)}{\sharp}{S}$
if and only if
$\phi_S^{-1}(\h_1)=\tilt{(\h_2)}{\sharp}{S[1]}$.
Thus we have a `half-twist' isomorphism of \emph{undirected} graphs between
$\EG_3(\qq{3},\h)$ and $\EG_3(\qq{3},\tilt{\h}{\sharp}{S})$,
as described.

It remains to prove that this restricts to a similar isomorphism between
their principal components, i.e.
\begin{align*}
    \EGp_3(\qq{3}, \h)_S^+
 &= \EGp_3(\qq{3}, \tilt{\h}{\sharp}{S})_{S[1]}^-
\\  \phi_S^{-1}\left(\EGp_3(\qq{3}, \h)_S^-\right)
 &= \EGp_3(\qq{3},\tilt{\h}{\sharp}{S})_{S[1]}^+
\end{align*}
For the first equation, suppose that $\h'\in\EGp_3(\qq{3}, \h)_S^+$.
Then the half twist isomorphism on intervals,
as described above, 
turns any path from $\h$ to $\h'$ in $\EG_3(\qq{3},\h)$
into a path from $\phi_S^{-1}(\h)$ to $\h'$ in $\EG_3(\qq{3},\h_S^\sharp)$,
and vice versa.
But there is also an edge
 $\tilt{\h}{\sharp}{S}\xrightarrow{S[1]} \phi_S^{-1}(\h)$,
which implies that $\h'\in\EGp_3(\qq{3}, \tilt{\h}{\sharp}{S})_{S[1]}^-$
if and only if $\h'\in \EGp_3(\qq{3}, \h)_S^+$, as required.
If, on the other hand, $\h'\in\EGp_3(\qq{3}, \h)_S^-$, then
the half twist turns a path from $\h$ to $\h'$ into
into a path from $\phi_S^{-1}(\h)$ to $\phi_S^{-1}(\h')$
and the second equation follows in the same way.
\end{proof}

\begin{corollary}\label{cor:sink-source}
For any $\h\in \EGp(\qq{3})$, the subgraph
$\EGp_3(\qq{3}, \h)$ has a unique source $\h$ and a unique sink $\h[1]$.
\end{corollary}

\begin{proof}
For any $\h'\in\EGp_3(\qq{3}, \h)$ and any simple $S\in\Sim\h'$,
we know from Lemma~\ref{pp:either-or} that
precisely one $S$-tilt of $\h'$ remains in $\EGp_3(\qq{3}, \h)$:
the forward tilt if $S\in\h$ or the backward tilt if $S\in\h[1]$.
If $\h'$ is a source, then it must be the forward tilt for every $S\in\Sim\h'$,
and so $\Sim\h'\subset\h$, and hence $\h'\subset\h$, which implies $\h'=\h$.
Thus $\h$ is the unique source.
The uniqueness of the sink follows similarly: if there is one, then all its simples must be in
$\h[1]$, so it must be $\h[1]$.

It remains to prove that we do actually have $\h[1]\in\EGp_3(\qq{3},\h)$.
We do this by induction, starting from $\h=\zero$,
where we know the result by Remark~\ref{rem:induced}.
Thus we want to show that $\h[1]\in\EGp_3(\qq{3}, \h)$ implies
$\tilt{\h}{\sharp}{S}[1]\in \EGp_3(\qq{3}, \tilt{\h}{\sharp}{S})$,
for every $S\in\Sim\h$, and likewise for $\tilt{\h}{\flat}{S}$,
although the argument is similar so we omit it.

But observe, by \eqref{eq:root}, that
$\tilt{\h}{\sharp}{S}[1]=\phi_S^{-1}(\tilt{\h}{\flat}{S}[1])$
and that $\tilt{\h}{\flat}{S}[1]=\tilt{(\h[1])}{\flat}{S[1]}$,
which we know is in $\EGp_3(\qq{3}, \h)$ by the inductive hypothesis
and Lemma~\ref{pp:either-or}, as $S[1]\in\h[1]$.
Furthermore, since $\tilt{\h}{\flat}{S}[1]$ contains $S$, it is in $\EGp_3(\qq{3}, \h)_S^-$
and so Proposition~\ref{pp:EG} implies that
$\tilt{\h}{\sharp}{S}[1]\in\EGp_3(\qq{3}, \tilt{\h}{\sharp}{S})^+_{S[1]}$, as required.
\end{proof}

\begin{example}\label{ex3}
For $Q$ the $A_3$-type quiver of Example~\ref{ex:flaw},
choose $\h=\zero$ and $S=\hua{I}(Y_1)$.
In Figure~\ref{fig:main}, the two parts of the two based exchange graphs are
separated by the dotted line: the left/right parts
of the top graph are $\EGp_3(\qq{3}, \h)^{-/+}_S$,
while the left/right parts of the bottom graph are
$\EGp_3(\qq{3}, \tilt{\h}{\sharp}{S})^{+/-}_{S[1]}$.
Moreover, the arrows that cross the dotted line
are labelled $S$ in the top graph and $S[1]$ in the bottom graph.
The vertices $\otimes$ and $\odot$ are the
unique source and sink in the graphs.
\end{example}

\begin{figure}[t]\centering
\begin{tikzpicture}[scale=0.8,
  arrow/.style={->,>=stealth,thick},
  equalto/.style={double,double distance=2pt},
  mapto/.style={|->},
  c-before/.style={cyan},
  c-after/.style={red!15!blue!40!green}, 
  c-cross/.style={red},
  c-div/.style={orange},
  c-fixed/.style={black}]
\newcommand{\vsource}{\otimes}
\newcommand{\vsink}{\odot}
\newcommand{\vertx}{\bullet}
\foreach \n/\a\b in {1/1/6, 2/1/4, 3/2/7, 4/2/5, 5/3.25/5, 6/4/7.75,
  7/4/6.5, 8/4/3.5, 9/4/2, 10/4.75/5, 11/6/7, 12/6/5, 13/7/6, 14/7/3.75}
 \coordinate (X\n) at (\a,\b);
\draw[c-before] (X3) node (x3) {$\vsource$};
\draw[c-fixed] (X14) node (x14) {$\vsink$};
\foreach \n in {1,2,4,5,8,9}
  \draw[c-before] (X\n) node (x\n) {$\vertx$};
\foreach \n/\x/\y in {6,7,10,11,12,13}
  \draw[c-fixed] (X\n) node (x\n) {$\vertx$};
\draw[c-cross] (x1) edge[arrow,dashed] (x13);
\foreach \t/\h in {3/6, 5/7, 8/10, 9/14}
  \draw[c-cross] (x\t) edge[arrow] (x\h);
\foreach \t/\h in {1/2, 2/9, 3/1, 3/4, 4/5, 4/2, 5/8, 8/9}
  \draw[c-before] (x\t) edge[arrow] (x\h);
\foreach \t/\h in {6/7, 6/11, 7/10, 10/12, 11/12, 11/13, 12/14, 13/14}
  \draw[c-fixed] (x\t) edge[arrow] (x\h);
\draw[c-fixed] (X6)++(3,-6) node (y6) {$\vsource$};
\draw[c-after] (X9)++(3,-6) node (y9) {$\vsink$};
\foreach \n in {1,2,3,4,5,8}
  \draw[c-after] (X\n)++(3,-6) node (y\n) {$\vertx$};
\foreach \n in {7,10,11,12,13,14}
  \draw[c-fixed] (X\n)++(3,-6) node (y\n) {$\vertx$};
\draw[c-cross] (y13) edge[arrow,dashed] (y1); 
\foreach \t/\h in {3/6, 5/7, 8/10, 9/14}
  \draw[c-cross] (y\h) edge[arrow] (y\t); 
\foreach \t/\h in {1/2, 2/9, 3/1, 3/4, 4/5, 4/2, 5/8, 8/9}
  \draw[c-after] (y\t) edge[arrow] (y\h);
\foreach \t/\h in {6/7, 6/11, 7/10, 10/12, 11/12, 11/13, 12/14, 13/14}
  \draw[c-fixed] (y\t) edge[arrow] (y\h);
\draw[c-div] (2.5,8) edge[dotted,thick] (8.5,-4);
\draw (9.2, 4.8) node (R1) {$\EGp_3(\qq{3},\h)_S^+$};
\draw (11, 2) node (R2) {$\EGp_3(\qq{3},\tilt{\h}{\sharp}{S})_{S[1]}^-$};
\draw (0.4, 1.8) node (L1) {$\EGp_3(\qq{3},\h)_S^-$};
\draw (2, -1) node (L2) {$\EGp_3(\qq{3},\tilt{\h}{\sharp}{S})_{S[1]}^+\quad $};
\draw (R1) edge[equalto] (R2);
\draw (L1) edge[mapto] node[above right] {\small $\phi^{-1}_S$} (L2);
\end{tikzpicture}
\caption{Half twist of $\EGp_3(\qq{3},\h)$ for an $A_3$-type quiver $Q$.}
\label{fig:main}
\end{figure}

Applying Proposition~\ref{pp:EG} inductively, starting from
Theorem~\ref{thm:main} applied to the canonical heart  $\zero$,
we obtain the following result.

\begin{theorem}\label{thm:egx}
For any heart $\h\in\EGp(\qq{3})$,
we have $\cc{\EGp_3}(\qq{3}, \h)\cong\CEG{2}{Q}$,
or equivalently,
$\EGp_3(\qq{3},\h)$
induces an orientation of the (unoriented) cluster exchange graph $\CEGun{Q}$.
\end{theorem}

In particular, Theorem~\ref{thm:egx} says that
$\EGp_3(\qq{3}, \h)$ is a fundamental domain for $\EGp(\qq{3})/\Br_3$,
for every $\h\in\EGp(\qq{3})$,
while in \CY{N} case, this is only proved (and most likely only true)
for the canonical heart $\zero$ as base (cf. Theorem~\ref{thm:main}).

\begin{remark}\label{rem:gluing}
Proposition~\ref{pp:EG} describes precisely
how the subgraphs $\EGp_3(\qq{3},\h)$,
each of which is an orientation of
the cluster exchange graph $\CEGun{Q}$ by Theorem~\ref{thm:egx},
glue together to form $\EGp(\qq{3})$.
\end{remark}

\section{Construction of A2-type exchange graph via the Farey graph}

In this section, we consider a quiver $Q$ of type $A_2$ and, by abuse of notation,
we write $\Gamma_N A_2$ for $\qq{N}$ and allow any $N\geq 2$.
We will demonstrate a roughly dual relationship between the quotient graph
$\EGp(\qq{3})/[1]$ and the Farey graph $\FG$, in its natural embedding
in the hyperbolic disc.
This graph has vertices the rational points on the boundary of the disc
\[
  \FG_0=\QQ\cup\{\infty\},
\]
with an edge from $p/q$ to $r/s$ if and only if $|ps-rq|=1$.
By convention here $\infty=1/0$. These edges, as hyperbolic geodesics,
give a triangulation of the disc.

The Farey graph arises in a variety of contexts;
for example, it is the curve complex of a (once-punctured) torus,
whose vertices are homotopy classes of simple closed curves
and whose edges join curves with intersection number one.
More directly relevant here, $\FG_0$ can be identified with
a conjugacy class of parabolic elements in $\PSL{2}{\ZZ}$
of the form
\begin{equation}\label{eq:defPsi}
  \Psi_{p/q}=\left(\begin{matrix} 1+pq & -p^2 \\ q^2 & 1-pq\end{matrix}\right).
\end{equation}
Note that the natural action of $\PSL{2}{\ZZ}$ on the hyperbolic disc preserves
$\FG$ and furthermore $\Psi_{p/q}$ fixes $p/q$.

\subsection{Spherical twists and vertices in FG}

Denote by $\Sph(\Gamma_N A_2)$ the set of
all spherical objects which are simples in some hearts in $\EGp(\Gamma_N A_2)$ and
\[
   \twists(\Gamma_N A_2)=\{\phi_S\mid S\in\Sph(\Gamma_N A_2)\}\subset\Br_3.
\]
Since $\phi_S=\phi_{S[1]}$, there is a surjective map
\begin{equation}\label{eq:Phi}
    \Phi:\Sph(\Gamma_N A_2)/[1] \to \twists(\Gamma_N A_2).
\end{equation}
Moreover, suppose $\phi_S=\phi_T$ for some $S,T\in\Sph(\Gamma_N A_2)$.
Then $\phi_T(S)=\phi_S(S)=S[1-N]$.
Since there is no non-zero map from $S$ to $S[1-N]$, we must have
\[
    T\otimes\Hom^\bullet(T,S)=S\oplus S[-N],
\]
which implies $T=S[m]$ for some integer $m$.
Thus $\Phi$ in \eqref{eq:Phi} is in fact an bijection.

Let $\Sim\zero=\{X_0,X_\infty\}$ with $\Ext^1(X_0,X_\infty)\neq0$.
Then $\{\phi_{X_0}, \phi_{X_\infty}\}$ is a generating set of $\Br_3$.
By \eqref{eq:ST}, we inductively deduce that
\[
    \twists(\Gamma_N A_2)=\{ \phi\circ\phi_{X_0}\circ\phi^{-1} \mid \phi\in\Br_3\},
\]
that is, $\twists(\Gamma_N A_2)$ is the conjugacy class of one of the generators of the braid group.
It is well-known that $\Br_3$ is a central extension
\[
    0 \to \ZZ \to \Br_3 \xrightarrow{p} \PSL{2}{\ZZ} \to 0,
\]
where the central generator is a braid of non-zero `length' (i.e. number of positive minus number of negative crossings).
Hence the map $p$ is injective restricted to the conjugacy class
$\twists(\Gamma_N A_2)$ and we can identify $\twists(\Gamma_N A_2)$ with its image,
which is $\{\Psi_{a}\mid a\in\FG_0\}$. Indeed we can arrange that
$p(\phi_{X_a})=\Psi_a$ for $a=0,1,\infty$, where $X_1=\phi^{-1}_{X_0}(X_\infty)$,
and thereby obtain a bijection
\[
    \chi=(p\circ \Phi)^{-1}\circ\Psi: \FG_0  \xrightarrow{\sim}  \Sph(\Gamma_N A_2)/[1],
\]
which satisfies
\[
    \chi(\Psi_a(b))=\phi^{-1}_{\chi(a)}(\chi(b)),
\]
for any $a,b\in\FG_0$.
\New{In other words, for each $a\in\FG_0$, the spherical twist $\phi^{-1}_{\chi(a)}$ acts on
$\Sph(\Gamma_N A_2)/[1]$ in the same way that
the parabolic element $\Psi_a$ acts on $\FG_0$, i.e. via \eqref{eq:defPsi}.}
\subsection{L-immersions and triangles in FG}\label{sec:L}

By the action of the $\PSL{2}{\ZZ}$ symmetry,
the properties of the triangle $\Tri=(\infty,1,0)$ in $\FG$
can be extended to any clockwise triangle $\Tri=(a,b,c)$ in $\FG$.
In particular, for each such triangle, we have the following:
\StartEnum
\item There is a triangle
\begin{equation}\label{eq:tri}
    X_a\to X_b \to X_c \to X_a[1]
\end{equation}
in $\D(\Gamma_N A_2)$ such that
$X_{j}$ is in the shift orbit $\chi(j)$ for $j=a,b,c$ satisfying
\[
    X_b={\phi^{-1}_{X_c}(X_a)},\quad X_c={\phi^{-1}_{X_a}(X_b)},\quad
    X_a[1]={\phi^{-1}_{X_b}(X_c)}.
\]
\item
There is an L-immersion $\imm_\Tri\colon\D(A_2)\to\D(\Gamma_N A_2)$,
unique up to the action of $\Aut\D(A_2)$, determined by
\[
    \imm_\Tri(\Ind \D(A_2))=
\shifts{X_a}\cup\shifts{X_b}\cup\shifts{X_c},
\]
where $\shifts{X_j}$ means $\{X_j[m]\}_{m\in\ZZ}$.
\item Up to shift, there are $3(N-1)$ hearts
induced by $\imm_\Tri$, given as follows
\begin{equation}\label{eq:hearts}
    \h^{ca}_{j}=\<X_a,X_c[j-1]\>,\quad
    \h^{ab}_{j}=\<X_b,X_a[j]\>,\quad
    \h^{bc}_{j}=\<X_c,X_b[j]\>.
\end{equation}
where $j=1,\ldots,N-1$.
\item
In \eqref{eq:hearts},
only the three hearts $\h^*_1$ are induced
by $\imm_\Tri$ from canonical hearts in $\D(A_2)$.
Their images in $\EG(\Gamma_N A_2)/[1]$ form a three-cycle $T_\Tri$.
Moreover, the images of hearts $\h^*_{j-1}$ and $\h^*_{j}$
in $\EG(\Gamma_N A_2)/[1]$ form a two-cycle $C_{*,j}$, for $j=2,\ldots,N-1$.
\item
If two triangles share an edge $(c,a)$, then the corresponding induced
hearts $\h^{ca}_{j}$ and $\h^{ac}_{N-j}$ coincide (up to shift).
N.B. the objects $X_a$ and $X_c$ for the two triangles
will also differ by shifts.
\StopEnum

Moreover, by iterated tilting from $\zero$, we can obtain every heart
$\h\in\EGp(\Gamma_N A_2)$ as (a shift of)
one of the hearts in \eqref{eq:hearts}.
In particular, if $\Sim\h=\{A,C\}$, then $(\shifts{A},\shifts{C})$
corresponds an edge in $\FG$.

Thus we can naturally draw the (oriented)
quotient graph
$
  \hua{G}_N=\EGp(\Gamma_N A_2)/[1]
$
on top of the Farey graph, as illustrated in Figure~\ref{fig:G3}
in the case $N=3$ and in Figure~\ref{fig:G24} in the cases $N=2,4$.
It consists of the three-cycles $T_\Tri$,
for each triangle $\Tri$ of $\FG$, joined by a chain of
$N-2$ two-cycles $C_{\Lambda,j}$, for each edge $\Lambda$ of $\FG$,
and thus $\hua{G}_N$ is `roughly' dual to $\FG$.


\begin{figure}\centering
\begin{tikzpicture}[scale=0.8,
 arrow/.style={->,>=stealth,thick,blue}, 
 border/.style={Periwinkle,dotted,thick},
 c-vrtx/.style={blue},
 c-arc/.style={Emerald}]
\newcommand{\vrtx}{\bullet}
\coordinate (O) at (0,0);
\coordinate (S1) at (0,6) ;
\draw [border] (O) circle (6cm);
\draw (-30:6cm) node[below right] {$0$}
    (90:6cm) node[above] {$1$}
    (-90:6cm) node[below] {$-1$}
    (210:6cm) node[below left] {$\infty$};
\draw[c-arc,thick] (S1) \foreach \j in {1,...,3}
    {arc(360/3-\j*360/3+180:360-\j*360/3:10.3923cm)}--cycle;
\draw[c-arc,semithick] (S1) \foreach \j in {1,...,6}
    {arc(360/6-\j*360/6+180:360-\j*360/6:3.4641cm)}--cycle;
\draw[c-arc] (S1) \foreach \j in {1,...,12}
    {arc(360/12-\j*360/12+180:360-\j*360/12:1.6077cm)}--cycle;
\foreach \j in {1,...,3}
{\draw[c-vrtx] (-90+120*\j:0.7cm) node (v\j) {$\vrtx$};
 \draw[c-vrtx] (-210+120*\j:0.7cm) node (w\j) {$\vrtx$};
 \draw[c-vrtx] (-90+120*\j:2.2cm) node (a\j) {$\vrtx$};
 \draw[c-vrtx] (-90+15+120*\j:3cm) node (b\j) {$\vrtx$};
 \draw[c-vrtx] (-90-15+120*\j:3cm) node (c\j) {$\vrtx$};}
\foreach \j in {1,...,3}
{\draw [arrow] (v\j) edge[bend left] (w\j);
 \draw [arrow] (a\j) edge[bend left] (b\j);
 \draw [arrow] (b\j) edge[bend left] (c\j);
 \draw [arrow] (c\j) edge[bend left] (a\j);}
\foreach \j in {1,...,3}
{\draw[c-vrtx] (60*\j*2-1-60:3.9cm) node (x1\j) {$\vrtx$};
 \draw[c-vrtx] (60*\j*2+1-120:3.9cm) node (x2\j) {$\vrtx$};}
\foreach \j in {1,...,3}
{\draw [arrow]  (v\j) edge[bend left] (a\j);
 \draw [arrow]  (a\j) edge[bend left] (v\j);
 \draw [arrow]  (b\j) edge[bend left] (x1\j);
 \draw [arrow]  (c\j) edge[bend left] (x2\j);
 \draw [arrow]  (x1\j) edge[bend left] (b\j);
 \draw [arrow]  (x2\j) edge[bend left] (c\j);}
\end{tikzpicture}
\caption{The Farey graph $\FG$ with `dual' quotient graph $\hua{G}_3$.}
\label{fig:G3}
\end{figure}
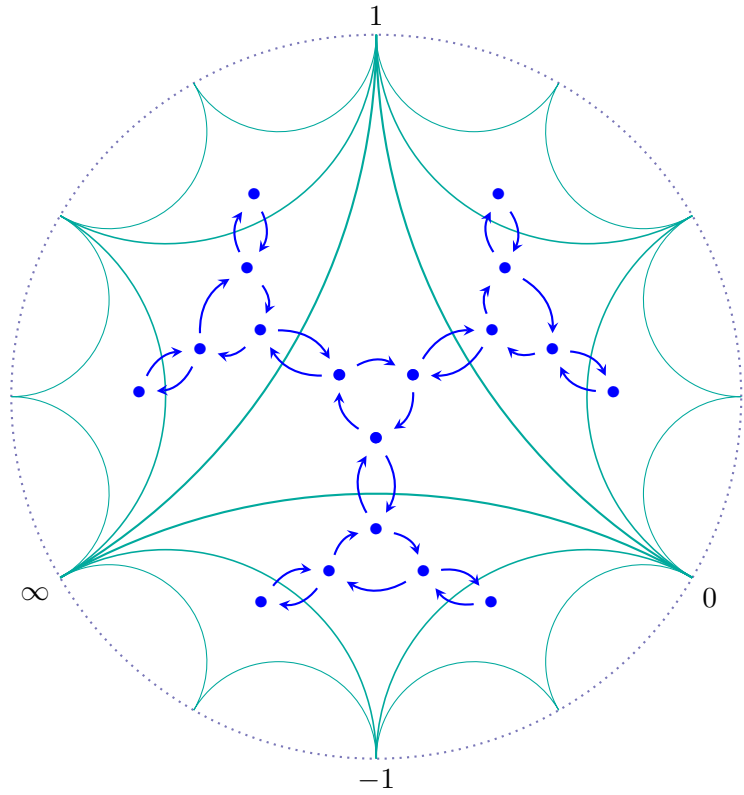

\begin{figure}\centering
\begin{tikzpicture}[scale=0.5,
 arrow/.style={->,>=stealth,thick,blue}, 
 border/.style={Periwinkle,dotted,thick},
 c-vrtx/.style={blue},
 c-arc/.style={Emerald}]
\newcommand{\vrtx}{\bullet}
\coordinate (O) at (0,0);
\coordinate (S1) at (0,6) ;
\draw [border] (O) circle (6cm);
\draw (S1)[c-arc,thick] \foreach \j in {1,...,3}
    {arc(360/3-\j*360/3+180:360-\j*360/3:10.3923cm)}--cycle;
\draw (S1)[c-arc,semithick] \foreach \j in {1,...,6}
    {arc(360/6-\j*360/6+180:360-\j*360/6:3.4641cm)}--cycle;
\foreach \j in {1,...,3}
{\draw[c-vrtx] (-90+120*\j:1.6cm) node {$\vrtx$};
 \draw[c-vrtx] (-210+120*\j:1.6cm) node {$\vrtx$};
 \draw[c-vrtx] (-90+30+120*\j:3.4641cm) node {$\vrtx$};
 \draw[c-vrtx] (-90-30+120*\j:3.4641cm) node {$\vrtx$};}
\foreach \j in {1,...,3}
{\draw[c-vrtx] (-90+120*\j:1.6cm) node (v\j) {};
 \draw[c-vrtx] (-210+120*\j:1.6cm) node (w\j) {};
 \draw[c-vrtx] (-90+30+120*\j:3.4641cm) node (b\j) {};
 \draw[c-vrtx] (-90-30+120*\j:3.4641cm) node (c\j) {};}
\foreach \j in {1,...,3}
{\draw [arrow] (v\j) edge[bend left=15] (w\j);
 \draw [arrow] (v\j) edge[bend left=15] (b\j);
 \draw [arrow] (b\j) edge[bend left=25] (c\j);
 \draw [arrow] (c\j) edge[bend left=15] (v\j);}
\end{tikzpicture}
\;
\begin{tikzpicture}[scale=0.5,
 arrow/.style={-,thick,blue}, 
 border/.style={Periwinkle,dotted,thick},
 c-vrtx/.style={blue},
 c-arc/.style={Emerald}]
\newcommand{\vrtx}{\bullet}
\coordinate (O) at (0,0);
\coordinate (S1) at (0,6) ;
\draw [border] (O) circle (6cm);
\draw[c-arc,thick] (S1) \foreach \j in {1,...,3}
    {arc(360/3-\j*360/3+180:360-\j*360/3:10.3923cm)}--cycle;
\draw[c-arc,semithick] (S1) \foreach \j in {1,...,6}
    {arc(360/6-\j*360/6+180:360-\j*360/6:3.4641cm)}--cycle;
 \foreach \j in {1,...,3}
{\draw[c-vrtx] (-90+120*\j:.6cm) node {$\vrtx$};
\draw[c-vrtx] (-210+120*\j:.6cm) node {$\vrtx$};
 \draw[c-vrtx] (-90+120*\j:1.6cm) node {$\vrtx$};
 \draw[c-vrtx] (-90+120*\j:2.4cm) node {$\vrtx$};
 \draw[c-vrtx] (-90+17+120*\j:3cm) node{$\vrtx$};
 \draw[c-vrtx] (-90-17+120*\j:3cm) node{$\vrtx$};}
  \foreach \j in {1,...,3}
{\draw[c-vrtx] (-90+120*\j:.6cm) node (v\j) {};
 \draw[c-vrtx] (-210+120*\j:.6cm) node (w\j){};
 \draw[c-vrtx] (-90+120*\j:1.6cm) node (x\j){};
 \draw[c-vrtx] (-90+120*\j:2.4cm) node (a\j){};
 \draw[c-vrtx] (-90+17+120*\j:3cm) node (b\j){};
 \draw[c-vrtx] (-90-17+120*\j:3cm) node (c\j){};}
\foreach \j in {1,...,3}
{\draw [arrow] (v\j) edge[bend left] (w\j);
 \draw [arrow] (a\j) edge[bend left=10] (b\j);
 \draw [arrow] (b\j) edge[bend left=20] (c\j);
 \draw [arrow] (c\j) edge[bend left=10] (a\j);}
\foreach \j in {1,...,3}
{\draw [arrow] (v\j) edge[bend left] (x\j);
 \draw [arrow] (x\j) edge[bend left] (v\j);}
\foreach \j in {1,...,3}
{\draw [arrow] (x\j) edge[bend left] (a\j);
 \draw [arrow] (a\j) edge[bend left] (x\j);}
\foreach \j in {1,...,3}
{\draw[c-vrtx] (-90+30+120*\j:3.4641cm) node {$\vrtx$};
 \draw[c-vrtx] (-90-30+120*\j:3.4641cm) node {$\vrtx$};}
\foreach \j in {1,...,3}
{\draw[c-vrtx] (-90+30+120*\j:3.4641cm) node (d\j) {};
 \draw[c-vrtx] (-90-30+120*\j:3.4641cm) node (e\j) {};}
 \foreach \j in {1,...,3}
{\draw [arrow] (b\j) edge[bend left] (d\j);
 \draw [arrow] (c\j) edge[bend left] (e\j);
 \draw [arrow] (d\j) edge[bend left] (b\j);
 \draw [arrow] (e\j) edge[bend left] (c\j);}
\end{tikzpicture}
\caption{The quotient graphs $\hua{G}_2$
 and $\hua{G}_4$ (orientation omitted).}
\label{fig:G24}
\end{figure}
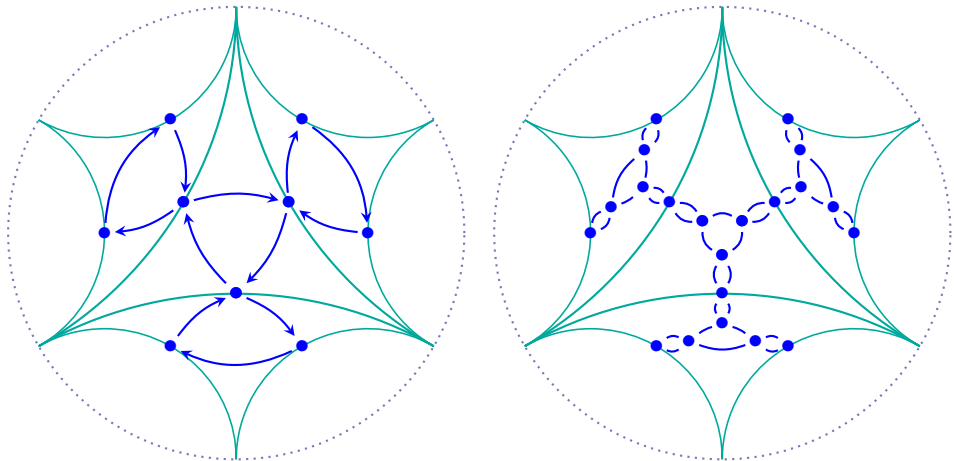

\subsection{Lifting to the exchange graph}

The exchange graph $\EGp(\Gamma_N A_2)$ may be recovered
\New{as the domain of}
a $\ZZ$-cover $\Pi\colon \widetilde{\hua{G}_N}\to\hua{G}_N$,
sitting inside $\hua{G}_N\times\frac{1}{6}\ZZ$
and determined by grading each edge $e$ by
\[
  \gr(e)=\begin{cases}
  \frac{1}{3}, & \text{if } e\in T_\Tri,\\
  \frac{1}{2}, & \text{if } e\in C_{\Lambda,j}.
  \end{cases}
\]
More precisely, each vertex $v$ of $\hua{G}_N$
lifts to $\Pi^{-1}(v)=\{v\}\times \mathbf{N}_v$,
where $\mathbf{N}_v$ is a coset of $\ZZ$ in $\frac{1}{6}\ZZ$.
\New{
Each edge $e\colon v\to w$ of $\hua{G}_N$
lifts to edges $\widetilde{e}\colon(v,n)\to (w,n+\gr(e))$,
for each $n\in \mathbf{N}_v$.
}

\New{
Observe that $\widetilde{\hua{G}_2}$ is tiled by
edge-oriented hexagons,
formed from the lifts of two adjacent $3$-cycles in $\hua{G}_2$ sharing a vertex, $v$ say.
Thus the hexagons have source and sink at vertices $(v, n)$ and $(v, n+1)$, respectively.
Similarly, $\widetilde{\hua{G}_3}$
is tiled by oriented pentagons (cf. Remark~\ref{rem:gluing})
which are the lifts of an adjacent 3-cycle and 2-cycle,
as illustrated on the left of Figure~\ref{fig:pentagon+square}.
For $N\geq 4$, the tiling of $\widetilde{\hua{G}_N}$ has similar pentagons but also includes squares
which are the lifts of two adjacent 2-cycles,
as illustrated on the right of Figure~\ref{fig:pentagon+square}.
The way these oriented pentagons and squares fit together to make larger intervals
may be seen in the solid parts of Figure~\ref{fig:2}.
}

\begin{figure}\centering
\begin{tikzpicture}[scale=0.47,
  egarrow/.style={->,>=stealth,thick,blue}, 
  vline/.style={dashed,gray},
  c-vrtx/.style={blue}]
\newcommand{\vrtx}{\bullet}
\draw[fill=gray!7,dotted]
    (-9, 2.5) to (7+10, 2.5) to (5+10, -3) to (-11,-3) -- cycle;
\draw[c-vrtx] (-4,0) node (A) {$\vrtx$};
\draw[c-vrtx] (-4,10) node (Ax) {$\vrtx$};
\draw[c-vrtx] (-4,4) node (Axx) {$\vrtx$};
\draw[c-vrtx] (1,1) node (C) {$\vrtx$};
\draw[c-vrtx] (1,6) node (Cx) {$\vrtx$};
\draw[c-vrtx] (0,-2) node (B) {$\vrtx$};
\draw[c-vrtx] (0,8) node (Bx) {$\vrtx$};
\draw[c-vrtx] (-6,0) node (D) {$\vrtx$};
\draw[c-vrtx] (-6,7) node (Dx) {$\vrtx$};
\draw [vline] (A) edge (Axx); \draw [vline] (Axx) edge (Ax);
\draw [vline] (C) edge (Cx); \draw [vline] (Cx) edge (1,10);
\draw [vline] (B) edge (Bx); \draw [vline] (Bx) edge (0,10);
\draw [vline] (D) edge (Dx); \draw [vline] (Dx) edge (-6,10);
\draw [egarrow] (A) edge[bend left] (D);
\draw [egarrow] (D) edge[bend left] (A);
\draw [egarrow] (A) edge[bend left] (C);
\draw [egarrow] (C) edge[bend left] (B);
\draw [egarrow] (B) edge[bend left] (A);
\draw [egarrow] (Axx) edge (Dx);
\draw [egarrow] (Axx) edge (Cx);
\draw [egarrow] (Cx) edge (Bx);
\draw [egarrow] (Bx) edge (Ax);
\draw [egarrow] (Dx) edge (Ax);

\draw[c-vrtx] (-3+12,0) node (A) {$\vrtx$};
\draw[c-vrtx] (-3+12,10) node (Ax) {$\vrtx$};
\draw[c-vrtx] (-3+12,4) node (Axx) {$\vrtx$};
\draw[c-vrtx] (-0+12,0) node (B) {$\vrtx$};
\draw[c-vrtx] (-0+12,7) node (Bx) {$\vrtx$};
\draw[c-vrtx] (-6+12,0) node (D) {$\vrtx$};
\draw[c-vrtx] (-6+12,7) node (Dx) {$\vrtx$};
\draw [vline] (A) edge (Axx);  \draw [vline] (Axx) edge (Ax);
\draw [vline] (B) edge (Bx);  \draw [vline] (Bx) edge (0+12,10);
\draw [vline] (D) edge (Dx);  \draw [vline] (Dx) edge (-6+12,10);
\draw [egarrow] (A) edge[bend left] (D);
\draw [egarrow] (D) edge[bend left] (A);
\draw [egarrow] (A) edge[bend left] (B);
\draw [egarrow] (B) edge[bend left] (A);
\draw [egarrow] (Axx) edge (Dx);
\draw [egarrow] (Axx) edge (Bx);
\draw [egarrow] (Bx) edge (Ax);
\draw [egarrow] (Dx) edge (Ax);
\end{tikzpicture}
\caption{Lifting to a pentagon or a square.}\label{fig:pentagon+square}
\end{figure}
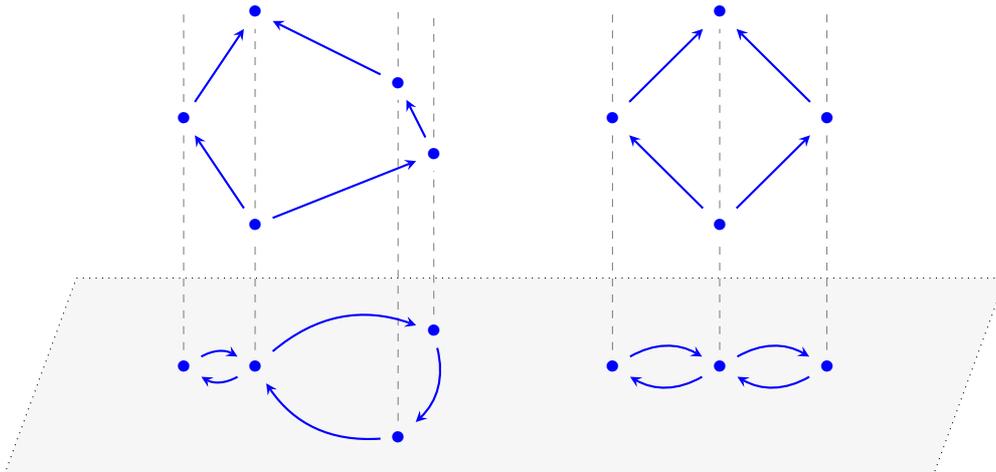


\end{document}